\documentclass[10pt]{amsart}
\usepackage[centering]{geometry}
\numberwithin{equation}{section}

\usepackage{comment}
\usepackage{tikz}

\usepackage{filecontents}
\usepackage{float}

\usepackage[english]{babel}
\usepackage[utf8]{inputenc}
\usepackage[T1]{fontenc}

\usepackage{amsmath}
\usepackage{amssymb}
\usepackage{mathtools}

\usepackage{xcolor}
\usepackage{bbm}

\usepackage{enumerate}
\usepackage{scalerel} 
\usepackage{subfigure}
\usepackage[shortlabels]{enumitem}

\usepackage{graphicx}
\newcommand{\indep}{\rotatebox[origin=c]{90}{$\models$}}

\usepackage{amsthm}
\newtheorem{theorem}{Theorem}
\numberwithin{theorem}{section}
\newtheorem{lemma}[theorem]{Lemma}
\newtheorem{proposition}[theorem]{Proposition}
\newtheorem{corollary}[theorem]{Corollary}
\theoremstyle{definition}
\newtheorem{defi}[theorem]{Definition}
\newtheorem{definition}[theorem]{Definition}
\newtheorem{example}[theorem]{Example}
\newtheorem{remark}[theorem]{Remark}
\newtheorem{asn}[theorem]{Assumption}

\newcommand{\bbX}{\mathbb{X}}
\newcommand{\bbY}{\mathbb{Y}}
\newcommand{\bbZ}{\mathbb{Z}}

\newcommand{\N}{\mathbb{N}}
\newcommand{\Q}{\mathbb{Q}}
\newcommand{\R}{\mathbb{R}}

\newcommand{\X}{\mathcal{X}}

\newcommand{\E}{\mathbb{E}}

\newcommand{\F}{\mathcal{F}}
\newcommand{\G}{\mathcal{G}}
\newcommand{\B}{\mathcal{B}}

\newcommand{\mesh}{\mathrm{mesh}}

\newcommand{\AF}{\textup{AF}}

\newcommand{\supp}{\textup{supp}}

\newcommand{\prob}{\mathcal{P}}
\newcommand{\pr}{\textup{pr}}
\newcommand{\id}{\textup{id}}

\newcommand{\law}{\mathcal{L}}

\newcommand{\cpl}{\textup{Cpl}}

\newcommand{\cplc}{\textup{Cpl}_{\mathrm{c}}}
\newcommand{\cplbc}{\textup{Cpl}_{\mathrm{bc}}}

\newcommand{\W}{\mathcal{W}}
\newcommand{\AW}{\mathcal{AW}}
\newcommand{\CW}{\mathcal{CW}}
\newcommand{\SCW}{\mathcal{SCW}}

\renewcommand{\epsilon}{\varepsilon}
\renewcommand{\subset}{\subseteq}

\renewcommand{\P}{\mathbb{P}}

\newcommand{\pp}{\textup{pp}}

\newcommand{\FP}{\textup{FP}}
\newcommand{\HK}{\textup{HK}}
\newcommand{\cont}{\textup{cont}}

\newcommand{\fp}[1]{{\mathbb #1}}
\newcommand{\cadlag}{c\`adl\`ag}

\newcommand{\cSFP}{\mathcal{NFP}}
\newcommand{\NFP}{\textup{NFP}}

\let\oldmarginpar\marginpar
\renewcommand\marginpar[1]{\-\oldmarginpar[\raggedleft\footnotesize #1]{\raggedright\footnotesize\color{red} #1}}
\usepackage{hyperref}

\title{The Wasserstein space of stochastic processes in continuous time}

\author{D.\ Bartl, M.\ Beiglböck, G.\ Pammer, S.\ Schrott, X.\ Zhang}

\begin{document}
	
	\maketitle

	\begin{abstract}
		Researchers from different areas have independently defined extensions of the usual weak convergence of laws of stochastic processes with the goal of adequately accounting for the flow of information. Natural approaches are convergence of the Aldous--Knight prediction process, Hellwig's information topology, convergence in adapted distribution in the sense of Hoover--Keisler and the weak topology induced by optimal stopping problems. The first main contribution of this article is that on   continuous processes with natural filtrations there exists a canonical \emph{adapted weak topology} which can be defined by all of these approaches; moreover, the adapted weak topology is metrized by a suitable  \emph{adapted Wasserstein distance} $\mathcal{AW}$.

		While the set of  processes with natural filtrations is not complete, we establish that its completion consists precisely of the space ${\rm FP}$ of stochastic processes with general filtrations. 
		We also show that $({\rm FP}, \mathcal{AW})$ exhibits several desirable properties. 
		Specifically, it is Polish, martingales form a closed subset and approximation results such as Donsker's theorem extend to $\mathcal{AW}$. 

\medskip

\noindent \emph{keywords:} causal transport, adapted topologies, continuity of optimal stopping, Donsker's theorem
        
	\end{abstract}

\section{Introduction}
When attempting to represent a natural phenomenon using a stochastic process model, it is likely that the model will not provide a perfect representation of reality.
In order to analyze how such a model error  affects the specific issue being addressed, it is convenient to endow the collection of (laws of) stochastic processes with an appropriate concept of distance or topology.

To that end, denote by $D([0,1];\mathbb{R})$ the space of \emph{\cadlag} paths  equipped with the usual $J_1$-topology and write $\mathcal{P}(D([0,1];\mathbb{R}))$ for the set of all Borel probability measures on $D([0,1];\mathbb{R})$ /  the set of laws of \cadlag{} processes.
Naturally, $\mathcal{P}(D([0,1];\mathbb{R}))$ can be equipped with the standard weak topology, or variants of it like the Wasserstein distance.
However,  the weak topology does not adequately address many of the continuity-related questions that often arise in the study of stochastic processes. 
For instance, classical optimization problems such as optimal stopping or utility maximization do not exhibit continuous dependence on the underlying process, and every process can be approximated by processes that become deterministic immediately after time 0.

As a result, numerous researchers have introduced various modifications of the weak topology tailored to address the specific goals of their research problems. 
That includes Aldous' extended weak convergence \cite{Al81} based on Knight's prediction process \cite{Kn77}, Hoover--Keisler's convergence in adapted distribution \cite{HoKe84}, Hellwig's topology \cite{Ho91}, and the topology induced by  causal / adapted Wasserstein distances \cite{PfPi12, PfPi14, BiTa19, La18, BaBaBeEd19a}.
The first main result of our article presented in Theorem~\ref{thm:intro.all.topo.equal.bounded} 
 shows that the seemingly different approaches can be used to  define the same \emph{adapted weak topology} on the laws of continuous processes.
Moreover, this topology is the coarsest topology for which optimal stopping problems become continuous.

\subsection{Different approaches to adapted topologies}
\label{sec:intro.canonical}

The weak topology on $\mathcal{P}(D([0,1];\mathbb{R}))$ falls short in the context of stochastic processes because it establishes similarity based solely on path resemblance, overlooking the crucial aspect of available information at any given time.
To illustrate this issue briefly, let us examine the following processes.
\begin{figure}[H]
\label{fig:intro}
\centering
\begin{tikzpicture}[scale=1]
			\begin{scope}[xshift=-2.5cm]
				\draw[->, line width=1pt] (-0.3,0) -- (2.3,0) node[below] {$t$};
				\draw[->, line width=1pt] (0,-0.3) -- (0,2.3) node[right] {$\R$};
				\draw[black, line width=1pt] (0,1) -- (1,1) -- (2,1.8) ;
				\draw[black, line width=1pt] (0,1) -- (1,1) -- (2,0.2) ;
			\end{scope}
			
			\begin{scope}[xshift=2.5cm]
				\draw[->, line width=1pt] (-0.3,0) -- (2.3,0) node[below] {$t$};
				\draw[->, line width=1pt] (0,-0.3) -- (0,2.3) node[right] {$\R$};
				\draw[black, line width=1pt] (0,1) -- (1,1.05) -- (2,1.8) ;
				\draw[black, line width=1pt] (0,1) -- (1,0.95) -- (2,0.2) ;
			\end{scope}
		\end{tikzpicture}

\vspace{-1.5em}
\caption{$\mathbb{P}$ on the left and $\mathbb{P}^\varepsilon$ on the right.}
\end{figure}
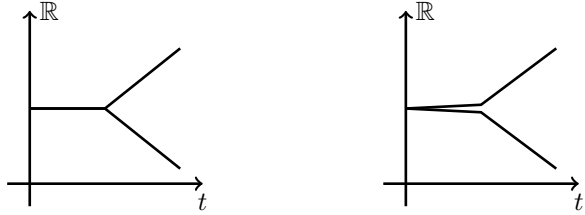
It is evident that $\mathbb{P}$ and $\mathbb{P}^\varepsilon$ in Figure \ref{fig:intro} are similar when viewed through the lens of the weak topology. 
However, when considering them from the perspective of stochastic processes they arguably differ substantially, because their `information' available on the future behaviour  varies considerably across these processes.
As an illustration, the values of certain optimal stopping problems evaluated under $\mathbb{P}^\varepsilon$ will not converge to the value under $\mathbb{P}$.

\medskip

\noindent
{\bf Aldous' topology \cite{Al81}.} 
One approach to formalizing the notion of `information' lies in examining the conditional distribution / conditional law of a given process.
Denote by $X$ the canonical process on $\mathcal{X}:=D([0,1];\mathbb{R})$ defined by $X_t(x):=x(t)$ and let  
$\mathcal{F}_t^{\mathbb{P}}:=\bigcap_{\varepsilon>0}\sigma(X_s: s \leq t+ \varepsilon)\vee\mathcal{N}(\mathbb{P})$ be the {natural} filtration augmented by the ${\mathbb{P}}$-null sets {$\mathcal{N}(\mathbb{P})$}.
Then, if $\mathbb{P}$ and $\mathbb{P}^\varepsilon$ are the laws in Figure \ref{fig:intro} and $t\in(0,\frac{1}{2})$, {we have}
\[ \mathcal{L}_{\mathbb{P}}(X_1 | \mathcal{F}_t^{\mathbb{P}})
=\frac{1}{2}\left(\delta_{\{0\}} + \delta_{\{2\}} \right)
\quad\text{but}\quad 
\mathcal{L}_{\mathbb{P}^\varepsilon}(X_1 | \mathcal{F}_t^{\mathbb{P}})
=\begin{cases}
\delta_{\{2\}} &\text{if } X_t\geq 1,\\
\delta_{\{0\}} &\text{if } X_t<1.
\end{cases}
 \]
This suggests that a natural approach to strengthen the weak topology within the framework of laws of stochastic process could involve exploring appropriate topological structures for the conditional laws of these processes.
Aldous' extended weak topology is obtained using that idea.

To be more precise, if $\mathcal{P}(\mathcal{X})$ is endowed with the usual weak topology, then it is standard to verify that the $\mathcal{P}(\mathcal{X})$-valued stochastic process $t\mapsto \mathcal{L}_{\mathbb{P}}(X | \mathcal{F}_t^{\mathbb{P}})$ has almost surely \cadlag{} paths.
In particular, the so-called (first order) \emph{prediction process}
\[ {\rm pp}(\mathbb{P})
=\big({\rm pp}_t(\mathbb{P}) \big)_{t\in[0,1]}
:= \big(\mathcal{L}_{\mathbb{P}}(X | \mathcal{F}_t^{\mathbb{P}})\big)_{t\in[0,1]} \]
is a \cadlag{} stochastic process taking its values in  $\mathcal{P}(\mathcal{X})$.
Aldous endows $D([0,1];\R \times \mathcal{P}(\mathcal{X}))$ with the $J_1$ distance $d_{J_1}$ and considers the coarsest topology for which the mapping
\begin{align}
\label{eq:intro.aldous}
\mathcal{P}(\mathcal{X}) \to \mathcal{P}\big( D([0,1];\R \times\mathcal{P}(\mathcal{X}) )\big),
\qquad \mathbb{P}\mapsto \mathcal{L}_{\mathbb{P}}\big( (X_t,{\rm pp}_t(\mathbb{P}))_{t \in [0,1]}\big)
\end{align}
is continuous.

 While Aldous used Skorohod's $J_1$-topology on $D([0,1]; \R \times \mathcal{P}(\X))$, one can also choose weaker mode of convergence on this space. A natural choice for this is the Meyer--Zheng topology. These two approaches yield the same topology on processes with continuous filtration, while the version with the $J_1$-topology is stronger in general, see Example~\ref{ex:counter}. To distinguish these topologies we will refer to them as Aldous\textsubscript{MZ} and Aldous\textsubscript{J1}.

\medskip 
\noindent
{\bf Hoover--Keisler's topology \cite{HoKe84, Ho91}} shares  certain resemblance with the topology introduced by Aldous, but there are notable distinctions between the two.
Hoover--Keisler endow the \cadlag{} paths with the 
Meyer--Zheng topology. 
Moreover, they consider \emph{iterated} prediction processes, namely ${\rm pp}^2(\mathbb{P})$ is the prediction process of the \cadlag{} process $({\rm pp}_t(\mathbb{P}))_{t\in[0,1]}$ and so forth.
Their topology is the initial topology  of $\mathbb{P}\mapsto \law_\mathbb{P}( ({\rm pp}^r(\mathbb{P}))_{r\in\mathbb{N}})$ (see Section \ref{sec:Recap} for the details). 

\medskip

\noindent
{\bf Hellwig's topology \cite{He96}} originates in a finite discrete-time context.
For any fixed time point $t$, Hellwig partitions the time index into the past and the future and considers disintegrations of the future w.r.t.\ the past. 
His topology is given as the initial topology of the corresponding mappings.
In Section \ref{sec:adapted topologies on plain} we detail the continuous-time extension of Hellwig's topology.

\vspace{0.5em}
\noindent
{\bf The optimal stopping topology.}
We recall that one of the motivations behind considering refined variations of the weak topology on $\mathcal{P}(\mathcal{X})$ stems from the fact that fundamental stochastic optimization problems do not exhibit continuity w.r.t.\ the weak topology.
Arguably, one fundamental  optimization problem in that context is that of optimal stopping. 
Let $\varphi\colon\mathcal{X}\times[0,1]\to\mathbb{R}$ be a  bounded, continuous and non-anticipative function (that is,  $\varphi(x,t)$ depends solely on $(x_s)_{s\in[0,t]}$), and set 
%
\[ {\rm OS}(\mathbb{P},\varphi)
:= \inf\big\{ \mathbb{E}_\mathbb{P}[ \varphi(X,\tau) ] : \tau \text{ is an $(\mathcal{F}_t^\mathbb{P})_{t\in[0,1]}$-stopping time} \big\}.\]
The optimal stopping topology is  defined as the coarsest topology for which  $\mathbb{P} {\mapsto} {\rm OS}(\mathbb{P},\varphi)$ is continuous for every function $\varphi$ as above.

\vspace{0.5em}
\noindent
{\bf The adapted Wasserstein distance.}
In the context of the weak topology on $\mathcal{P}(\mathcal{X})$, the influential role of the optimal transport perspective using Wasserstein distances has left a profound mark on various fields from PDE, calculus of variation and probability to machine learning. 
Interestingly, although the classical Wasserstein distance is not well-suited for analyzing differences between stochastic processes (because it metrizes the weak topology), there exists a suitable variant thereof designed  precisely for this purpose.
Fix a metric $d_\mathcal{X}$ on $\mathcal{X}$ that is complete and compatible with the $J_1$-topology (see Remark~\ref{rem:defAWflexibility1}), denote by $\mathcal{P}_p(\mathcal{X})$ the set of all $\mathbb{P}\in\mathcal{P}(\mathcal{X})$ for which $\mathbb{E}_\mathbb{P}[d^p_{\mathcal{X}}(0,X)]<\infty$, and define the \emph{strict adapted Wasserstein distance}
\[ 
\mathcal{AW}_p^{\rm (s)} (\mathbb{P},\mathbb{Q}) = \inf_{\pi\in \cpl_{\rm bc}(\mathbb{P},\mathbb{Q})}  \mathbb{E}_\pi[d^p_{\mathcal{X}}(X,Y)]^{1/p} .
\]
The difference to the usual Wasserstein distance is that one restricts to so-called \emph{bi-causal} couplings, i.e.\ couplings that are non-anticipative and can be viewed as a Kantorovich analogue of non-anticipative transport maps.
It was shown in a discrete-time framework that $\mathcal{AW}_p^{(s)}$ constitutes the analogue of the classical Wasserstein distance in the realm of stochastic processes: it induces the same topology as the (discrete-time variants of) Aldous' topology, Hoover--Keisler's topology, the optimal stopping topology and has many further desirable properties, see \cite{BaBaBeEd19b, BaBePa21}.

Perhaps surprisingly, in the present continuous-time setting, the previous observations  no longer apply.
Indeed, the distance $\mathcal{AW}_p^{(s)}$ seems too strong to analyse \emph{general} stochastic processes.
For instance, it is not separable and  Donsker's theorem does not hold true with $\mathcal{AW}_p^{(s)}$, as the \emph{only} bi-causal coupling between a random walk and the Brownian motion is the product measure, see Appendix \ref{sec:aw.strong}.

With this in mind, when our objective is to  modify $\mathcal{AW}_p^{(s)}$ tailored for comparing general stochastic processes, it becomes necessary to \emph{relax the bi-causality constraint}.
In fact, a natural approach towards that goal draws a parallel to the reasoning behind equipping the \cadlag{} paths $D([0,1];\mathbb{R})$ with the $J_1$ distance rather than the sup-norm, and it consists in setting
\[ \mathcal{AW}_p (\mathbb{P},\mathbb{Q})
:= \inf_{\varepsilon>0 } \left( \varepsilon + \inf_{\pi\in \cpl^{\epsilon}_{\text{\rm bc}}(\mathbb{P},\mathbb{Q})} \mathbb{E}_\pi[d^p_{\mathcal{X}}(X,Y)]^{1/p} \right).\]

Roughly put, the idea is that whether a certain information is arriving $\varepsilon$ earlier or later in time should not have a substantial impact but be penalized only on the level of $\varepsilon$, see Definition~\ref{def:AWcont}.

\medskip
\noindent
{\bf Causal distance \cite{La18, AcBaZa20}.} 
Another distance based on optimal transport principles stems from Lassalle's causal  transport problem.
Set
\[ \mathcal{CW}_p (\mathbb{P},\mathbb{Q})
:= \inf_{\varepsilon>0 } \inf_{\pi\in \cpl^\varepsilon_{{\rm c}}(\mathbb{P},\mathbb{Q})}\Big( \varepsilon+ \mathbb{E}_\pi[d^p_{\mathcal{X}}(X,Y)]^{1/p} \Big)\]
where couplings are causal if they are non-anticipative from $\mathbb{P}$ to $\mathbb{Q}$ (but not necessarily vice versa) and $\varepsilon$-causality represents the same relaxation as previously employed.
While  $\mathcal{CW}_p$ cannot be a distance (as it lacks symmetry), a natural remedy is to consider the \emph{symmetrized  causal Wasserstein distance} $\mathcal{SCW}_p (\mathbb{P},\mathbb{Q})=\max\{\mathcal{CW}_p (\mathbb{P},\mathbb{Q}),\mathcal{CW}_p (\mathbb{Q},\mathbb{P})\}$.

\subsection{The adapted weak topology on $\mathcal{P}(\mathcal{X})$.}
Recall that $X$ is the canonical process on $\mathcal{X}=D([0,1];\mathbb{R})$.
We say that the $p$-th moments of $\mathbb{P}^n$ converge to $\mathbb{P}$ if $\E_{\mathbb{P}^n}[d_\mathcal{X}^p(0,X)]\to \E_{\mathbb{P}}[d_\mathcal{X}^p(0,X)]$.
With this notation set in place, we are ready to state our first main result.

\begin{theorem}
\label{thm:intro.all.topo.equal.bounded}
	Let $p\in[1,\infty)$ and $\mathbb{P},\mathbb{P}^n \in\mathcal{P}_p(\mathcal{X})$ for $n\in\mathbb{N}$.
	Then the following are equivalent.
	\begin{enumerate}[label = (\roman*)]
    \item $\mathbb{P}^n\to\mathbb{P}$ in the Aldous$_{\text{MZ}}$ topology  and the $p$-th moments converge.
	\item $\mathbb{P}^n\to\mathbb{P}$ in the Hoover--Keisler topology and the $p$-th moments converge.
	\item $\mathbb{P}^n\to\mathbb{P}$ in Hellwig's topology and the $p$-th moments converge. 
	\item $\mathcal{AW}_p(\mathbb{P}^n,\mathbb{P})\to 0$.
     \item $\mathcal{SCW}_p(\mathbb{P}^n,\mathbb{P})\to 0$.
     \item $\mathcal{CW}_p(\mathbb{P},\mathbb{P}^n)\to 0$.
	\end{enumerate}
	If, in addition, $X$ has $\mathbb{P}$-a.s.\ continuous paths,
 (i)-(v) are further equivalent to 
	\begin{enumerate}[(i)]
	\item[(vii)] $\mathbb{P}^n\to\mathbb{P}$ in the optimal stopping topology and the $p$-th moments converge.
	\end{enumerate}
\end{theorem}

Theorem \ref{thm:intro.all.topo.equal.bounded} is reassuring as it demonstrates that the previously discussed approaches give rise to the same concept.
Moreover, we will see in Section~\ref{sec:adapted topologies on plain} that convergence in (i)-(vi) is equivalent to the convergence of all probabilistic operations built from iterating $d^p$-bounded continuous functions and conditional expectations. We refer to this topology as the  adapted weak topology. 

In addition, Theorem \ref{thm:intro.all.topo.equal.bounded} turns out to be very useful in applications of the adapted weak topology.
For example, it is often more straightforward to show that a sequence converges in the adapted weak topology when employing Hellwig's definition.
Conversely, when the goal is to establish a continuity property for a specific problem defined on $\mathbb{P}$, a stronger formulation of the topology proves advantageous.
For instance, the characterization via $\AW_p$ is useful since causal couplings enable the transfer of features between processes, specifically the proof that optimal stopping problems continuously depend on the adapted weak topology is based on transferring stopping times from one process to another via causal couplings.

Notably, the continuity assumptions for the equivalence of  (vii) and the other items stated in Theorem \ref{thm:intro.all.topo.equal.bounded} are found to be crucial. 
For instance, when $X$ does not have $\mathbb{P}$-a.s.\ continuous paths, then (i)-(vi) do not imply (vii). Also, Theorem \ref{thm:intro.all.topo.equal.bounded} fails if we replace adapted Wasserstein distance by its strict counterpart, not even (iv) and (v) are equivalent in this case, see Remark~\ref{rem:sSCW}. Note also that the order of $\P, \P^n$ in (vi) is important: $\mathcal{CW}_p(\mathbb{P}^n,\mathbb{P})\to 0$ is precisely equivalent to convergence in the ordinary $p$-Wasserstein distance. 

	If the metric  $d_{\mathcal{X}}$ is bounded 
    in the context of Theorem \ref{thm:intro.all.topo.equal.bounded} the requirement that the $p$-th moments converge is satisfied automatically.
	In particular, in this case the topology induced by $\mathcal{AW}_p$ does not depend on $p$ and is equal to the  Hoover--Keisler topology.

\begin{remark}
    While in the discrete-time setting there is  only \emph{one} adapted topology on the set of adapted processes \cite{BaBaBeEd19b, Pa22, BoLiOb23}, the situation is more nuanced in continuous time.  
    Theorem \ref{thm:intro.all.topo.equal.bounded} asserts that when interpreting existing  approaches to adapted topologies in the weak way, we arrive at the same concept. One the other hand, there are also stricter ways to interpret these approaches, which lead to different topologies.
 Examples are the Aldous$_{J_1}$ topology which is stronger than the Aldous$_{\text{MZ}}$ topology as well as the strict version of the adapted Wasserstein distance.
 Specifically the
  strict adapted Wasserstein distance is useful for instance in connection with Talagrand type inequalities \cite{La18, Fo22a} and to provide stability and sensitivity results in mathematical finance, see e.g.\ \cite{BaBaBeEd19a, JiOb24}. 
    Our view  is that the adapted weak topology is a natural `weakest' topology which accounts for the role of information and that depending on the intended purposes one may want to consider different strengthenings. 
\end{remark}

\vspace{0.5em}

\medskip

Our next result shows that \emph{Donsker's theorem} extends to the realm of the adapted weak topology. 
Moreover, this also holds true for Euler approximations of Stochastic Differential Equations (SDEs). 
This result seems interesting when combined with the previously mentioned continuous dependence observed for different  optimization problems.
Indeed, it implies that they can be approximated effectively through discrete-time stochastic processes. We give here an informal version and  refer to Section \ref{sec:donsker} for the full details.

\begin{theorem}
Let $p\in[1,\infty)$.
\begin{enumerate}[label = (\roman*)]
    \item Let $\mathbb{P}$ be the standard Wiener measure and let $\mathbb{P}^n$ be the law of the standard scaled random walk with step size $\frac{1}{n}$. 
    Then $\AW_p(\mathbb{P}^n,\mathbb{P})\to 0$.
    \item Let $\mathbb{P}$ be the law of the solution to the SDE $dY_t=\mu_t(Y_t) dt + \sigma_t(Y_t) dB_t$ where $B$ is the standard Brownian motion and $\mu$ and $\sigma$ satisfy mild regularity conditions, and let  $\mathbb{P}^n$ be the law of the Euler approximation of $Y$ with step size $\frac{1}{n}$.
    Then $\AW_p(\mathbb{P}^n,\mathbb{P})\to 0$.
\end{enumerate}
\end{theorem}

\subsection{General stochastic processes}

The findings in the previous section provide a fairly comprehensive depiction of the adapted weak topology on $\mathcal{P}(\mathcal{X})$, i.e.\ for stochastic processes with natural filtrations.
Nevertheless, this framework presents two notable limitations.
Firstly, the metric space $(\mathcal{P}_p(\mathcal{X}),\AW_p)$ is not complete, and this persists also when using other natural distances instead of $\AW_p$.
Secondly, following the usual theory of stochastic processes, one would like to consider processes together with a \emph{general filtration} and not just the natural filtration. 

In fact, these limitations present their mutual resolution: an interpretation of the incompleteness of $(\mathcal{P}_p(\mathcal{X}),\AW_p)$ is that the space $\mathcal{P}_p(\mathcal{X})$ is too `small' to adequately represent all processes one would like to consider.
We show that the extra information that can be stored in an ambient filtration is precisely what is needed to arrive at the completion of $(\mathcal{P}_p(\mathcal{X}),\AW_p)$.  Put differently, there exist sequences of stochastic processes with natural filtration that converge to processes whose filtration stores strictly more information. 
To make this precise we need the following definition.

\begin{definition}
	A \emph{filtered process} $\fp X$ is a five tuple
	\[ \fp X = \big( \Omega,\mathcal{F},\mathbb{P}, (\mathcal{F}_t)_{t\in[0,1]},X \big) \]
	where $( \Omega,\mathcal{F},\mathbb{P})$ is a probability space, $(\mathcal{F}_t)_{t\in[0,1]}$ is a filtration satisfying the usual conditions, and $ X=(X_t)_{t\in[0,1]}$ is an adapted \cadlag{} process.
	We write $\mathcal{FP}$ for the collection of all filtered processes and $\mathcal{FP}_p$ for those satisfying $\mathbb{E}_{\mathbb{P}}[d_\mathcal{X}^p(0,X)]<\infty$.
\end{definition}

Clearly $\mathcal{P}(\mathcal{X})$ can be viewed as a subset of $\mathcal{FP}$: it corresponds to all  processes $\fp X$ with \emph{natural} filtration, i.e.\ for which $\mathcal{F}_t=\bigcap_{s>t} \sigma(X_u: u\leq s)\vee \mathcal{N}(\mathbb{P})$.
Moreover, it is straightforward to extend the concept of bi-causality and the definition of $\mathcal{AW}_p$ from the previous setting of $\mathcal{P}(\mathcal{X})$ to the present setting of $\mathcal{FP}$, see Section \ref{sec:AW}.

It is worthwhile to emphasize that the space $\mathcal{FP}$ contains a multitude of processes one would  perceive as being `equivalent'.
For example, if $\fp X$ is a filtered process and $\fp Y$ is derived from $\fp X$ by adding independent information into the filtration that does not influence the behavior of $X$, one would typically not treat $\fp X$ and $\fp Y$ as distinct processes.
In analogy to the definition of the $\mathcal{L}_p /  L_p$ spaces in analysis, we consider the quotient space
\[ \FP_p := \mathcal{FP}_p/_{\sim_{\AW}} \quad \text{where $\fp X\sim_{\AW}\fp Y$ iff $\AW_p(\fp X,\fp Y)=0$ }.
\]
The following is our second main result.

\begin{theorem}
	$\AW_p$ is a metric on $\FP_p$ and $(\FP_p,\AW_p)$ is the completion of $(\mathcal{P}_p(\mathcal{X}),\AW_p)$.
\end{theorem}

It turns out that the equivalence relation $\sim_{\AW}$ captures essential properties of a stochastic process.
For instance, if two processes are equivalent, then the values of the respective optimal stopping problems are equal.
Indeed, let  $\varphi\colon\mathcal{X}\times[0,1]\to\mathbb{R}$ be a bounded, non-anticipative function, and for  $\fp X=(\Omega,\mathcal{F},\P,(\mathcal{F}_t)_{t\in[0,1]},X)\in\mathcal{FP}$ set
\[ 
{\rm OS}(\fp X,\varphi) := \inf\big\{ \mathbb{E}_\mathbb{P}[ \varphi(X,\tau) ] : \tau\text{ is  an } (\mathcal{F}_t)_{t\in[0,1]}\text{-stopping time} \big\}.
\]

\begin{theorem}
\label{thm:intro.os}
    Let $\varphi$ be as above.
    \begin{enumerate}[(i)]
	\item Let $\fp X, \fp Y \in \mathcal{FP}_p$.
    If $\fp X\sim_{\AW}\fp Y$, then ${\rm OS}(\fp X, \varphi)={\rm OS}(\fp Y,\varphi)$.
	\item Let $\fp X,\fp X^n\in\mathcal{FP}_p$ such that $\AW_p(\fp X^n,\fp X)\to 0$.
    If $X$ has almost surely continuous paths and $\varphi$ is continuous, then ${\rm OS}(\fp X^n, \varphi) \to {\rm OS}(\fp X, \varphi)$.
	\end{enumerate}
\end{theorem}

To put it differently, Theorem \ref{thm:intro.os} (i) states that $\sim_{\AW}$ selects all information of processes relevant to optimal stopping problems; in particular, ${\rm OS}(\cdot,\varphi)$ may be defined on $\FP_p$.
The following section explores this phenomenon in a broader context, and we show that $\sim_{\AW}$ selects `all' relevant information of a given process.

\vspace{1em}
\noindent
{\bf The Hoover--Keisler equivalence and  topology.}
Based on formalizing what assertions belong to the `language of probability' Hoover--Keisler \cite{HoKe84, Ho91} have made precise what it should mean that two processes $\fp X,\fp Y\in\mathcal{FP}$  have the same probabilistic properties; we write $\fp X \sim_{\rm HK} \fp Y$.
They show that $\sim_{\rm HK}$ can be characterized via the \emph{iterated} variant of the prediction process considered in the previous section.
For $\fp X =(\Omega,\mathcal{F},\P,(\mathcal{F}_t)_{t\in[0,1]},X)\in\mathcal{FP}$, denote its prediction process by
\[ {\rm pp}(\fp X) := \left( \mathcal{L}_{\mathbb{P}}(X | \mathcal{F}_t ) \right)_{t\in[0,1]}. \]
Since ${\rm pp}(\fp X)$ is a \cadlag{} process taking its values in $\mathcal{P}(\mathcal{X})$, one may define \emph{its} prediction process
\[ {\rm pp}^2(\fp X) := \left( \mathcal{L}_{\mathbb{P}}( {\rm pp}(\fp X)  | \mathcal{F}_t ) \right)_{t\in[0,1]}, \]
which again is a \cadlag{} process,  this time taking its values in $\mathcal{P}(D([0,1];\mathcal{P}(\mathcal{X})))$.
This procedure is  iterated  infinitely often  resulting in the \cadlag{} process ${\rm pp}^\infty(\fp X) := ({\rm pp}^r(\fp X))_{r\in\mathbb{N}}$, see Section \ref{sec:Recap} for more details, and 
\[\fp X \sim_{\rm HK} \fp Y \quad\text{iff} \quad \mathcal{L}_{\mathbb{P}}\left( {\rm pp}^\infty(\fp X)\right) 
= \mathcal{L}_{\mathbb{P}}\left( {\rm pp}^\infty(\fp Y)\right).\]
It turns out that $\sim_{\rm HK}$ and $\sim_{\AW}$ are equal:

\begin{theorem}
	Let $\fp X,\fp Y\in\mathcal{FP}_p$.
	Then $\fp X\sim_{\rm HK}\fp Y$ if and only if  $\fp X\sim_{\AW}\fp Y$. 
\end{theorem}

\begin{remark}
It is natural to ask whether it is necessary to iterate the prediction process (infinitely many times). Indeed, it follows from \cite[Section 7]{BaBePa21} that for each $r \in \N$ there  exist filtered processes $\fp X, \fp Y$ with $\law(\pp^r(\fp X)) =\law(\pp^r(\fp Y))$, which lead to different optimal stopping values. Thus,  $\fp X, \fp Y$ are distinct in terms of $\sim_{\HK}$ but equivalent from the perspective of $\pp^r$.  Hence, the weak topology induced by $\pp^r$ (as well as the Aldous\textsubscript{J1}  topology) are not even Hausdorff on $\FP$. In particular, they cannot coincide with the topology generated by $\AW$. 
\end{remark}

In analogy to the Hoover--Keisler topology on $\mathcal{P}(\mathcal{X})$ introduced in Section \ref{sec:intro.canonical}, the characterisation of $\sim_{\rm HK}$ via ${\rm pp}^\infty$ naturally induces a topology on $\FP$, the \emph{Hoover--Keisler topology} (on $\FP$), defined as the initial topology of  $\fp X \mapsto \mathcal{L}_{\mathbb{P}}(\rm pp^\infty (\fp X))$.

\begin{theorem}
	For $\fp X,\fp X^n\in\FP_p$, $n\in\mathbb{N}$, the following are equivalent.
 \begin{enumerate}[label = (\roman*)]
	\item $\fp X^n\to \fp X$ in the Hoover--Keisler topology and $\E_{\mathbb{P}^{\fp X^n}}[d_{\mathcal{X}}^p(X^n,0)] \to\E_{\mathbb{P}^{\fp X}}[d_{\mathcal{X}}^p(X,0)]$.
    \item
    $\AW_p(\fp X^n,\fp X)\to 0$.
    \item
    $\SCW_p(\fp X^n,\fp X)\to 0$.    
 \end{enumerate}
\end{theorem}

An especially valuable feature of the usual weak topology on $\mathcal{P}(\mathcal{X})$ is the ubiquity  of compact sets ensured by Prohorov's theorem.
Remarkably, the latter extends to the present context.

\begin{theorem} \label{prop:comp.intro}
    For a set $\Pi\subset \FP_p$, the following are equivalent.
\begin{enumerate}[label = (\roman*)]
\item  
 $\Pi$ is precompact w.r.t.\ $\AW_p$.
 \item 
 $\{ \law(\fp X) : \fp X\in \Pi\}$ is precompact w.r.t.\ $\W_p$.
 \end{enumerate}
\end{theorem}

In point (ii), $\W_p$ denotes the usual Wasserstein distance, in particular, (ii) is equivalent to the following: for every $\varepsilon>0$ there is a compact set $K\subset \mathcal{X}$ and $\sup_{\fp X\in \Pi} \mathbb{E}_{\mathbb{P}^{\fp X}}[ (1+d_\mathcal{X}^p(X,0))1_{K^c}(X)]\leq\varepsilon$.
Let us note that Proposition \ref{prop:comp.intro} was first established by Hoover \cite{Ho91} (in a somewhat different framework), see also \cite{BePaScZh23}.

\vspace{0.5em}
\noindent
{\bf Canonical representatives.} 
It is possible to construct for each filtered process $\fp X$ a canonical filtered process $\overline{\fp X}$ satisfying $\AW(\fp X, \overline{\fp X})=0$. These canonical filtered process have a specific standard Borel probability space as their basis, see Section~\ref{sec:Recap} for details. In particular, every $\sim_\AW$-equivalence class contains a representative that is based on a standard Borel probability space.

\subsection{Related literature}

As already noted, topologies (or distances) that account for the flow of information were introduced by different groups of authors in different discrete or continuous-time settings. This includes Aldous in stochastic analysis \cite{Al81}, 
Hoover--Keisler in mathematical logic \cite{HoKe84, Ho91}, Hellwig in economics \cite{He96}, Pflug--Pichler \cite{PfPi12, PfPi14} in optimization, Bion-Nadal--Talay and Lasalle \cite{BiTa19, La18} in stochastic control / stochastic analysis, Nielsen--Sun \cite{NiSu20} in machine learning, the topology induced by the Knothe--Rosenblatt coupling \cite{BePaPo21} and, using higher rank signatures, Bonnier--Liu--Oberhauser \cite{BoLiOb23}. 

In  \cite{BaBaBeEd19b, Pa22} it is established that the all of these topologies coincide in finite discrete time, while \cite{BaBePa21} shows that the completion of all  filtered discrete-time processes can be represented by allowing for general filtrations. Hence \cite{BaBaBeEd19b, Pa22, BaBePa21} are discrete-time precursors of the present work. We emphasize that the proofs in the discrete-time setting are based on induction on the number of time steps and differ fundamentally from the arguments given here. 

Adapted topologies and adapted transport theory have been applied to geometric inequalities \cite{BaBeLiZa17, La18, Fo22a}, stochastic optimization  \cite{Pf09, PfPi12, KiPfPi20, BaWi23}, starting with the article \cite{Do14} (pricing of game options), they have proved useful at various instances in  mathematical finance, see e.g.\ \cite{AcBaZa20} (questions of insider trading and enlargement of filtrations), \cite{GlPfPi17, BaDoDo20, BaBaBeEd19a, BeJoMaPa21a, BeJoMaPa21b, Wi20, JoPa24} (stability of pricing / hedging and utility maximization) and \cite{AcBePa20} (interest rate uncertainty) and appear in machine learning, see e.g.\ \cite{AkGaKi23} (causal graph learning), \cite{XuLiMuAc20, AcKrPa23} (video prediction and generation),  \cite{XuAc21} (universal approximation). We refer to \cite{PfPi16, BaBaBeWi20, AcHo22, BlWiZhZh24} for results on statistical estimation and to \cite{PiWe21, EcPa22, BaHa23} for efficient numerical methods for adapted transport problems and adapted Wasserstein distances.
Recent articles which consider sensitivity w.r.t.\ adapted Wasserstein distance and make the connection with distributionally robust optimization are \cite{BaWi23, Ji24, JiOb24, SaTo24, MiWi24}. The connection between adapted and classical weak topologies / Wasserstein distances is explored in \cite{Ed19, BlLaPaWi24}. Duality and barycenters are analyzed in \cite{KrPa24, AcKrPa24}.
Explicit calculations of the adapted Wasserstein distance between Gaussian processes are given in \cite{GuWo24, AcHoPa24}. 

As hinted above, in continuous time the behaviour of strict  (see Section \ref{sec:aw.strong})  causal / adapted Wasserstein distance is sensitive to the underlying distance; e.g.\ the Cameron--Martin norm is used in \cite{La18}; pathwise quadratic variation is used in \cite{Fo22a} to obtain a $T_2$ inequality, in \cite{BaBeHuKa20} for stability results, in \cite{BaBeHuKa20, BePaRi24} to characterize (geometric) Bass martingales (see also \cite{BaLoOb24}) and in \cite{JiOb24} for sensitivity of DRO; $L_p$-distances on paths are considered in \cite{BiTa19, BaKaRo22, RoSz24}. 
The recent article \cite{CoLi24} puts (strictly) causal / adapted couplings in continuous time center stage, establishing in particular denseness results for Monge-type couplings between SDEs.

\subsection{Organization of the article}
Section~\ref{sec:notbg} summarizes the necessary background on continuous-time filtered process from \cite{BePaScZh23}. In Section~\ref{sec:AW} we develop the central concepts of this article: We introduce the notion of $\epsilon$-bicausal couplings, use it to define the adapted Wasserstein distance $\AW$ and show that $\AW$ induces the Hoover--Keisler topology. We further show that the symmetrized causal Wasserstein distance $\SCW$ generates the same topology as $\AW$.  Section~\ref{sec:OS} contains quantitative and qualitative continuity results for optimal stopping problems. Moreover, we show that martingales form an $\AW$-closed subset of $\FP$. Section~\ref{sec:plain} concerns processes with their natural filtration, in particular we show that the space $\FP$ is the completion of the set of process with natural filtration w.r.t.\ $\AW$. Section~\ref{sec:adapted topologies on plain} is devoted to the proof of Theorem~\ref{thm:intro.all.topo.equal.bounded} on the equivalence of topologies on processes with natural filtration.  In Section~\ref{sec:donsker} we prove qualitative and quantitative approximation results for the Brownian motion and SDEs.

\section{Notation and background on filtered processes}\label{sec:notbg}
\subsection{Notations and conventions}\label{sec:not}
For a filtered process $\fp X$, we usually refer to the elements of the tuple $\fp X$ as $\Omega^{\fp X}, \F^{\fp X}, \P^{\fp X}, (\F_t^{\fp X})_{t \in [0,1]}, X$, that is,  $\Omega^{\fp X}$ denotes the base space of the filtered process $\fp X$ etc. 
For notational simplicity, for a stochastic process $(Y_t)_{t\in[0,1]}$ and  $t>1$, we set $Y_t:=Y_1$ and employ similar conventions for filtrations etc.

Given a set $\Omega$,  a $\sigma$-algebra $\F$ on $\Omega$, and a probability measure $\P$ on $(\Omega,\mathcal{F})$, we write $\F^\P$ for the completion of  $\F$ with $\P$-nullsets. 
Equalities between random variables on the same probability space $(\Omega,\F,\P)$ are always to be understood $\P$-a.s. 
On a product space $\Omega^{\fp X} \times \Omega^{\fp Y}$ we will always be explicit, as a.s.-equality depends on the choice of the coupling $\pi$. 
If $\mathcal{F}^{\fp X}$ is a $\sigma$-algebra on $\Omega^{\fp X}$, we can consider it as a $\sigma$-algebra on the product space $\Omega^{\fp X} \times \Omega^{\fp Y}$, i.e.\ we identify $\mathcal{F}^{\fp X}$ with $\mathcal{F}^{\fp X}\otimes \{\emptyset, \Omega^\fp Y\}$ to shorten the notation. 
For $\sigma$-algebras $\F, \G, \mathcal{H}$ we write $\mathcal{F} \indep_{\mathcal{H}} \mathcal{G}$ if $\F$ and $\G$ are independent given  $\mathcal{H}$.  
For a topological space $A$, we denote by $\mathcal{P}(A)$ the set of all Borel probability measures on $A$.
If $A$ is a metric space, $\mathcal{P}_p(A)\subset \mathcal{P}(A)$ is the subset of all those probability measures with finite $p$-th moment.
On $\mathcal{P}(A)$ we consider the topology of weak convergence of probability measures (induced by testing with integrals w.r.t.\ continuous bounded functions) and on $\mathcal{P}_p(A)$  the topology induced by the $p$-Wasserstein distance $\W_p$.
If $B$ is another metric space and $f\colon A\to B$ a Borel map, for $\mu\in\mathcal{P}(A)$ we denote by $f_\# \mu\in\mathcal{P}(B)$ the push-forward of $\mu$ under $f$.

We write $\mathcal X$ for the path space, i.e.\ $\mathcal X = D([0,1];\R^d)$ or $\mathcal X = C([0,1];\R^d)$ if we restrict to continuous paths. The set of continuous functions from $A$ to $B$ is denoted by $C(A;B)$, the subset of bounded functions by $C_b(A;B)$.
If $B=\R$ or $B=\R^d$, we often simply write $C(A)$ and  $C_b(A)$ if clear from context.
For a function $f \in \mathcal X$ (i.e.\ a path)  and $t \in [0,1]$, we define the stopped path $f^t$ by $f^t(s)=f(t \wedge s)$ for all $s \in [0,1]$.

Finally, finite and countable product spaces are always endowed with the product topology.

\subsection{An introduction to filtered processes}
\label{sec:Recap}
In this section we summarize the necessary background from continuous-time filtered processes.
Most of the material is taken from \cite{BePaScZh23}.

First recall from the introduction that a \emph{filtered process} $\fp X$ is a five tuple
\[ \fp X = \big( \Omega,\mathcal{F},\mathbb{P}, (\mathcal{F}_t)_{t\in[0,1]}, X \big), \]
where $( \Omega,\mathcal{F},\mathbb{P})$ is a probability space, $(\mathcal{F}_t)_{t\in[0,1]}$ is a filtration satisfying the usual conditions, and $X=(X_t)_{t\in[0,1]}$ is an adapted \cadlag{} process.  We write $\mathcal{FP}$  for the class of all filtered processes.

The prediction process of $\fp X \in \mathcal{FP}$, denoted by $\pp(\fp X)$, is the up to indistinguishability unique \cadlag{} version of the measure-valued-martingale
\[
 (\law_\P(X | \F_t^{\fp X}))_{t \in [0,1]}.
\] 
As $\pp(\fp X)$ can be seen as a random variable on the space $\mathsf{M}_1 := D([0,1]; \prob(\mathsf{M}_0))$, where $\mathsf{M}_0 := \mathcal{X}$. It is possible to consider \emph{its} prediction process and further iterate this procedure, cf.\ \cite{HoKe84,Ho91}:  

\begin{itemize}
    \item The prediction process of rank 1 is $\pp^1(\fp X) = \pp(\fp X) \in \mathsf{M}_1$.
    \item The prediction process of rank $r+1$ is the (unique up to indistinguishability) \cadlag{}-version of 
    \[
(\law_{\P^{\fp X}}(\pp^r(\fp X) |\F_t^{\fp X} ))_{t \in [0,1]} \in D([0,1]; \prob(\mathsf{M}_r) ) =: \mathsf{M}_{r+1}. 
    \]
    \item The prediction process of rank $\infty$ is defined as 
    \[
    \pp^\infty(\fp X) := (t \mapsto (\pp^1_t(\fp X), \pp^2_t(\fp X),\dots)) \in D\bigg([0,1];\prod_{r=0}^\infty \prob(\mathsf{M}_r) \bigg) =: \mathsf{M}_\infty.
    \] 
\end{itemize}

For $r \ge 1 $, we use the Meyer--Zheng topology \cite{MeZh84}, i.e.\ the spaces $\mathsf{M}_r$ are endowed with topologies according to the  convention\footnote{In the case $\mathsf{M}_0=\X= D([0,1];\R^d)$ this space is still equipped with Skorohod's $J_1$-topology as outlined in the introduction.} 
that the set of \cadlag{} paths $D(S) = D([0,1];S)$ with values in a topological space $S$ is equipped with the topology of convergence in measure w.r.t.\ the measure $\lambda + \delta_1$, where $\lambda$ is the Lebesgue measure on $[0,1]$.

It is important to consider the Meyer--Zheng topology on the space of \cadlag{} processes, as it ensures relative compactness in law for a sequence of martingales $(X^n)_n$ if and only  if $\sup_n \E[|X^n_1|] < \infty$, see \cite{MeZh84}. 
This leads to a Prokhorov-type compactness result for filtered processes (see Proposition~\ref{prop:Compact}), which is crucial for the analysis in this paper.

\begin{remark}
The chosen topology on  $D([0,1]; S)$  has the drawback of not being Polish.
However, it remains Lusin, i.e., homeomorphic to a Borel subset of a Polish space \cite[Appendix]{MeZh84}. 
Lusin spaces are stable under operations such as $S \mapsto \prob(S)$, $ S \mapsto D(S)$ and countable products; in particular $\mathsf{M}_r $ is Lusin for all  $r \in \mathbb{N} \cup \{\infty\}$.
Moreover, they provide sufficient regularity for our purposes, behaving similarly as Polish spaces.
For instance, conditional laws of  $S$-valued random variables are well-defined, measure-valued martingales in $\prob(S)$ have \cadlag{} versions, and all expressions in the definition of iterated prediction processes are well-posed. For further details, see \cite[Section~2]{BePaScZh23}.
\end{remark}

Note that the prediction processes of finite rank are measure-valued martingales, that is for every Borel and bounded $f : \mathsf{M}_r \to \R$, the real-valued process $ Y_t(\omega) := \int f d\pp^r_t(\fp X)(\omega)$ is a real-valued $(\F^{\fp X}_t)_{t \in [0,1]}$-martingale. As the random variable $\pp^r(\fp X)$ is $ \F_1^{\fp X}$-measurable, we have for every $r \in \N$,
$$
\pp^{r+1}_1(\fp X)= \law(\pp^r(\fp X) |\F_1^{\fp X}) = \delta_{\pp^r(\fp X)}. 
$$
In this sense, $\pp^{r+1}(\fp X)$ contains more information about $\fp X$ than $\pp^{r}(\fp X)$ does. In a specific sense, which was made precise by Hoover--Keisler \cite{HoKe84}, $\pp^\infty(\fp X)$ contains all relevant probabilistic information about $\fp X$. Following their approach, we identify two filtered processes with each other if the prediction processes of rank $\infty$ have the same law:
\begin{definition}\label{def:HKequiv}
We say that $\fp X , \fp Y \in \mathcal{FP}$ are Hoover--Keisler equivalent, denoted as $\fp X \sim_{\HK} \fp Y$ if $\law(\pp^\infty(\fp X)) = \law(\pp^\infty(\fp Y))$.  
The space of filtered process is defined as $\FP = \mathcal{FP}/_{\sim_{\HK}}$. 
\end{definition}

Note that in the introduction we defined the space of filtered processes as the factor space $\FP = \mathcal{FP}/_{\sim_{\AW}}$. We will show in Section~\ref{sec:AW} that these definitions agree. 
For now, we stick to $\mathcal{FP}/_{\sim_{\HK}}$ to be consistent with the papers \cite{BePaScZh23,Ho91,HoKe84}. 

\begin{definition}\label{def:HKtop}
The Hoover--Keisler topology on $\FP$, denote by $\HK$, is the initial topology w.r.t.\ 
\[
\FP \to \prob(\mathsf{M}_\infty) : \quad  \fp X \mapsto \law_{\P^{\fp X}}(\pp^\infty(\fp X)).
\]
\end{definition}
It is immediate from this definition that $(\FP,\HK)$ is homeomorphic to a subspace of the metrizable space $\mathsf{M}_\infty$ and hence metrizable. Therefore, it is sufficient to work with sequences when proving properties of the topology $\HK$. 
Moreover, it is shown in \cite[Theorem~1.5]{BePaScZh23} that $(\FP,\HK)$ is a Polish space, but in this paper no explicit metric for this topology is constructed. 
In Section~\ref{sec:AW} of the present paper, we will show that the adapted Wasserstein distance is a complete metric for $(\FP,\HK)$.

We proceed and state the announced Prohorov-type theorem for $(\FP,\HK)$, see \cite[Theorem~1.6]{BePaScZh23}. 
 
\begin{proposition}\label{prop:Compact}
A set of filtered processes $A \subset \FP$ is relatively compact in the Hoover--Keisler topology if and only if the respective set of laws $\{ \law(X) : \fp X \in A \}$ is tight. 
\end{proposition}

\begin{defi}\label{def:FPpHKp}
We denote with $\mathcal{FP}_p$ the collection of filtered processes $\fp X \in \mathcal{FP}$ with finite $p$-th moment, that is $\E_{\P^{\fp X}}[d^p(X,0)] < \infty$ and set $\FP_p = \mathcal{FP}_p /_{\sim_{\HK}}$. 

Let $\fp X^n ,\fp X \in \FP_p$ for $n\in\N$. We say that $\fp X^n \to \fp X$ in $\HK_p$ if $\fp X^n \to \fp X $ in $\HK$ and $\E_{\P^{\fp X^n}}[d^p(X^n,0)] \to \E_{\P^{\fp X}}[d^p(X,0)]$.
\end{defi}

Since relative compactness in the space of probability measure with finite $p$-th moment, equipped with the $p$-Wasserstein distance $(\prob_p(\mathcal X), \W_p)$ is equivalent to tightness and uniform integrability of the $p$-th moments (see e.g.\ \cite[Proposition~7.1.5]{AGS}), the following consequence of Proposition \ref{prop:Compact} is immediate.

\begin{corollary}\label{cor:CompFPp}
A set of filtered processes $A \subset \FP_p$ is relatively compact w.r.t.\ $\HK_p$ if and only if the respective set of laws $\{ \law(X) : \fp X \in A \}$ is relatively compact in $(\prob_p(\mathcal X),\W_p)$.
\end{corollary}

Next we want to assign to each filtered process a canonical filtered process. To that end, we need some notation: $Z = (Z_t)_{t \in [0,1]}$ denotes the canonical process on $\mathsf{M}_\infty$, so for every $t \in [0,1]$, $Z_t$ is a (infinite) vector of measures $(Z^r_t)_{r \ge 1}$ where $Z^r_t \in \prob(\mathsf{M}_{r-1})$.

Moreover, we set $\overline{X} := Z^0 := \delta^{-1}(Z_1^1)$, where $\delta^{-1} : \prob(\mathcal X) \to \mathcal X$ is a Borel function that satisfies $\delta^{-1}(\delta_x)=x$ for all $x \in \mathcal X$. Note that $\delta^{-1}$ restricted to the set of Dirac measures is continuous.  

\begin{defi}\label{def:asociatedCFP}
Let $\fp X \in \mathcal{FP}$. The \emph{canonical filtered process} associated to $\fp X$ is given by
\[
\overline{\fp X} = (\mathsf{M}_\infty,\mathcal{B}_{\mathsf{M}_\infty}, \law(\pp^\infty(\fp X)), (\F_t)_{t \in [0,1]}, \overline X),
\]
where $\mathcal{B}_{\mathsf{M}_\infty}$ is the Borel $\sigma$-algebra on $\mathsf{M}_\infty$, $(\F_t)_{t \in [0,1]}$ is the right-continuous completion of the canonical filtration on $\mathsf{M}_\infty$, and $\overline{X}$ is the map defined above. 

We denote the collection of all canonical filtered process as  $\mathcal{CFP}$. 
\end{defi}
The canonical filtered process associated to $\fp X$ has the property that its prediction process of rank $\infty$ satisfies $\law(\pp^\infty(\overline{\fp X})) = \law(\pp^\infty(\fp X))$. Thus, every $\HK$-class contains exactly one element from $\mathcal{CFP}$. This fact was established in \cite[Theorem 1.3]{BePaScZh23} and implies the following
\begin{lemma}\label{lem:canonRep}
    Let $\fp X, \fp Y \in \mathcal{FP}$.
    Then $\fp X \sim_{\HK} \overline{\fp X}$.
    Moreover, $\fp X \sim_{\HK} \fp Y$ if and only if $\overline{\fp X} = \overline{\fp Y}$.  
\end{lemma}

\begin{remark}\label{rem:canonCont}
Note that for every $\fp X \in \mathcal{FP}$ the function $\overline{X} : \mathsf{M}_\infty \to \mathcal X$ is continuous on the support of $\law(\pp^\infty(\fp X))$.
This will be used in the sequel.
\end{remark}

We end this section with useful observations that relate continuity of generated filtrations with continuity properties of the prediction process.

\begin{defi}\label{def:cont_pt}
Let $\fp X \in \mathcal{FP}$. 
We call $t \in [0,1]$ a continuity point of $\fp X$ if $s \mapsto \law(\pp^\infty_s(\fp X))$ is continuous at $t$.
\end{defi}

The next result, which was shown in \cite[Corollary 4.16]{BePaScZh23}, is a consequence of the prediction processes being measure-valued martingales with \cadlag{} paths.

\begin{lemma}\label{lem:cont_pt_dense}
    Let $\fp X \in \mathcal{FP}$. Then $[0,1] \setminus \cont(\fp X)$ is countable. Hence, $\cont(\fp X)$ is dense in $[0,1]$.
\end{lemma}

Next, we recall the characterization of continuity points given in \cite[Proposition 4.19]{BePaScZh23}, that will be useful throughout this article.

\begin{proposition}\label{prop:contPt}
    Let $\fp X \in \mathcal{FP}$ and $\G_s := \sigma^{\P^{\fp X}}(\pp^\infty_r(\fp X) : r \le s)$. 
    The filtration $(\G_s)_{s\in[0,1]}$ is right-continuous and for $t \in [0,1]$ the following are equivalent:
	\begin{enumerate}[label = (\roman*)]
            \item $t \in \cont(\fp X)$, that is, $s \mapsto \law(\pp^\infty_s(\fp X))$ is continuous at $t$.
		\item The filtration $(\G_s)_{s\in[0,1]}$ is continuous at $t$, that is $\G_t=\G_{t-}$.
		\item The paths of $\pp^\infty(\fp X)$ are a.s.\ continuous at $t$.
	\end{enumerate}
\end{proposition}

Finally, we introduce the concept of adapted functions. An adapted function is an operation that takes a filtered process $\fp X$ and returns a random variable on the underlying probability space of this process. More precisely, $f$ is an adapted function (we write $f \in \AF$) if it can be built using the following operations:
\begin{enumerate}
		\item[(AF1)] Every $f \in C_b(\X)$ is an adapted function, and we set $f[\fp X] := f(X)$. 
		\item[(AF2)] If $f_1, \dots, f_n \in \AF$ and $g \in C_b(\R^n)$, then $g(f_1,\dots,f_n) \in \AF$ and we set $g(f_1,\dots,f_n)[{\fp X}] := g(f_1[{\fp X}], \dots, f_n[{\fp X}])$.
		\item[(AF3)] If $f \in \AF$ and $t \in [0,1]$, then $(f|t) \in \AF$ and $(f|t)[{\fp X}] := \E[f[{\fp X}]|\F_t^{\fp X}]$.
	\end{enumerate}
It is possible to characterize Hoover--Keisler-equivalence via adapted functions.
\begin{proposition}\label{prop:simHKviaAF}
	For $\fp X, \fp Y \in \mathcal{FP}$ we have $\fp X \sim_\HK \fp Y$ if and only if $\E [f[\fp X]] = \E[f[ \fp Y ]]$ for every $f \in \AF$.
\end{proposition}
 
For $f \in \AF$, we write $T(f) \subset [0,1]$ for the collection of those $t \in [0,1]$ such that $\E[ \, \cdot \, | \F_t]$ appears in $f$. Specifically, we set $T(f) = \emptyset$ for $f \in C_b(\X)$, set  $T(g(f_1,\dots,f_n)) =\bigcup_{i=1}^n T(f_i)$, and further set $T((f|t)) = T(f) \cup \{t\}$.

Using this notions, we can state the following definition. 
\begin{definition}
We say that $\fp X^n \to \fp X$ in the \emph{adapted weak topology} if $\E[f[\fp X^n]] \to \E[f[ \fp X]]$ for every $f \in \AF$ satisfying $T(f) \subset \cont(\fp X) \cup \{ 1 \}$. 
\end{definition}

\begin{theorem}\label{thm:HKAF}
The Hoover--Keisler topology coincides with the adapted weak topology. 
\end{theorem}

\section{The adapted Wasserstein distance}\label{sec:AW}
The aim of this section is to give a precise definition of the adapted Wasserstein distance $\AW_p$ on the space of filtered processes $\FP_p$. The main result of this section is Theorem~\ref{thm:AWmetrizesHK}, which states that $\AW_p$ metrizes the Hoover--Keisler topology plus convergence of $p$-th moments.

\subsection{$\epsilon$-bicausal couplings and the definition of $\AW$} \label{ssec:def AW}

	\begin{defi}\label{def:cpl2}
		Let $\fp X, \fp Y \in \mathcal{FP}$ and let $\varepsilon\geq 0$. A coupling between $\fp X$ and $\fp Y$ is a probability measure on $(\Omega^{\fp X} \times \Omega^{\fp Y}, \F^{\fp X} \otimes \F^{\fp Y})$  with marginals $\P^{\fp X}$ and $\P^{\fp Y}$. We say that a coupling $\pi$ between $\fp X$ and $\fp Y$ is  
		\begin{enumerate}[(a)]
			\item 
			\emph{$\epsilon$-causal} from $\fp X$ to $\fp Y$ if $
			\F_t^{\fp Y} \indep_{\F_{t+\epsilon}^{\fp X}} \F_1^{\fp X}
			$ under $\pi$ for all $t \in [0,1]$;
			\item 
			\emph{$\epsilon$-bicausal} if $\pi$ is both $\epsilon$-causal from $\fp X$ to $\fp Y$ and $\epsilon$-causal from $\fp Y$ to $\fp X$.
		\end{enumerate}
        We write $\cpl(\fp X, \fp Y)$, $\cpl^{\epsilon}_{\rm{c}}(\fp X,\fp Y)$  and $\cpl^{\epsilon}_{\rm{bc}}(\fp X,\fp Y)$ for the set of couplings, $\epsilon$-causal couplings and $\epsilon$-bicausal couplings between $\fp X$ and $\fp Y$, respectively. 
		For $\epsilon=0$ we shall drop the prefix $\epsilon$ to the causality constraints, e.g.\ causal means $0$-causal.
	\end{defi}

In what follows, we set $\F^{\bbX,\bbY}_{t,s}:=\F^{\fp X}_t\otimes \F^{\fp Y}_s$ for $t,s\in[0,1]$. 
The following characterization of causality  will be used frequently in what follows.
	\begin{lemma}\label{lem:causal_equiv}
		Let $\fp X,\fp Y\in\mathcal{FP}$,  $\pi \in \cpl(\bbX,\bbY)$,  $t\in[0,1]$, and  $\epsilon \ge 0$.
		 The following are equivalent:
		\begin{enumerate}[(i)]
			\item $
			\F_t^{\fp Y} \indep_{\F_{t+\epsilon}^{\fp X}} \F_1^{\fp X}
			$ under $\pi$.
			\item $\E_\pi[V|\F^{\bbX}_{1}] = E_\pi[V|\F^{\bbX}_{t+\epsilon}]$ for all bounded $\F_t^\bbY$-measurable $V$.
			\item $\E_\pi[U|\F^{\bbX,\bbY}_{t+\epsilon,t}] = \E_\pi[U|\F^{\bbX}_{t+\epsilon}]$ for all  bounded $\F_1^\bbX$-measurable $U$.
            \item $\E_\pi[ V (U - \E_\pi[U |\F^{\fp X}_{t+\epsilon} ] )   ] =0 $ for all bounded $\F_1^\bbX$-measurable $U$, bounded $\F_t^\bbY$-measurable $V$.
		\end{enumerate}
	\end{lemma}
\begin{proof}
The equivalence of (i), (ii) and (iii) is a consequence of \cite[Proposition~5.8]{Ka97}. We show that (ii) is equivalent to (iv). Clearly (ii) is equivalent to 
\begin{align}\label{eq:prf:causalch}
\E_\pi [ U  \E_\pi[ V|\F^{\fp X}_{1}  ]] =    \E_\pi [ U   \E_\pi[ V|\F^{\fp X}_{t+\epsilon}  ]   ] 
\end{align}
for all bounded $\F_1^\bbX$-measurable $U$ and bounded $\F_t^\bbY$-measurable $V$. Since the conditional expectation is a self-adjoint operator\footnote{This means that  $\E[F\E[G|\mathcal{H}]]=\E[\E[F|\mathcal{H}]G]$ for all bounded random variables $F,G$ and $\sigma$-fields $\mathcal{H}$.} and  $\E_\pi[U|\F^{\fp X}_{1}]=U$, \eqref{eq:prf:causalch} is  equivalent to 
\[
\E_\pi [ U  V] =    \E_\pi [ \E_\pi[U|\F^{\fp X}_{t+\epsilon}  ]  V  ], 
\]
which is another way of writing (iv).
\end{proof}
\begin{definition}\label{def:AWcont}
On $\mathcal{FP}_p$ we define the \emph{adapted Wasserstein distance} by 
	\begin{align}
	    \label{eq:def.AW}
        \mathcal{AW}_p(\fp X,\fp Y) 
		:=  \inf_{\epsilon \ge 0}    \inf \left\{ \E_\pi[d_\mathcal{X}^p(X,Y)]^{1/p} + \epsilon :  \pi \in \cpl^{\varepsilon}_{\rm bc}(\fp X,\fp Y) \right\}.
	\end{align}
  If $\mathcal X = C([0,1];\R^d)$, we set $d_{\mathcal{X}}$ to be the supremum norm; if $\mathcal X = D([0,1];\R^d)$ we fix
  a complete metric $d_{\mathcal{X}}$ that induces the $J_1$-topology and satisfies $d_\X(f,0) \le C\|f\|_\infty$.
\end{definition}

\begin{remark}\label{rem:defAWflexibility1}
Note that while the usual $J_1$-distance is not complete, there are complete distances that induce the $J_1$-topology  and satisfy $d_\X(f,0) \le \|f\|_\infty$, see e.g.\ \cite[Section~12]{Bi99}. 
Moreover, the definition of the adapted Wasserstein distance and the results in this chapter are largely independent of the specific path space  $\mathcal{X}$  and its metric, except for separability and completeness. 
There is also flexibility in penalizing couplings that are $\epsilon$-bicausal rather than strictly bicausal. For simplicity, we penalize with $\epsilon$, but alternative penalties may be more natural for specific classes of processes (e.g., $\sqrt{\epsilon}$ for continuous martingales).
\end{remark}

An observation that turns out to be useful is that for canonical filtered processes the infimum in the definition of $\mathcal{AW}_p$ is attained.
	This follows, similar as in the case of the usual Wasserstein distance, from a suitable compactness of $\varepsilon$-bicausal couplings:
	
	\begin{proposition}\label{prop:epsbcclosed}
		For $\bbX,\bbY\in \mathcal{CFP}$ and $\epsilon \ge 0$, the sets  $\cpl^{\epsilon}_{\rm{c}}(\bbX,\bbY)$ and $\cpl^{\epsilon}_{\rm{bc}}(\bbX,\bbY)$ are compact in $\mathcal P(\mathsf{M}_\infty \times \mathsf{M}_\infty)$.
	\end{proposition}
    \begin{proof} 
    We first claim that $\cpl^\epsilon_{\rm c}(\bbX,\bbY)$ is compact.
    To that end, write 
        \[
		\cpl^\epsilon_{\rm c}(\bbX,\bbY) = \bigcap_{t \in [0,1]} C_t, \quad \text{where }  C_t:= \left\{ \pi \in \cpl(\bbX,\bbY) :  \F_t^\bbY \indep_{\F_{t+\epsilon}^\bbX}{\F_1^\bbX} \text { under }  \pi \right\}.
	\]
    The claim follows if $C_t$ is compact for every fixed $t\in[0,1]$.
    Fix such $t$.

    By Lemma~\ref{lem:causal_equiv}, a coupling $\pi \in \cpl(\bbX,\bbY) $ is an element of $C_t$ if and only if 
    \begin{align}
        \label{eq:lem.compact.char.causal}
    \E_\pi\left[ \E_\pi[V|\F^{\fp Y}_{t}] \left(U - \E_\pi[U |\F^{\fp X}_{t+\epsilon} ] \right)   \right] =0 \text{ for all } U\in L^\infty(\F_1^\bbX) \text{ and } V\in L^\infty(\F_1^\bbY).
    \end{align}   
    We reformulate the latter condition in a manner that makes compactness more readily apparent.

    To that end recall that $\Omega^{\fp X}=\Omega^{\fp Y} = \mathsf{M}_\infty$ and the latter is a Lusin space. 
    Denote by $\P^{\fp Y}_t\colon  \Omega^{\fp Y} \to \prob(\Omega^{\fp Y})$ a regular disintegration of $\P^{\fp Y}$ w.r.t.\ $\F_t^{\fp Y}$, and similarly for $\P^{\fp X}_{t+\varepsilon}$ and $\P^{\fp X}_1$.
    These disintegrations exist since $\mathsf{M}_\infty$ is Lusin, see \cite[Corollary~10.4.13]{Bo07a}.
 Moreover, for measurable bounded $f: \Omega^{\fp X} \to \R$ and $g: \Omega^{\fp Y} \to \R$, define $\Phi_{f,g}\colon\Omega^{\fp X}\times\Omega^{\fp Y}\to\mathbb{R}$ by
    \[\Phi_{f,g}(\omega_x,\omega_y)
    :=\int g(\eta_y) \, \P^{\fp Y}_{t}(\omega_y,d\eta_y) \left(   \int f(\eta_x) \,\P^{\fp X}_{1}(\omega_x, d\eta_x) -   \int f(\eta_x) \, \P^{\fp X}_{t+\varepsilon}(\omega_x,d\eta_x)   \right).\]
    Then, by definition, \eqref{eq:lem.compact.char.causal} is equivalent to $\int \Phi_{f,g} \,d \pi=0$ for all $f$ and $g$ as above.

    In the next step, we change the topology on $\mathsf{M}_\infty$:
    By \cite[Theorem~13.2 and Theorem~13.11]{Ke95} there is a finer Polish topology $\mathcal{T}$ on $\mathsf{M}_\infty$ that generates the original Borel $\sigma$-field and for which the disintegrations 
   
     $\P^{\fp Y}_t$,  $\P^{\fp X}_{t+\varepsilon}$, and $\P^{\fp X}_1$
     are continuous.

    Thus, by  a monotone class argument, 
    \[ C_t = \bigcap_{f,g\in C_b(\mathsf{M}_\infty,\mathcal{T})} \left\{ \pi \in \cpl(\bbX,\bbY) :  \int \Phi_{f,g}\,d\pi=0 \right\} .\]

    Denote by $w_\mathcal{T}$ the coarsest topology on $\mathcal{P}(\mathsf{M}_\infty\times \mathsf{M}_\infty)$ for which $\rho\mapsto \int \Psi \,d\rho$ is continuous for every bounded $\Psi\colon \mathsf{M}_\infty\times \mathsf{M}_\infty\to\mathbb{R}$ that is $ \mathcal{T}\otimes\mathcal{T}$-continuous.
    Since $\mathcal{T}$ is a Polish topology, $\cpl(\bbX,\bbY)$ is $w_\mathcal{T}$-compact, see e.g.\ \cite[Lemma~4.4]{Vi09}.
    Moreover, since $\Phi_{f,g}$ is $\mathcal{T}\otimes\mathcal{T}$-continuous and bounded, $C_t$ is a $w_\mathcal{T}$-closed subset, hence $w_\mathcal{T}$ compact.
    Finally, note that $w_\mathcal{T}$ is finer than the weak topology (because $\mathcal{T}$ is finer topology than the original topology on $\mathsf{M}_\infty$). Hence, $C_t$ is weakly compact, as claimed.

    The same argument with the roles of $\fp X$ and $\fp Y$ reversed shows that $\cpl^\epsilon_{\rm c}(\bbY,\bbX)$ is compact, which completes the proof. \qedhere

    \end{proof}

	\begin{lemma}\label{lem:eps_causa_rightcont}
		Let $\bbX,\bbY\in \mathcal{FP}$, $\pi\in\cpl(\bbX,\bbY)$, and $\varepsilon\geq 0$.
		Then $\pi$ is $\varepsilon$-bicausal if and only if $\pi$ is $\delta$-bicausal for all $\delta>\varepsilon$.
		
	\end{lemma}
	\begin{proof}
		We only prove the non-trivial direction, i.e.\ that if $\pi$ is $\delta$-bicausal for all $\delta>\varepsilon$, then $\pi$ is $\varepsilon$-bicausal.
		
		To that end, let $t \in [0,1]$.
		By  Lemma \ref{lem:causal_equiv}, we have for all bounded random variables $U$ and $V$ that are measurable w.r.t.\ to $\F_t^\bbY$ and  $\F_1^\bbX$, respectively, 
		\begin{align}
			\label{eq:eps.causal.closed.in.eps}
			\E_\pi[UV|\F_{t+\delta}^\bbX] 
			= \E_\pi[U|\F_{t+\delta}^\bbX]\E_\pi[V|\F_{t+\delta}^\bbX].
		\end{align}
		By right-continuity of $(\mathcal{F}^{\bbX}_t)_t$ and backward martingale convergence, as $\delta\downarrow \varepsilon$, all terms in  \eqref{eq:eps.causal.closed.in.eps} converge to the respective terms with $\varepsilon$ in place of $\delta$.
		Thus, $\pi$ is $\varepsilon$-causal by  Lemma \ref{lem:causal_equiv}.
		Reversing the roles of $\bbX$ and $\bbY$ completes the proof.
	\end{proof}

	\begin{proposition}\label{prop:AW_inf_attained}
		The infimum in the definition of $\AW_p(\bbX,\bbY)$ is attained for $\bbX,\bbY\in \mathcal{CFP}_p$.
	\end{proposition}
	\begin{proof}
		Let $(\varepsilon_n)_n$ and $(\pi_n)_n$ be  minimizing for $\AW(\bbX,\bbY)$, that is, 
		\[ \AW_p(\fp X,\fp Y)=\lim_n \left( \E_{\pi^n}[ d_{\mathcal X}^p(X,Y)]^{1/p} + \varepsilon_n \right) \]
		and $\pi_n\in\cpl^{\varepsilon_n}_{\rm bc}(\fp X,\fp Y)$.
		After passing to a (not-relabelled) subsequence, there is $\varepsilon\in[0,1]$ such that $\varepsilon_n\to \varepsilon$.
		Clearly, it is enough to show that there is $\pi\in\cpl^{\varepsilon}_{\rm bc}(\bbX,\bbY)$ such that 
				\begin{align}
					\label{eq:attainement.lsc}
			\E_\pi[ d_{\mathcal X}^p(X,Y)] \leq \liminf_n \E_{\pi_n}[ d_{\mathcal X}^p(X,Y)].
				 \end{align}
    
    It is a property of canonical filtered processes that the maps $X : \Omega^{\fp X} \to \mathcal{X}$ and $Y : \Omega^{\fp Y} \to \mathcal{X}$ are continuous on the support of $\P^{\fp X}$ and $\P^{\fp Y}$ respectively, see Remark~\ref{rem:canonCont}. Hence, there is a closed set $A \subset \Omega^{\fp X} \times \Omega^{\fp Y}$ such that $d_{\mathcal X}^p(X,Y) : A \to [0,\infty)$ is continuous and $\supp(\pi) \subset A$ for all $\pi \in \cpl(\fp X, \fp Y)$.  Therefore, $\cpl(\fp X, \fp Y) \to \R : \rho \mapsto \E_\rho[ d_{\mathcal X}^p(X,Y)]$ is lower-semicontinuous.

    To prove \eqref{eq:attainement.lsc},  possibly after passing to another subsequence, we may assume without loss of generality that $(\varepsilon_n)_n$ is either increasing or decreasing.
		
		In case that $(\varepsilon_n)_n$ is increasing, the claim follows from the compactness of $\cpl^{\varepsilon}_ {\rm bc}(\bbX,\bbY)$ that was established in Proposition \ref{prop:epsbcclosed}, and from lower-semicontinuity of $\rho\mapsto \E_\rho[ d_{\mathcal X}^p(X,Y)]$.
		
		In case that $(\varepsilon_n)_n$ is decreasing, by Proposition \ref{prop:epsbcclosed}, possibly after passing to a subsequence, there is a coupling $\pi$ such $\pi_n\to\pi$ weakly. 
		Moreover, for every $\delta>\varepsilon$ and all $n$ large enough, $\pi_n$ is $\delta$-bicausal.
        By Proposition \ref{prop:epsbcclosed}, $\pi$ is $\delta$-bicausal and as $\delta>\varepsilon$ was arbitrary, Lemma \ref{lem:eps_causa_rightcont} implies that  $\pi$ is $\varepsilon$-bicausal.
		The claim again follows from lower semicontinuity of $\rho\mapsto \E_\rho[ d_{\mathcal X}^p(X,Y)]$.
	\end{proof}

A basic yet important concept in the classical theory of Wasserstein distances is that of `gluing'. In the present setting it takes the following form.
	
	\begin{lemma}
	\label{lem:glueing}
		Let $\bbX,\bbY,\bbZ\in \mathcal{CFP}$, let $\pi \in \cpl(\bbX,\bbY)$ and $\rho \in \cpl(\bbY,\bbZ)$. 
		Then there is $\Pi\in \prob(\Omega^\bbX\times\Omega^\bbY \times \Omega^\bbZ)$ such that 
		\[ \pi = ({\pr_{\Omega^\bbX \times \Omega^\bbY}})_\#\Pi,
        \qquad
		\rho = ({\pr_{\Omega^\bbY \times \Omega^\bbZ}})_\#\Pi,
	\quad\text{and}\quad
	\F^\bbX \indep_{ \F^\bbY_1} \F^\bbZ \text{ under $\Pi$}.\]
		
		If, in addition, $\pi$ is $\epsilon_1$-causal from $\fp X$ to $\fp Y$ and $\rho$ is $\varepsilon_2$-causal from $\fp Y$ to $\fp Z$, then $\sigma:=({\pr_{\Omega^\bbX \times \Omega^\bbZ}})_\#\Pi$ is $(\varepsilon_1+\varepsilon_2)$-causal from $\fp X$ to $\fp Z$.
	\end{lemma}
	\begin{proof}
	Since $\F_1^{\fp Y} \subset \F^{\fp Y}$ and the latter is countably generated (because $\fp Y$ is a canonical filtered process and hence based on a standard Borel probability space, see Section~\ref{sec:notbg}), there are disintegrations $\pi^{\F_1^{\fp Y}}$ and $\rho^{\F_1^{\fp Y}}$  of $\pi$ and $\rho$ w.r.t.\ $\F_1^{\fp Y}$:
		\begin{align*}
			\pi(d\omega_x,d\omega_y)
			&=\pi^{\F_1^{\fp Y}}_{\omega_y}(d\omega_x)\, \P^{\bbY}(d\omega_y), \\
			\rho(d\omega_y,d\omega_z)
			&=\rho^{\F_1^{\fp Y}}_{\omega_y}(d\omega_z)\,\P^{\bbY}(d\omega_y),
		\end{align*}
		see e.g.\ \cite[Corollary~10.4.13]{Bo07a}.
		Now define
		\begin{align*}
			\Pi(d\omega_x,d\omega_y,d\omega_z)
			 :=\rho^{\F_1^{\fp Y}}_{\omega_y} (d\omega_z)\,\pi^{\F_1^{\fp Y}}_{\omega_y}(d\omega_x)\,\P^{\bbY}(d\omega_y).
		\end{align*}
		Clearly,  we have  $({\pr_{\Omega^\bbX \times \Omega^\bbY}})_\#\Pi= \pi$ and $({\pr_{\Omega^\bbY \times \Omega^\bbZ}})_\#\Pi = \rho$. For functions $U,V,W$ that are  bounded and measurable w.r.t.\ $\F^\bbX,\F^\bbY_1,\F^\bbZ$, respectively, it follows from the definition of $\Pi$ that 
		\begin{align*}
		\E_\Pi[UVW] 
		 = \E_\Pi \left[ V \E_\Pi[U|\F^\bbY_1]\E_\Pi[W|\F^\bbY_1]  \right], 		
		\end{align*}
            i.e.\ we have  $\F^\bbX \indep_{ \F^\bbY_1 } \F^\bbZ$ under $\Pi$.

	It remains to prove that $\F_t^\bbZ \indep_{\F^\bbX_{t+\epsilon_1+\epsilon_2}} \F_1^\bbX$  for every $t\in[0,1]$ under $\Pi$ and hence under $\sigma$. To that end, note that we have
		\begin{enumerate}[(a)]
			\item $\F^\bbY_t \indep_{ \F^\bbX_{t+\epsilon_1}} \F^\bbX_1$ under $\Pi$ for all $t \in [0,1]$, \label{it:prf.glue:a} 
			\item $\F^\bbZ_t \indep_{ \F^\bbY_{t+\epsilon_2}} \F^\bbY_1$ under $\Pi$ for all $t \in [0,1]$,\label{it:prf.glue:b}
			\item $\F^\bbX \indep_{ \F^\bbY_1 } \F^\bbZ$ under $\Pi$.\label{it:prf.glue:c}
		\end{enumerate}
Fix $t \in [0,1]$. For any bounded $\mathcal{F}^{\bbZ}_t$-measurable random variable $W$, by first using the conditional independence \ref{it:prf.glue:c} and then the causality condition \ref{it:prf.glue:b}, we have that
	\begin{align}\label{eq:prf.glue1}
	   \E_\Pi[W | \, \mathcal{F}^{\bbX,\bbY}_{1,1}] =\E_\Pi[W  | \, \mathcal{F}^{\bbY}_{1}] =\E_\Pi[W  | \, \mathcal{F}^{\bbY}_{t+\epsilon_2}]. 
	\end{align}
   It follows that for every  bounded $\mathcal{F}^{\bbZ}_t$-measurable random variable $W$,
	\begin{align*}
	\E_\Pi[W| \, \mathcal{F}^{\bbX}_1]=&\E_\Pi[ \E_\Pi[W| \, \mathcal{F}^{\bbX,\bbY}_{1,1}]| \, \mathcal{F}^{\bbX}_1]=
	\E_\Pi[\E_\Pi[W  | \, \mathcal{F}^{\bbY}_{t+\epsilon_2}] | \, \mathcal{F}^{\bbX}_1]=\E_\Pi[\E_\Pi[W  | \, \mathcal{F}^{\bbY}_{t+\epsilon_2}] | \, \mathcal{F}^{\bbX}_{t+\epsilon_1+\epsilon_2}] \\
	=& \E_\Pi[\E_\Pi[W  | \, \mathcal{F}^{\bbX,\bbY}_{1,1}] | \, \mathcal{F}^{\bbX}_{t+\epsilon_1+\epsilon_2}]=\E_\Pi[W | \, \mathcal{F}^{\bbX}_{t+\epsilon_1+\epsilon_2}],
	\end{align*}
    where the second and fourth equality are due to \eqref{eq:prf.glue1} and the third equality is due to the causality condition \ref{it:prf.glue:a} at time $t+\epsilon_2$. Therefore, we have shown that $\E_\sigma[W| \, \mathcal{F}^{\bbX}_1] = \E_\sigma[W|\mathcal{F}^{\bbX}_{t+\epsilon_1+\epsilon_2}]$. By Lemma~\ref{lem:causal_equiv} the latter is equivalent to $\F_t^\bbZ \indep_{\F^\bbX_{t+\epsilon_1+\epsilon_2}} \F_1^\bbX$ under $\sigma$. 
\end{proof}

	\begin{lemma}\label{lem:AW_triangluar_ineq}
	For all $\fp X, \fp Y, \fp Z \in \mathcal{CFP}_p$ we have $\AW_p(\fp X, \fp Z) \le \AW_p(\fp X, \fp Y) + \AW_p(\fp Y, \fp Z)$.
	\end{lemma}	
	\begin{proof}
	By Proposition~\ref{prop:AW_inf_attained} there are $\epsilon_1\ge 0$ and $\pi_1 \in \cpl^{\epsilon_1}_{\rm bc}(\bbX,\bbY)$ which attain the infimum in the definition of $\AW_p(\fp X,\fp Y)$, and $\varepsilon_2$ and  $\pi_2 \in \cpl^{\epsilon_2}_{\rm bc}(\bbY,\bbZ)$ which attain the infimum in the definition of $\AW_p(\fp Y,\fp Z)$.

		By Lemma \ref{lem:glueing} there exists $\Pi\in \prob(\Omega^\bbX \times \Omega^\bbY \times \Omega^\bbZ)$ such that $({\pr_{\Omega^\bbX \times \Omega^\bbY}})_\#\Pi= \pi_1$, $({\pr_{\Omega^\bbY \times \Omega^\bbZ}})_\#\Pi= \pi_2$ and 
		\[\sigma:=({\pr_{\Omega^\bbX \times \Omega^\bbZ}})_\#\Pi\in \cpl^{\epsilon_1+\epsilon_2}_{\rm  bc}(\bbX,\bbZ).\]
	Thus
		\begin{align*}
			\AW_p(\bbX,\bbZ) 
			&\le \E_\sigma[d_{\mathcal X}^p(X,Z)]^{1/p} + \epsilon_1 + \epsilon_2 \\
			&\le \E_\Pi[d_{\mathcal X}^p(X,Y)]^{1/p} + \E_\Pi[d_{\mathcal X}^p(Y,Z)]^{1/p} + \epsilon_1 + \epsilon_2 \\
			&= \AW_p(\bbX,\bbY) + \AW_p(\bbY,\bbZ). \qedhere
		\end{align*}
	\end{proof}
    For the proof of the next proposition we need to discretize filtered processes in time. To that end, we introduce some notation:
     A \emph{grid} is a finite set $T \subset [0,1]$  that satisfies $0 \notin T$ and $1 \in T$. We denote the elements of $T$ as $0<t_1 <t_2 < \dots < t_N=1$, in particular $N$ denotes the cardinality of $T$. For notational convenience\footnote{Further on, the behaviour of $0$ is different from that of grid points in $(0,1]$. To avoid cumbersome case distinctions between 0 and the other grid points, we agreed that 0 is not a grid point (i.e.\ $0 \notin T$), but we still allow the notation $t_0=0$. Note that $t_0$ appears in the definition of $\mesh(T)$ below, so $\mesh(T)$ small also implies that $t_1$ is close to $0$.}, set $t_0:=0$.    
     Define a function that rounds $t \in [0,1)$ up to the next element in $T$, that is 
        \begin{align}
            \label{eq:def.rounding}
    \lceil t \rceil_T := \min \{  s \in T : s >t\}.
    \end{align}
        and set $\lceil 1 \rceil_T = 1$. In particular, $\lceil t_i \rceil_T = t_{i+1}$ for $i<N$. Finally, the \emph{mesh} of the grid $T$ is given by  
        \[
        \textup{mesh}(T) := \max_{i=1,\dots,N} |t_i -t_{i-1}|.
        \]

    \begin{proposition}\label{prop:FiniteOmegaDense}
 Let $\fp X \in \mathcal{FP}_p$ and $\epsilon >0$. Then there is  $\fp Y \in \mathcal{FP}_p$ that is defined on a finite probability space such that $\AW_p(\fp X, \fp Y) < \epsilon$. 
    \end{proposition}

    \begin{proof} 
    Fix $\fp X \in \mathcal{FP}$  and  $\epsilon>0$. We will discretize the process $\fp X $ in time and employ that the discrete-time analogue to Proposition~\ref{prop:FiniteOmegaDense} is already known.
    
    \noindent
    \emph{Step 1: Time-discretization of the paths of $\fp X$:} 
        
        Let $T$ be a grid. We first introduce a map   
         $     \iota_T : \R^{d(N+1)} \to \mathcal{X}  $
    that assigns to a vector $a = (a_i)_{i=0}^N  \in \R^{d(N+1)} $ a path $f \in \mathcal{X}$ satisfying $f(t_i)=a_i$ for all $i=0,\dots,N$.  More precisely, 
    \begin{itemize}
        \item  if $\mathcal{X}=C([0,1];\R^d)$, we set $\iota_T(a)(t_i) :=a_i$ and interpolate linearly between the grid points;
        \item if $\mathcal{X}=D([0,1];\R^d)$, we set $\iota_T(a)(t) :=a_{\lceil t \rceil_T \,-1}$ for $t \in [0,1)$ and $\iota_T(a)(1)=a_N$.
    \end{itemize}
        It follows immediately from the construction that for all $a, b \in \R^{d(N+1)}$,
        \begin{align}\label{eq:prf:GridEst}
            d_{\mathcal X}(\iota_T(a),\iota_T(b)) \le \max_{i=0,\dots, N} |a_i-b_i|.
        \end{align}
        Consider the evaluation map\footnote{Note that we have to be careful because $e_T$ is not continuous in the case  $\mathcal X = D([0,1];\R^d)$. } $e_T : \X \to \R^{d(N+1)} : f \mapsto (f(t_i))_{i=0}^N$. For every  $f \in \mathcal{X}$ we have
        \begin{align*}
            d_{\mathcal X}(f, \iota_T(e_T(f))) \to 0 \qquad \textup{as } \mathrm{mesh}(T) \to 0.
        \end{align*}
        Indeed, in the case $\mathcal{X} = C([0,1];\R^d)$ this is easy to check and in the case $\mathcal{X} = D([0,1];\R^d)$ we refer to \cite[Section 12, Lemma 3]{Bi99}. Moreover, it is clear from the construction that $\|\iota_T(e_T(f))\|_\infty \le \|f \|_\infty$ for all $f \in \mathcal{X}$. As $d_{\mathcal X}(g,0) \le C\|g\|_\infty$ for all $g \in \mathcal{X}$, it follows from the dominated convergence theorem that
        $$
        \E_{\P^{\fp X}}[ d^p_{\mathcal{X}}(X, \iota_T(e_T(X))) ] \to 0  \qquad \textup{as } \mathrm{mesh}(T) \to 0.
        $$
        In particular, for the rest of the proof,  we can fix a grid $T$ satisfying $\textup{mesh}(T)<\epsilon$ and 
        \begin{align}\label{eq:prf:XtoIotaX}
                    \E_{\P^{\fp X}}[ d^p_{\mathcal{X}}(X, \iota_T(e_T(X))) ]^{1/p} < \epsilon.
        \end{align}

    \noindent
        \emph{Step 2: Approximation by a process on finite $\Omega$ in discrete time:}
        
        Next, we consider the discrete-time process
        $$
        \fp X^T  := (\Omega^{\fp X}, \F^{\fp X}, \P^{\fp X}, (\F_{t_i}^{\fp X})_{i=0}^N, (X_{t_i})_{i=0}^N).
        $$
        Note that this process is a discrete-time filtered process as considered in \cite{BaBePa21}. According to \cite[Theorem~5.4]{BaBePa21} every discrete-time filtered process can be approximated by processes that are defined on a finite probability space w.r.t.\ the discrete-time adapted Wasserstein distance; in particular, there is a discrete-time filtered process 
    $$
    \fp Z = (\Omega^{\fp Z }, \F^{\fp Z}, \P^{\fp Z}, (\F_{t_i}^{\fp Z})_{i=0}^N, (Z_{t_i})_{i=0}^N)
    $$ 
    with $\Omega^{\fp Z}$ finite  and a coupling $\pi \in \cpl(\fp X^T, \fp Z )$ such that
    \begin{align}\label{eq:prf:GridClose}
            \E_\pi \Big[ \max_{i=0,\dots,N} |X_{t_i} - Z_{t_i}|^p \Big]^{1/p} < \epsilon
    \end{align}
    and $\pi$ is bicausal on the grid $T$, i.e.\ we have under $\pi$ for all $i=1,\dots,N$
    \begin{align} \label{eq:prf:discrDense:C}
    \F^{\fp Z}_{t_i} \indep_{ \F_{t_i}^{\fp X} } \F_1^{\fp X} \quad\text{and}\quad
    \F^{\fp X}_{t_i} \indep_{ \F_{t_i}^{\fp Z} } \F_1^{\fp Z}. 
    \end{align}

    \noindent
     \emph{Step 3: Extend $\fp Z$ to a continuous-time process $\fp Y$ and estimate $\AW(\fp X, \fp Y)$:} 
     
     By applying the map $\iota_T$ to $Z$ and $\lceil \cdot \rceil_T$ to the filtration, we extend $\fp Z$ to a continuous time process:
    \[
\fp Y := (\Omega^{\fp Z }, \F^{\fp Z}, \P^{\fp Z}, (\F_{t}^{\fp Y})_{t \in [0,1]}, \iota_T(Z)), 
    \]
    where we set $\F_t^{\fp Y} := \F^{\fp Z}_{\lceil t \rceil_T  }$ for $t \in [0,1]$. 
    We use the coupling $\pi$ from Step 2 to estimate $\AW_p(\fp X, \fp Y)$. First, we claim that $\pi \in \cplbc^\epsilon( \fp X, \fp Y)$. 
    
    To prove causality from $\fp Y$ to $\fp X$, fix $t \in [0,1]$ and let $i \in \{0,\dots, N\}$ such that $\lceil t \rceil_T = t_i$. As $\F_t^{\fp Y} = \F_{t_i}^{\fp Z}$ and $\F_t^{\fp X} \subset \F_{t_i}^{\fp X}$, the second statement in \eqref{eq:prf:discrDense:C} implies  $\F_t^{\fp X} \indep_{\F_t^{\fp Y}} \F_1^{\fp Y}$. 
    
    To prove causality from $\fp X$ to $\fp Y$,   
    fix $t \in [0,1]$ and note that, by Lemma~\ref{lem:causal_equiv}, it suffices to show that $\E_\pi[V | \F_1^{\fp X} ] = \E_\pi[V | \F_{t+\epsilon}^{\fp X} ]$ for every $V \in L^\infty(\F_t^{\fp Y})$. To that end, let $i \in \{0,\dots, N\}$ such that $\lceil t \rceil_T = t_i$ and observe that, by Lemma~\ref{lem:causal_equiv}, the first statement in \eqref{eq:prf:discrDense:C} implies $\E_\pi[V | \F_1^{\fp X} ] = \E_\pi[V | \F_{t_i}^{\fp X} ]$. As $\F^{\fp X}_{t_i} \subset \F_{t+\epsilon}^{\fp X} \subset \F_1^{\fp X}$, the latter implies $\E_\pi[V | \F_1^{\fp X} ] = \E_\pi[V | \F_{t+\epsilon}^{\fp X} ]$. 

    Next, let us estimate the cost of $\pi$. We have  
    \begin{align*}
    \E_\pi[ d_\mathcal{X}^p(X,Y)]^{1/p} &\le \E_{\P^{\fp X}}[ d_\mathcal{X}^p(X,\iota_T(e_T(X)))]^{1/p} +\E_\pi[ d_\mathcal{X}^p(\iota_T(e_T(X))),\iota_T( \widehat{Y}))]^{1/p} =: A+B. 
    \end{align*}  
    It follows from \eqref{eq:prf:XtoIotaX} that $A \le \epsilon$ and from \eqref{eq:prf:GridEst} and \eqref{eq:prf:GridClose} that $B \le \epsilon$. We conclude that $\AW_p(\fp X, \fp Y) \le 3 \epsilon$.
    \end{proof}

    \begin{proposition}\label{prop:AW_CFP}
    Let $\fp X, \fp Y \in \mathcal{FP}_p$ and let $\overline{\fp X}, \overline{\fp Y} \in \mathcal{CFP}_p$ be the associated canonical filtered processes. Then we have 
    \[
    \AW_p(\fp X, \fp Y) = \AW_p(\overline{ \fp X}, \overline{ \fp Y}).     
    \] 
    \end{proposition}

    The proof of Proposition~\ref{prop:AW_CFP} relies on Proposition~\ref{prop:FiniteOmegaDense} and is technically rather involved and therefore postponed to Appendix~\ref{sec:NonPolish}. 

    In particular, it follows from Proposition \ref{prop:AW_CFP} that if $\fp X \sim_{\HK} \widetilde{\fp X}$ and $\fp Y \sim_{\HK} \widetilde{\fp Y}$, then $\AW_p(\fp X, \fp Y) = \AW_p(\widetilde{ \fp X}, \widetilde{ \fp Y})$, cf.\ Lemma~\ref{lem:canonRep}. 
    Therefore, the value of $\AW_p$ is independent of the choice of $\sim_{\HK}$-representatives. Thus, $\AW_p$ is well-defined on the factor space $\FP_p = \mathcal{FP}_p /_{ \sim_{\HK}}$.

    \begin{remark}\label{rem:defAWflexibility2}
   
    Proposition~\ref{prop:FiniteOmegaDense} is the only instance in this section, where the specific choice of the path spaces and the metrics thereon is used, i.e.\ when generalizing the concept of adapted Wasserstein distance to other path spaces with a separable metric it suffices to establish an analog of Proposition~\ref{prop:FiniteOmegaDense}. Note that if the given metric is not complete, all results except for completeness of $\AW_p$ still hold true.
    \end{remark}

\subsection{$\AW$ as a metric on $\FP$.}
Recall from Definition~\ref{def:FPpHKp} that a sequence $(\fp X^n)_n$ in $\FP_p$ converges to $\fp X \in \FP_p$ w.r.t.\ $\HK_p$ if $\fp X^n  \to \fp X $ in $\HK$ and $\E_{\mathbb{P}^n}[d_{\mathcal{X}}^p(X^n,0)]\to \E_{\mathbb{P}}[d_{\mathcal{X}}^p(X,0)]$. The main result of this section is that $\AW_p$ induces the topology $\HK_p$:

\begin{theorem}\label{thm:AWmetrizesHK}
$\AW_p$ is a complete metric on $(\FP_p,\HK_p)$. In particular, 
\begin{enumerate}[label = (\roman*)]
    \item For all $\fp X, \fp Y \in \mathcal{FP}_p$ we have $\AW_p(\fp X, \fp Y) =0$ if and only if $\fp X \sim_{\HK} \fp Y$.
    \item For every sequence $(\fp X^n)_n$ in $\FP_p$ and $\fp X \in \FP_p$ we have that $\fp X^n\to \fp X$ in $\HK_p$ if and only if $\AW_p(\fp X^n,\fp X)\to 0$.
\end{enumerate}
 \end{theorem}

\begin{remark}
If $d_\X$ is a bounded metric, then $\FP_p = \FP$ and the adapted Wasserstein distance w.r.t.\ this metric induces the Hoover--Keisler topology. 
\end{remark} 

As mentioned in Section~\ref{sec:Recap}, the topology $\HK_p$ is metrizable, hence the claim about sequences in Theorem~\ref{thm:AWmetrizesHK} indeed implies equality of the topologies. In order to prove Theorem~\ref{thm:AWmetrizesHK} we first show that the equivalence classes of $\sim_{\AW}$ and $\sim_{\HK}$ coincide:       
\begin{proof}[Proof of Theorem~\ref{thm:AWmetrizesHK}, part (i)]
    Assume that $\fp X \sim_{\HK} \fp Y$. Then we have $\overline{\fp X} = \overline{ \fp Y}$ and hence Proposition~\ref{prop:AW_CFP} yields 
    \[
    \AW_p(\fp X, \fp Y)= \AW_p(\overline{ \fp X }, \overline{  \fp Y }) = \AW_p(\overline{\fp X}, \overline{\fp X})=0.
    \]
    Conversely, assume that $\AW_p(\fp X, \fp Y) = 0$. By Proposition~\ref{prop:AW_CFP} the associated canonical processes $\overline{\fp X},\overline{\fp Y} \in \mathcal{CFP}_p$ satisfy $\AW_p(\overline{ \fp X}, \overline{ \fp Y})=0$ as well.  By Proposition~\ref{prop:AW_inf_attained} the infimum in the definition of $\AW_p(\overline{\fp X}, \overline{\fp Y})$ is attained, i.e.\ there is $\pi \in \cpl_{\rm bc}(\overline{ \fp X},\overline{ \fp Y})$ such that $\overline X=\overline Y$ $\pi$-a.s.

    We show by induction on  $r \in \N \cup \{  0 \}$ that $\pp^r(\overline{ \fp X}) = \pp^r(\overline{ \fp Y})$ $\pi$-a.s.\ For $r = 0$ we have by assumption $\pp^0(\overline{ \fp X}) = \overline X = \overline Y = \pp^0(\overline{ \fp Y})$ $\pi$-a.s.
	    Now assume that the claim is true for $r \in \N \cup \{ 0 \}$, then for every $t \in [0,1]$, by Lemma \ref{lem:causal_equiv}, 	    \begin{align*}
	        \pp^{r+1}_t(\overline{ \fp X}) &= \law_\pi \big( \pp^r({ \overline{ \fp X}}) \big| \F_t^{\overline{ \fp X}} \big) 
	        = \law_\pi \big( \pp^r({ \overline{ \fp X}} ) \big| \F_{t,t}^{\overline{ \fp X}, \overline{ \fp Y}} \big)
	        \\
	        &= \law_\pi \big( \pp^r(\overline{ \fp Y}) \big| \F_{t,t}^{\overline{ \fp X},\overline{ \fp Y}} \big)
	        = \law_\pi \big( \pp^r(\overline{ \fp Y}) \big| \F_t^{\overline{ \fp Y}} \big)
	        = \pp^{r + 1}_t(\overline{ \fp Y})
	    \end{align*}
	    $\pi$-a.s.
	    Hence, $\pp^{r + 1}(\fp X) = \pp^{r + 1}(\fp Y)$ as both processes have \cadlag{} paths, which concludes the induction step.     Therefore, 
     \[
     \law(\pp^\infty(\fp X)) = \law(\pp^\infty( \overline{ \fp X})) = \law(\pp^\infty( \overline{ \fp Y})) = \law(\pp^\infty(\fp Y)),
     \]
     i.e.\ $\fp X \sim_{\HK} \fp Y$.
	\end{proof}

        The second part of the proof of Theorem~\ref{thm:AWmetrizesHK} employs that the respective statement is already known in discrete time: It was shown in \cite{BaBePa21} that for discrete-time filtered processes $\fp X, \fp X^1, \fp X^2,\ldots$, we have that $\AW_p(\fp X^n , \fp X) \to 0$ if and only if $\fp X^n \to \fp X$ in the discrete-time variant of the Hoover--Keisler topology and $\E_{\mathbb{P}^n}[d_{\mathcal{X}}^p(X^n,0)]\to \E_{\mathbb{P}}[d_{\mathcal{X}}^p(X,0)]$. We start by defining a discretization operator that relates the discrete- and continuous-time frameworks. Recall  $\lceil\cdot \rceil_T$ defined in \eqref{eq:def.rounding} and set 
        \begin{align}\label{eq:def:DT}
            \mathcal{D}^T(\fp X) := (\Omega^{\fp X}, \F^{\fp X}, \P^{\fp X}, (\F_{\lceil t \rceil_T })_{t \in [0,1]}, X).
        \end{align}
        Note that $\mathcal{D}^T(\fp X) $ is still a continuous-time filtered process but its filtration is piecewise constant.

    \begin{lemma}\label{lem:AWMesh}
    For every $\fp X \in \mathcal{FP}_p$ and grid $T \subset [0,1]$, we have $\AW_p(\fp X, \mathcal{D}^T(\fp X)) \le \mesh(T)$. 
    \end{lemma}
    \begin{proof}
    Set $\epsilon := \mesh(T)$ and note that $\F_t^{\fp X} \subset \F^{\fp X}_{\lceil t \rceil_T} \subset \F^{\fp X}_{t +\epsilon}$ for all $t \in [0,1]$. Hence, it follows immediately from Lemma~\ref{lem:causal_equiv}(ii) that the identity coupling $(\id,\id)_\#\P^{\fp X}$ is $\epsilon$-causal from $\fp X$ to $\mathcal{D}^T(\fp X)$ and causal from $\mathcal{D}^T(\fp X)$ to $\fp X$. 
    \end{proof}

    \begin{proposition}\label{prop:HKtoAWdiscr}
    Let $(\fp X^n)_n$ be a sequence in $\mathcal{FP}_p$ and let $\fp X \in \mathcal{FP}_p$. Fix a grid $T \subset \cont(\fp X) \cup \{1\}$. If 
    $\fp X^n \to \fp X$ in $\HK_p$, then $\AW_p(\mathcal{D}^T(\fp X^n), \mathcal{D}^T(\fp X)) \to 0$.   
    \end{proposition}
    \begin{proof}
    The proof uses results of \cite{BaBePa21} formulated in a discrete-time framework. For technical reasons\footnote{The more natural approach of using $(X_{t_i})_{i=1}^N$ instead of $(0,\dots, 0,X)$ bears the problem that  $f \mapsto (f(t_i))_{i=1}^N$ is not continuous in the case $\mathcal X = D([0,1];\R^d)$.} it is convenient to consider the following modification of $\mathcal{D}^T(\fp X)$: For $\fp X \in \mathcal{FP}_p$, we set 
    \[
    \widetilde{\mathcal{D}}^T(\fp X) := (\Omega^{ \fp X }, \F^{ \fp X } , \P^{ \fp X }, (\F^{ \fp X }_{t_i})_{i=1}^N, (0, \dots, 0, X)   ),
    \] 
    where $0 \in \mathcal{X}$ is the constant function 0, and $X \in \mathcal{X}$ is the whole path of $\fp X$; see \cite[Chapter~8]{BePaScZh23} for more  details about this construction.  

    Next, we  note that a coupling between $\mathcal{D}^T(\fp X)$ and $\mathcal{D}^T(\fp Y)$ is bicausal if and only if it is bicausal (in discrete time) between $\widetilde{\mathcal{D}}^T(\fp X)$ and $\widetilde{\mathcal{D}}^T(\fp Y) $, that is  $\F_{t_i}^{\fp Y} \indep_{\F_{t_i}^{\fp X}} \F_1^{\fp X}$ and $\F_{t_i}^{\fp X} \indep_{\F_{t_i}^{\fp Y}} \F_1^{\fp Y}$ under $\pi$ for all $i =1,\dots,N$. Indeed, this follows from the definition of $\widetilde{\mathcal{D}}^T(\fp X)$ since its filtration is precisely the collection of the $\sigma$-algebras that apprear in the piecewise constant filtration of $\mathcal{D}^T(\fp X)$, see \eqref{eq:def:DT}.

    Now, consider a sequence $(\fp X^n)_n$ in $\mathcal{FP}_p$, $\fp X \in \mathcal{FP}_p$ and fix a grid $T \subset \cont(\fp X)$. Suppose that $\fp X^n \to \fp X$ in $\HK_p$. It was shown in \cite[Corollary~8.9]{BePaScZh23} that  $\widetilde{\mathcal{D}}^T(\fp X^n) \to \widetilde{\mathcal{D}}^T(\fp X)$ in the discrete-time Hoover--Keisler topology. Moreover, it was shown in \cite{BaBePa21} that the latter implies $\fp X^n \to \fp X$ w.r.t.\ the discrete-time adapted Wasserstein distance, i.e.\ there are bicausal couplings $\pi^n$ between $\widetilde{\mathcal{D}}^T(\fp X)$ and $\widetilde{\mathcal{D}}^T(\fp X^n)$ such that $\E_{\pi^n} [d_{\mathcal X}^p(X,X^n)] \to 0$.  As explained before, $\pi^n$ is also bicausal between ${\mathcal{D}}^T(\fp X)$ and ${\mathcal{D}}^T(\fp X^n)$, hence $\AW_p(\mathcal{D}^T(\fp X^n), \mathcal{D}^T(\fp X)) \to 0$. 
    \end{proof}

We are now ready to prove the remaining parts of Theorem~\ref{thm:AWmetrizesHK}.

\begin{proof}[Proof of Theorem~\ref{thm:AWmetrizesHK}, part (ii)]
        First note that $\AW_p$ is a well-defined metric: Proposition~\ref{prop:AW_CFP} establishes that the value of $\AW_p(\fp X, \fp Y)$ is independent of the representatives of the equivalence classes of $\fp X$ and $\fp Y$. Hence, Lemma~\ref{lem:AW_triangluar_ineq} implies that $\AW_p$ satisfies the triangle inequality. We have already shown that $\AW_p(\fp X, \fp Y)=0$ if and only if $\fp X \sim_{\HK} \fp Y$; the symmetry of $\AW_p$ is obvious. As the product coupling is bicausal, we have $\AW_p(\fp X, \fp Y) < \infty$ for all $\fp X, \fp Y \in \FP_p$.	    
        
        Next, let us show that $\AW_p$ metrizes the $\HK_p$-topology. Suppose that $\fp X^n \to \fp X$ in $\HK_p$ and fix $\epsilon >0$.  As $\cont(\fp X)$ is dense in $[0,1]$, see Lemma~\ref{lem:cont_pt_dense}, there exists a  grid $T \subset \cont(\fp X) \cup \{1\} $ with $\mesh(T) < \epsilon$. By Proposition~\ref{prop:HKtoAWdiscr} we have 
        $
	    \AW_p(\mathcal{D}^T(\fp X^n), \mathcal{D}^T(\fp X)) \to 0$ and by Lemma~\ref{lem:AWMesh}, 
        \begin{align*}
           \AW_p(\fp X^n,\fp X) &\le  \AW_p(\fp X^n, \mathcal{D}^T(\fp X^n))+ \AW_p( \mathcal{D}^T(\fp X^n),  \mathcal{D}^T( \fp X )) + \AW_p(\mathcal{D}^T(\fp X), \fp X) \\
           &\le \AW_p( \mathcal{D}^T(\fp X^n),  \mathcal{D}^T( \fp X )) + 2\epsilon.
        \end{align*}
        Hence $ \limsup_n \AW_p(\fp X^n,\fp X) \le 2\epsilon $ and as $\epsilon>0$ was arbitrary, this proves $\AW_p(\fp X^n,\fp X) \to 0$.

	To prove the converse implication, assume that $\AW_p(\fp X^n,\fp X) \to 0$. 
		 Then  $\W_p(\law(X^n),\law(X)) \to 0$, hence $\{ \law(X^n) : n \in \N \}$ is $\mathcal W_p$-relatively compact. Therefore, Corollary~\ref{cor:CompFPp} asserts that $\{ \fp X^n : n \in \N \}$ is relatively compact in $(\FP_p,\HK_p)$. Let $\fp Y$ be any $\HK_p$-limit point of this sequence. Then  $\AW_p(\fp X^n, \fp Y) \to 0$ by the already shown implication, so $\fp X \sim_{\HK} \fp Y$. We conclude that  $\fp X^n \to \fp X$ in $\HK_p$.

Finally, we show  completeness.
To that end, let $(\fp X^n)_n$ be an $\AW_p$-Cauchy sequence. In particular, $(\law(X^n))_n$ is a $\W_p$-Cauchy sequence. By the completeness of $\W_p$, this sequence is converging and hence relatively compact in $\prob_p(\mathcal X)$. Therefore, Corollary \ref{cor:CompFPp} implies that $(\fp X^n)_n$ is relatively compact in $(\FP_p,\HK_p)$. If $\fp X, \fp Y$ are $\HK_p$-limit points of this sequence, then $\AW_p(\fp X^n, \fp X) \to 0$ and $\AW_p(\fp X^n, \fp Y) \to 0$. Hence, $\AW_p(\fp X, \fp Y)=0$, and thus the sequence $(\fp X^n)_n$ is convergent.
\end{proof}

\subsection{Symmetrized causal distance}

When showing that certain processes converge w.r.t.\ $\mathcal{AW}_p$, it can be challenging to construct couplings that are $\varepsilon$-causal in both directions at the same time.
In contrast, it is often simpler to find couplings that are just $\varepsilon$-causal in one direction.
Hence, computing the \emph{causal distance}
\begin{align}
\label{eq:def.CW}
	\mathcal{CW}_p(\fp X,\fp Y) 
	:=  \inf_{\epsilon \ge 0}    \inf \left\{ \E_\pi[d^p(X,Y)]^{1/p}  + \epsilon :  \pi \in \cpl_{\rm c}^\epsilon(\fp X,\fp Y) \right\}  
\end{align}
	between $\fp X,\fp Y\in\mathcal{FP}_p$ turns out to be simpler than computing $\mathcal{AW}_p(\fp X,\fp Y)$.
Clearly, $\mathcal{CW}_p$ is not a distance as it lacks symmetry, and it  is natural to consider the  \emph{symmetrized causal Wasserstein distance}, defined via
\[
	\SCW_p(\fp X, \fp Y) := \max\{  \CW_p(\fp X, \fp Y), \CW_p(\fp Y, \fp X) \}.
\]

We first show that $\SCW_p$ can be defined on $\FP_p=\mathcal{FP}_p/_{\sim_\HK}$, and in Proportion \ref{prop:scw} that $\SCW_p$ in fact induces the same topology as $\AW_p$. To this end, we start with the following observation.
\begin{remark}\label{rem:CWattainment}
For every $\fp X, \fp Y \in \mathcal{CFP}_p$, the infima in the definitions of $\mathcal{CW}_p(\fp X, \fp Y)$ and $\SCW_p(\fp X, \fp Y)$ are attained. This follows from Proposition~\ref{prop:epsbcclosed} using the same arguments as in the proof of Proposition~\ref{prop:AW_inf_attained}.  
\end{remark}

\begin{lemma}
\label{lem:scw.definite}
	$\SCW_p$ is a pseudo-metric on $\mathcal{FP}_p$.
	Moreover, for every $\fp X,\fp Y\in \mathcal{FP}_p$, we have that $\SCW_p(\fp X,\fp Y)=0$ if and only if $\fp X\sim_{\HK} 
  \fp Y$.
\end{lemma}
\begin{proof}
As in the case of $\AW_p$, Proposition~\ref{prop:nonPolish} yields that $\SCW_p$ does not depend on the representative of the $\HK$-class, so we can assume w.l.o.g.\ that $\fp X, \fp Y \in \mathcal{CFP}$. 	It is clear that $\SCW_p$ takes values in $[0,\infty)$ and that it is symmetric. 	Moreover, by exactly the same arguments as presented in the proof for  Lemma~\ref{lem:AW_triangluar_ineq}, one can show that  $\SCW_p$ satisfies the triangle inequality.
	
Regarding the second claim, we need to show that $\SCW_p(\fp X, \fp Y) = 0$ if and only if $\AW_p(\fp X, \fp Y)=0$. 
Since $\SCW_p\leq \AW_p$, it suffices to show that $\AW_p(\fp X,\fp Y)=0$ whenever $\SCW_p(\fp X,\fp Y)=0$.

By Remark~\ref{rem:CWattainment}, the infimum in the definition of $\SCW_p(\fp X, \fp Y)$ is attained, i.e.\ there are  $\pi \in \cplc(\fp X, \fp Y)$ and $\sigma \in \cplc(\fp Y, \fp X)$ with $X=Y$ $\pi$-a.s.\ and $\sigma$-a.s. 
Following line by line the arguments given in \cite[Lemma~2.6]{Pa22}, one can show by induction on the rank $r$ that $\pp^r(\fp X) = \pp^r(\fp Y)$ $\pi$-a.s.\ 
Hence, $\pp^\infty(\fp X) = \pp^\infty(\fp Y)$ $\pi$-a.s.\ and in particular, $\law(\pp^\infty(\fp X)) = \law(\pp^\infty(\fp Y))$. 
\end{proof}

\begin{proposition}
\label{prop:scw}
	$\SCW_p$ is a complete metric on $\FP_p$  and induces the same topology as $\AW_p$.  
\end{proposition}
	\begin{proof}
	By Lemma \ref{lem:scw.definite}, $\SCW_p$ is a metric on $\FP_p$.
	It remains to show that $\SCW_p$ induces the same topology as $\AW_p$.
	To that end, note that the topology induced by $\SCW_p$ is clearly weaker than the one induced by $\AW_p$, and at the same time finer than the weak-convergence topology.
	In particular, if $\fp X^n, \fp X$ are such that  $\SCW_p(\fp X^n, \fp X) \to 0$, then, by Corollary~\ref{cor:CompFPp}, the set $\{\law(X^n) : n\in\N \}$ is relatively compact w.r.t.\ $\AW_p$.
	Since any limit point $\fp Y$ of $(\fp X^n)_n$ w.r.t.\ $\AW_p$ also satisfies $\SCW_p(\fp Y,\fp X)=0$, Lemma \ref{lem:scw.definite} asserts that $\fp X=\fp Y$.
	In particular, $\AW_p(\fp X^n,\fp X)\to 0$.	
	Completeness of $\SCW_p$ can be shown exactly the same way as completeness of $\AW_p$ in Theorem~\ref{thm:AWmetrizesHK}.
	\end{proof}

	\section{Continuity of optimal stopping problems} \label{sec:OS}

    The main result of this section is that the value of an optimal stopping problem
	\[ {\rm OS}(\fp X,\varphi):= \inf_{\tau\in \mathrm{ST}(\bbX)}  \E[\varphi(X,\tau)] \]
	\emph{continuously} depends on the process $\bbX$ with respect to the adapted weak topology, for all cost functions  $\varphi\colon \X \times  [0,1]  \to\mathbb{R}$ that are regular in a suitable sense. Remarkably, such continuous dependence on the input process does not hold for the usual weak topology, as it was already pointed out by Aldous.

    Here and in the following, for $\fp X\in\mathcal{FP}$, the set of all $[0,1]$-valued stopping times w.r.t.\ $(\mathcal{F}_t^{\bbX})_{t\in[0,1]}$ is denoted by $\rm{ST}(\bbX)$.	
	
    \begin{definition} 
    A measurable function  $\varphi\colon  \mathcal{X} \times [0,1] \to [-\infty,\infty]$ is said to be  \emph{non-anticipative} if for every $t\in[0,1]$ and $f,g\in \mathcal{X}$ which coincide on $[0,t]$, $\varphi(f,t)=\varphi(g,t)$.
   \end{definition}
   
    We will see in Corollary~\ref{cor:OSwelldef} below that ${\rm OS}(\cdot,\varphi)$ is well-defined on $\FP$ in the sense that its value does not depend on the particular choice of the $\sim_\HK$-representative. 
    
     \begin{asn}\label{asn:costassump}
		The cost function $\varphi$ is non-anticipative and satisfies the following continuity property: for all  sequences $(s_n)_n \subset [0,1]$ and $(f_n) \subset \mathcal{X}$ such that $s_n \to s$ and $\lVert f_n-f \rVert_{\infty} \to 0$ with $f \in C([0,1];\mathbb{R}^d)$, we have that $\varphi(f_n,s_n) \to \varphi(f,s)$. 
	\end{asn}
	
	Note that Assumption \ref{asn:costassump} is satisfied whenever the map $(t,x)\mapsto \varphi(x,t)$ is jointly continuous. In particular, it is satisfied for Markovian cost functions, i.e.\ $\psi(x,t) = \psi(t,x(t))$ for $\psi$ continuous. 	
	
	\begin{proposition}[Qualitative continuity of OS]
		\label{prop:OST.qual}
		Suppose that $\bbX^n\to \bbX$ in the adapted weak topology and that $\fp X$ has continuous paths. Then, for every bounded cost function $\varphi$ satisfying Assumption~\ref{asn:costassump},
		\[ 
  {\rm OS}(\fp X^n,\varphi) \to  {\rm OS}(\fp X,\varphi).
            \]
	\end{proposition}
    
    Note that if we have $\AW_p(\fp X^n, \fp X) \to 0$, we can replace the boundedness assumption on $\varphi$ by the growth condition $\varphi(t,f) \le C(1 + d_\X^p(f,0))$. Further note that the continuity assumptions on the paths of $\fp X$ is necessary, see Example~\ref{ex:counter} below. 
    
    The proof of the continuity result for optimal stopping is based on the following observation: 
    A causal coupling $\pi \in \cplc(\fp X, \fp Y)$ allows to pull back a stopping time $\tau \in \rm{ST}(\fp Y)$ to a randomized stopping time for $\fp X$. More generally, $\epsilon$-causal couplings allow to do this up to a time shift of $\epsilon$, as the following proposition details.

    \begin{proposition}\label{prop:ST.transfer}
    Let $\varepsilon\geq 0$,  $\pi \in \cplc^\epsilon(\fp X, \fp Y)$, and $\tau \in \rm{ST}(\fp Y)$ be given. Then there is a family $(\sigma_u)_{u \in [0,1]} \subset \rm{ST}(\fp X)$ such that for every $\varphi: \X \times [0,1] \to \R$ bounded from below
    \begin{align}\label{eq:st.transf1}
        \int_{[0,1]} \E_{\P^{\fp X}}[ \varphi(X,\sigma_u) ]  \,du =  \E_{\pi} [ \varphi(X,\tau+\epsilon)].
    \end{align}
    In particular, there exists a stopping time $\sigma \in \rm{ST}(\fp Y)$ such that 
    \[
    \E_{\P^{\fp X}}[ \varphi(X,\sigma) ]  \le   \E_{\pi} [ \varphi(X,\tau+\epsilon)].
    \]
    \end{proposition}
    \begin{proof}
        For every $u \in [0,1]$, define
		\begin{align}\label{eq:stoptime}
			\sigma_u:= \inf \left\{t \geq 0: \, \pi (\tau \leq t-\epsilon \, | \, \mathcal{F}^{\bbX}_{1} ) \geq u \right\} \wedge 1. 
		\end{align}
		By the definition of $\epsilon$-causality, the conditional probability $\pi (\tau^* \leq t-\epsilon \, | \, \mathcal{F}^{\bbX}_{1 }) $ is $\mathcal{F}^{\bbX}_t$-measurable. 
		As the hitting time of a right-continuous adapted process, $\sigma_u$ is a stopping time of $\bbX$ and the infimum over $t$ in \eqref{eq:stoptime} is attained.
  
		By conditioning on $\F_1^{\fp X}$ and Fubini's theorem, we find
		\begin{align*}
			\int_{[0,1]}  \E_{\pi}[ \varphi(X,\sigma_u)] \, du
			&= \int_{[0,1]} \E_{\pi} \left[ \E_{\pi} \left[ \varphi(X,\sigma_u) \, | \, \mathcal{F}_{1}^{\bbX}\right]\right] \,du 
			= \E_{\pi} \left[\int_{[0,1]} \E_{\pi} \left[ \varphi(X,\sigma_u) \, | \, \mathcal{F}_{1}^{\bbX}\right] \, du \right] 
		\end{align*}
        As $u \mapsto \sigma_u(\omega^{\fp X})$ is the quantile function of $\mathcal{L}_{\pi}(\tau+\epsilon  \, | \,  \mathcal{F}_{1}^{\bbX})(\omega^{\fp X})$ for a.e.\ $\omega^{\fp X} \in \Omega^{\fp X}$, we find
            
		\begin{align*}
			\int_{[0,1]} \E_{\pi} [ \varphi(X,\sigma_u) \, | \, \mathcal{F}_{1}^{\bbX} ]  \, du
			= \int_{[0,1]} \varphi(X,\sigma_u) \, du 
			= \E_{\pi} [ \varphi(X,\tau+\epsilon) \, | \, \mathcal{F}_{1}^{\bbX}].
		\end{align*}
  Therefore, we conclude \eqref{eq:st.transf1} by taking expectations. 
    \end{proof}
    
    \begin{corollary}\label{cor:OSwelldef}
    Let $\varphi$ be non-anticipative and bounded from below.   If $\fp X,\fp Y\in\mathcal{FP}$ satisfy $\fp X\sim_{\rm HK} \fp Y$, then
		 ${\rm OS}(\fp X,\varphi)={\rm OS}(\fp Y,\varphi)$.
    \end{corollary}
    \begin{proof}
    By Lemma~\ref{lem:canonRep}, it suffices to show that $\rm{OS}(\fp X, \varphi) = \rm{OS}(\overline{\fp X}, \varphi)$, where $\overline{\fp X} \in \mathcal{CFP}$ is the canonical representative associated to $\fp X$. By Lemma~\ref{lem:canonical.couplings}\ref{it:lem.canonical.couplings.1} there is a coupling $\pi \in \cplbc(\fp X, \overline{\fp X})$ such that $X = \overline{X}$ $\pi$-a.s. Therefore, by Proposition~\ref{prop:ST.transfer}, for any $\tau \in \rm{ST}(\overline{\fp X })$ there is $\sigma \in \rm{ST}(\fp X)$ such that
    $
    \E_{\P^{\fp X}}[ \varphi(X, \sigma) ] \le \E_\pi[ \varphi(X, \tau)] = \E_{\P^{\overline{\fp X}}}[ \varphi(\overline{X}, \tau) ]. 
    $
    By reversing the roles of $\fp X$ and $\overline{\fp X}$ in the above argument, for any $\tau \in \rm{ST}(\fp X)$ there is $\sigma \in \rm{ST}(\overline{\fp X})$ such that
    $    
    \E_{\P^{\overline{\fp X}}}[ \varphi(\overline{X}, \sigma) ] \le  \E_{\P^{\fp X}}[ \varphi(X, \tau) ]. 
    $
    Hence, $\rm{OS}(\fp X, \varphi) = \rm{OS}(\overline{\fp X}, \varphi)$.
    \end{proof}
   
    \begin{lemma}\label{lem:OSerr0}
    Let $\varphi$ be a bounded cost function that satisfies Assumption~\ref{asn:costassump} and let $(\pi_n)_n$ with $\pi_n \in \cpl(\fp X^n, \fp X)$ such that $\E_{\pi_n}[d_{\X}(X^n,X) \wedge 1 ] \to 0$. We have, for every null sequence $(\varepsilon_n)_n$,
    \begin{align}\label{eq:lem:phierr}
        \lim_{n \to \infty}\E_{\pi_n}\bigg[\sup_{ |s-t| \le \epsilon_n} | \varphi(X^n,s) - \varphi(X,t) | \bigg] = 0.
    \end{align}
    \end{lemma}
    \begin{proof}
    By conditional independent gluing of the couplings $\pi_n$, $n\in\N$, at their common marginal $\P^{\fp X}$, we find random variables $\tilde{X}^n, \tilde{X}$ on a common probability space such that $\law(\tilde{X}^n, \tilde{X})= \law(X^n,X)$ and hence $\tilde{X}^n \to \tilde{X}$ in probability.  After possibly passing to a subsequence, we have  $\tilde{X}^n \to \tilde{X}$ almost surely. 

    If $d_{J_1}(f_n,f) \to 0$ and $f \in C([0,1];\R^d)$, we have $\|f_n-f\|_\infty\to 0$ as well. Using this, Assumption~\ref{asn:costassump} and the compactness of $[0,1]$, we find that a.s.
    \[
    \lim_{n \to \infty} \sup_{ |s-t| \le \epsilon_n} | \varphi(\tilde{X}^n,s) - \varphi(\tilde{X},t) | = 0.
    \]
    Hence, \eqref{eq:lem:phierr} follows using dominated convergence.
    \end{proof}

    We are now ready to prove the qualitative continuity result. 
    \begin{proof}[Proof of Proposition~\ref{prop:OST.qual}]
    Suppose that $\fp X^n \to \fp X$ in the adapted weak topology and that $X$ has continuous paths. By Theorem~\ref{thm:AWmetrizesHK}, there is a sequence $(\pi_n)_n$ of $\epsilon_n$-bicausal couplings such that $\epsilon_n \to 0$ and $\E_\pi[d_\X(X^n,X) \wedge 1 ] \to 0$. 
    
    We start by showing that ${\rm OS}( \bbX,\varphi) \leq \liminf_{n\to\infty} {\rm OS}( \bbX^n,\varphi)$.
    To that end let $\tau_n$  be an optimal stopping time for ${\rm{OS}}({\fp{X}}^{n}, \varphi)$ (or, almost optimal if it does not exists).
    By Proposition~\ref{prop:ST.transfer}, we have ${\rm{OS}}(\fp X, \varphi) \leq \E_{\pi_n}[ \varphi(X,\tau_n+\varepsilon)]$, hence
    \begin{align*}
    {\rm{OS}}(\fp X, \varphi) - {\rm{OS}}(\fp X^n ,\varphi)  
    &\le \E_{\pi_n}[ \varphi(X,\tau_n) - \varphi(X^n, \tau_n +\epsilon_n) ] \\
    &\le \E_{\pi_n} \!\bigg[ \sup_{t \in [0,1]}\! | \varphi(X,t) - \varphi(X^n,t+\epsilon_n) |  \bigg]
    \to 0,
    \end{align*}
    where the last step follows from Lemma~\ref{lem:OSerr0}.
    
    The same arguments show that 
    \[
    {\rm{OS}}(\fp X^n, \varphi) - {\rm{OS}}(\fp X ,\varphi)  
    \le \E_{\pi_n} \!\bigg[ \sup_{t \in [0,1]}\! | \varphi(X^n,t) - \varphi(X,t+\epsilon_n) |  \bigg]
    \to 0  \]
    and hence ${\rm OS}( \bbX,\varphi) \geq \limsup_{n\to\infty} {\rm OS}( \bbX^n,\varphi)$.
    This completes the proof.
    \end{proof}

    For Lipschitz cost functions, more refined quantitative  estimates are possible.
    To that end, we require the following concept.
	
	\begin{definition}
		For a filtered process $\bbX\in \FP$,  we call 
		\[ \delta_\bbX(\varepsilon) := \sup_{\tau\in \rm{ST}(\bbX)} \E\bigg[\sup_{s\in[0,\varepsilon]} |X_{\tau+s} - X_\tau| \bigg]
		\qquad\text{for } \varepsilon\in[0,1]\]
		the \emph{modulus of continuity} of $\bbX$.
	\end{definition}
    
    Using the same arguments as in the proof of Corollary~\ref{cor:OSwelldef}, one can see that the modulus of continuity is independent of the $\sim_{\HK}$-representative. It is worthwhile to note that the modulus of continuity can be easily estimated for certain standard processes (their standard definitions can be found in Section~\ref{sec:donsker}).
	
	\begin{example}	
		\label{ex:modulus}
		The following hold.
		\begin{enumerate}[(i)]
			\item
			If $\fp B$ is the standard Brownian motion, then $\delta_{\mathbb{B}}(\varepsilon)\leq 2 \sqrt{\varepsilon}$.
			\item\label{ex:modulus.item.2}
			If $\fp S$  is the  solution of the SDE $dS_t = \mu_t(S_t) \, dt + \sigma_t( S_t)\, d B_t$, then $\delta_{\mathbb{S}}(\varepsilon)\leq  \|\mu\|_\infty \varepsilon + 2  \|\sigma\|_\infty \sqrt{\varepsilon}$.
			\item
			If $\fp B^n$ is the scaled  random walk with step size $1/n>0$, then $\delta_{\mathbb{B}^n}(\varepsilon) \leq 4 \sqrt{ 1/n \vee \epsilon}$.
			\item
			If $\fp S^n$ is the Euler approximation  for $\fp S$, then $\delta_{\mathbb{S}^n}(\varepsilon) \leq 2 \lVert \mu \rVert_{\infty} ( 1/n \vee \varepsilon )+ 4 \lVert \sigma \rVert_{\infty} \sqrt{1/n \vee \varepsilon} $ where $1/n>0$ is the step size.
		\end{enumerate}
	\end{example}
	
	The proof of Example~\ref{ex:modulus} is standard and given in Appendix~\ref{sec:appQuantOS}.
	Here is our main quantitative continuity result for the value of optimal stopping problems.

	\begin{proposition}[Quantitative continuity of OS]
		\label{prop:OST.quant}
		Assume that the cost $\varphi$ is non-anticipative and $L$-Lipschitz, that is, 
 for every $f,g\in D([0,1];\mathbb{R}^d)$ and $0\leq s\leq t\leq 1$, 
			\[ |\varphi(f,t)-\varphi(g,s)|\leq L \left( \sup_{u\in[0,s]} |f(s)-g(s)| + \sup_{u\in[s,t]} |f(u)-g(s)|\right).\]
    Moreover, let $\bbX,\bbY \in\FP$ and let $\pi\in \cpl^{\varepsilon}_{\rm{bc}}(\bbX,\bbY)$.
		Then we have 
		\[ |{\rm OS}(\fp X,\varphi) - {\rm OS}(\fp Y, \varphi)|
		\leq  L \left( \E_{\pi}[ \|X-Y\|_\infty] + \delta_{\bbX}(\varepsilon) + \delta_{\bbY}(\varepsilon) \right). \]
	\end{proposition}
	
A particular simple and yet relevant example of a class of functions $\varphi$ that satisfy the assumption in Proposition \ref{prop:OST.quant} are those depending only on the current state, 
      i.e., $\varphi(f,t)=\psi(f(t))$ where $\psi\colon\R^d\to\R$ is Lipschitz. A further example are cost functions depending on the running maximum, e.g.\ $\varphi(f,t)=\psi( \sup_{s \in [0,t] } f(s) )$ where $\psi\colon\R^d\to\R$ is Lipschitz.

	\begin{proof}[Proof of Proposition \ref{prop:OST.quant}] For notational simplicity,  assume that there is an optimal stopping time $\tau \in \rm{ST}(\fp Y)$, i.e.\  $\rm{OS}(\fp Y,\varphi) = \E[\varphi(Y,\tau)]$ (else take an almost optimal one).
    By Proposition~\ref{prop:ST.transfer} and the assumptions on $\varphi$ we find
		\begin{align*}
            \rm{OS}(\fp X, \varphi) - \rm{OS}(\fp Y, \varphi) &\le 
			\E_{\pi}[ \varphi(X,\tau+\epsilon)-\varphi(Y,\tau)] \\
			&\leq \E_{\pi}[ | \varphi(X,\tau+\epsilon)-\varphi(Y,\tau+\epsilon)| ]  + \E_\pi[ |\varphi(Y,\tau+\epsilon) - \varphi(Y,\tau)) |] \\
			&\leq L  \E_{\pi}\left[ |X_{\tau+\epsilon} - Y_{\tau+\epsilon}| \right] + L \E_\pi\bigg[ \sup_{u\in[0,\epsilon]} | Y_{ \tau+u} - Y_{\tau}| \bigg] \\
			&\leq L \E_\pi[ \|X-Y\|_\infty] + L \delta_{\bbY}(\epsilon).
		\end{align*}
		Reversing the roles of $\bbX$ and $\bbY$ completes the proof.
	\end{proof}

        Note that Example~\ref{ex:counter} shows that the assumption that the limit process has continuous paths is necessary in Proposition~\ref{prop:OST.qual}. This example naturally raises the question of which continuity assumptions on $\varphi$ and $(\fp X^n)_n, \fp X$ are necessary to ensure that ${\rm OS}(\fp X^n, \varphi) \to {\rm OS}(\fp X, \varphi)$. Addressing this question lies beyond the scope of this paper and is deferred to future research.

We end this section with a `continuity' result of a different flavor:

 \begin{proposition}
    \label{lem:mart_closed}
   For every $p\geq 1$, the set of martingales is closed in $(\FP_p,\AW_p)$. 
\end{proposition} 

\begin{proof}
    Assume first that $\mathcal X = C([0,1])$.
    Let $(\fp X^n)_n$ be a sequence of martingales that converges to $\fp X$  w.r.t.\ $\AW_p$.
    We have to show that $\fp X$ is a martingale, i.e.,  that $ \E [ | \E[X_1 - X_t | \F_t^\fp X] | ] = 0 $ for all $t\in[0,1]$.    
    To this end, fix $t \in [0,1]$.
    For each $n$, let $\varepsilon_n \ge 0$ and $\pi^n \in \cplbc^{\varepsilon_n}(\fp X^n, \fp X)$ be optimal for $\AW_p(\fp X^n, \fp X)$.
    Let $\varepsilon > 0$ and let $n$ be sufficiently large such that $\varepsilon \ge \varepsilon_n$.
    Invoking (the conditional version of) Jensen's inequality, one may show that 
    \[
        \E \Big[ \big| \E[X_1 - X_t | \F_t^\fp X] \big| \Big] \le
        \E_{\pi^n} \Big[ \big| \E_{\pi^n}[X_1 - X_t | \F_{t + \varepsilon, t}^{\fp X^n,\fp X}] \big| \Big].
    \]
    Thus, by the triangle inequality
    \begin{align} \label{eq:lem.mart_closed.0'}
        \begin{split}
      & \E \Big[ \big| \E[X_1 - X_t | \F_t^\fp X] \big| \Big] \\
       & \le
       \E_{\pi^n} \Big[ \big| \E_{\pi^n}[X_1 | \F_{t + \varepsilon, t}^{\fp X^n,\fp X}] - X_{t + \varepsilon}^n \big| \Big] + 
       \E_{\pi^n} \Big[ \big| X_{t + \varepsilon}^n - X_{t + \varepsilon}  \big| \Big] + 
      \E \Big[ \big| X_{t + \varepsilon} - X_t \big| \Big]\\
      &=:  I_1^n + I_2^n + I_3.
      \end{split}
    \end{align}
    We  estimate each term separately.
    Clearly $I_2^n \le \AW_p(\fp X^n, \fp X) \to 0$.
    Moreover $I_3\to 0$ when  $\varepsilon \searrow 0$ by (right) continuity of the paths of $X$ and the dominated convergence theorem.
    In order to control $I_1^n$, observe that
    \begin{align} \label{eq:lem.mart_closed.1}
        X_{t + \varepsilon}^n = \E[X_1^n | \F_{t + \varepsilon}^{\fp X^n}] = \E_{\pi^n}\big[ X_1^n | \F_{t + \varepsilon, t}^{\fp X^n,\fp X}  \big],
    \end{align}
    where the first equality is due to $\fp X^n$ being a martingale and the second follows from $\varepsilon_n$-bicausality of $\pi^n$.
    Plugging \eqref{eq:lem.mart_closed.1} into $I_1^n$, we obtain
    \[
        I_1^n 
            = 
        \E_{\pi^n} \Big[ \big| \E_{\pi^n} [ X_1 - X_1^n | \F_{t + \varepsilon, t}^{\fp X^n, \fp X}] \big| \Big] 
            \le
        \E_{\pi^n}\Big[ \big| X_1 - X_1^n \big| \Big] \le \AW_p(\fp X^n, \fp X),
    \]
    where the first inequality is again due to the conditional Jensen's inequality.
    Hence, also $I_1^n \to 0$ for $n \to \infty$, which completes the proof as $\varepsilon > 0$ was arbitrary.

    The proof when  $\mathcal{X} = D([0,1])$  follows from the same arguments as presented above, with the additional step of replacing  $X_t $ by approximations  $\frac{1}{\delta}\int_t^{t+\delta} X_s \,ds$  to ensure continuity (since point evaluations  $f\mapsto f(t)$ are not continuous on $D([0,1])$.
\end{proof}

\section{Processes with {natural} filtration: basic properties}\label{sec:plain} 
As explained in the introduction, the set $\mathcal{P}(\mathcal{X})$ corresponds to processes with {natural} filtration. We start with a precise definition of the class of processes which are Hoover--Keisler-equivalent to processes with natural (right-continuous, completed) filtrations.
\begin{defi}
\label{def:plain}
    Let $\mu \in \prob(\X)$,  let $X$ be the identity on $\mathcal{X}$ and set $\mathcal F_t^\mu := \bigcap_{\varepsilon>0 } \sigma^\mu( X_s: s\leq t+\varepsilon)$ for $t\in[0,1]$.
	We call 
	\[
	    \fp S^\mu := \left( \X, \B^\mu_\X, \mu,  (\mathcal F^\mu_t)_{t \in [0,1]}, X \right)
	\]
	the \emph{standard naturally filtered process} with law $\mu$.
  A process  $\fp X \in \mathcal{FP}$ is called \emph{naturally filtered} if $\fp X \sim_{\rm HK} \fp S^{\mu}$, where $\mu=\mathcal{L}(X)$. 
  The collection of all naturally filtered processes is denoted by $\cSFP$.
\end{defi}

By definition, being naturally filtered is a property that is independent of the representative of the Hoover--Keisler-equivalence class, hence the set $\NFP:=\cSFP/_{\sim\HK}$ is a well-defined subset of $\FP$.

The aim of this section is to prove fundamental properties of this type of processes, e.g.\ that they are dense among all processes. To this end, we start with the following useful characterization of naturally filtered processes. 

\begin{lemma}\label{lem:plaineq}
Let $\fp Y \in \mathcal{FP}$ and set $\mu := \law(Y)$. Then $(\id,Y)_\#\P^{\fp Y} \in \cplc(\fp Y, \fp S^\mu)$. Further set 
\[\G_t := \bigcap_{\epsilon >0} \sigma^{\P^{\fp Y}}(Y_s : s \le t+\epsilon), \, \G^1_t := \bigcap_{\epsilon >0} \sigma^{\P^{\fp Y}}(\pp^1_s(\fp Y) : s \le t+\epsilon), \,\G^\infty_t := \bigcap_{\epsilon >0} \sigma^{\P^{\fp Y}}(\pp^\infty_s(\fp Y) : s \le t+\epsilon).\] Then 
the following are equivalent:
\begin{enumerate}[label = (\roman*)]
    \item\label{it:lem:plaineq1} $\fp Y$ is naturally filtered.
    \item\label{it:lem:plaineq2} $\G^1_t  \subseteq \G_t$ for every $t \in [0,1]$.
    \item\label{it:lem:plaineq3} $\G^\infty_t  \subseteq \G_t$ for every $t \in [0,1]$.
    \item\label{it:lem:plaineq4} $(\id,Y)_\#\P^{\fp Y} \in \cplbc(\fp Y, \fp S^\mu)$.
\end{enumerate}
\end{lemma}

Before we start the proof, note that the inclusion $\G_t \subset \G^1_t \subset \G_t^\infty$ is true for all $t \in [0,1]$ and all filterd processes $\fp Y \in \mathcal{FP}$, cf.\ \cite[Lemma~4.10]{BePaScZh23}. Therefore, the conditions (ii) and (iii) are equivalent to $\G_t = \G^1_t$ and $\G_t = \G^\infty_t$ respectively.

\begin{proof}
We start by proving that $\pi := {(\id,Y)_\#} \P^{\fp Y}$ is causal. Note that  $X=Y$ $\pi$-a.s.
In order to show causality, we need to check that $\E_\pi[ U | \F_1^{\fp Y}  ] =  \E_\pi[ U | \F_t^{\fp Y}  ]$ for every $U \in L^\infty(\F^{\fp S^\mu}_t)$, see Lemma~\ref{lem:causal_equiv}. By the definition of $\F^{\fp S^\mu}_t$, we find that $U=f(X_{[0,t+\epsilon]})$ for some Borel function $f$. As $X=Y$ $\pi$-a.s., this implies that $U$ is $\F_{t+\epsilon}^{\fp Y}$-measurable; and as $\epsilon>0$ is arbitrary and $(\F_t^{\fp Y})_t$ is right-continuous, we find that $U$ is $\F_{t}^{\fp Y}$-measurable. Hence, $\E_\pi[ U | \F_1^{\fp Y}  ] =  \E_\pi[ U | \F_t^{\fp Y}  ]$.

(i) $\implies$ (ii): Note that (ii) is equivalent to $\pp^1_{[0,s]}(\fp Y)$ being $\sigma(Y_{[0,t]})$-measurable for all $s<t$, which is equivalent to  $\law(Y_{[0,t]}, \pp^1_{[0,s]}(\fp Y))$ being concentrated on the graph of a function for all $s<t$. The latter is a property of $\law(\pp^\infty(\fp Y))$ and hence independent of the representative of the $\sim_{\HK}$-class. Therefore, it suffices to show (ii)
for the standard naturally filtered process $\fp S^\mu = \left( \X, \B_\X, \mu, (\mathcal F^\mu_t)_t, X \right)\sim_{\HK} \fp Y$. As $\pp^1(\fp S^\mu)$ is adapted to $\F=\G$, we conclude $\G_t^1 \subset \G_t$ for every $t \in [0,1]$.   

(ii) $\implies$ (iii): 
We show by induction on the rank $r$ that there is a Borel function $f_r: \mathcal X \to \mathsf{M}_r$ for which $\pp^{r}(\fp Y) = {f_r}(Y)$. For the case $r=1$ it suffices to observe that (ii) immediately implies that there is a  Borel function $f_1 : \mathcal X \to \mathsf{M}_1$ such that $\pp^1(\fp Y) = f_1(Y)$. 

Assuming that the claim is true for $r$, observe that for every $t \in [0,1]$,
\begin{align}\label{eq:prf:plaineq:ind}
\pp^{r+1}_t(\fp Y) = \law(\pp^r(\fp X) | \F^{\fp Y}_t) = \law(f_r(Y)|\F_t^\fp Y) =  {f_r}_\# \law(Y|\F_t^{\fp Y}) = {f_r}_\# \pp^1_t(\fp Y) .  
\end{align}
Hence, the function $f_{r+1} : \mathcal X \to \mathsf{M}_{r+1} : (z_t)_t \mapsto ({f_r}_\# z_t)_{t \in [0,1]}$ satisfies $ \pp^{r+1}(\fp Y)=  f_{r+1}(Y)$, which  completes the induction. 

By \eqref{eq:prf:plaineq:ind} we have $\pp^\infty_t(\fp Y) = (\id, {f_1}_\#, {f_2}_\#, \dots)(\pp_t^1(\fp Y))$ for every $t \in [0,1]$; hence, $\G^\infty \subset \G^1$.

(iii) $\implies$ (iv): As we have already shown that 
$\pi= (\id,Y)_\#\P^{\fp Y}$ is causal from $\fp Y$ to $\fp S^\mu$, it remains to show causality from $\fp S^\mu$ to $\fp Y$.  By \cite[Corollary~4.10]{BePaScZh23} we have for every $t \in [0,1]$ under $\P^{\fp Y}$ and hence under $\pi$,
\begin{align}\label{eq:prf:plaineq2}
    \G_1^\infty \indep_{\G_t^\infty} \F_t^{\fp Y}.
\end{align}

    As $X=Y$ $\pi$-a.s., we have that the $\pi$-completions of $\F_t^{\fp S^{\mu}}$ and $\G_t$ coincide for every $t \in [0,1]$.
    Therefore, (iii) implies that also the $\pi$-completions of $\F_t^{\fp S^{\mu}}$ and $\G_t^\infty$ coincide for every $t \in [0,1]$.

Thus, \eqref{eq:prf:plaineq2} reads as 
\[
\F_1^{\fp S^\mu} \indep_{ \F_t^{\fp S^\mu} } \F_t^{\fp Y},
\]
which is precisely causality from $\fp S^\mu$ to $\fp Y$. 

(iv) $\implies$ (i): As $X=Y$ a.s.\ under the bicausal coupling $\pi:= (\id,Y)_\#\P^{\fp Y}$, we have $\fp Y \sim_{\HK} \fp S^\mu$.
\end{proof}

The following characterization of naturally filtered processes is  a simple consequence of Lemma~\ref{lem:plaineq}.

\begin{corollary}
    \label{cor:plain.via.ordering} 
    Let $\fp X \in \FP_p$ and set $\mu:=\law(X)$.
    Then $\fp X$ is naturally filtered if and only if every $\fp Y \in \FP$ with $\law(Y) = \mu$ satisfies $\CW_p(\fp Y, \fp X) = 0$.
\end{corollary}
\begin{proof}
Suppose that $\fp X$ is naturally filtered and let $\fp Y$ be another process that satisfies $\law(Y)=\mu$. By Lemma~\ref{lem:plaineq}, $\pi:=(\id,Y)_\# \P^{\fp Y} \in \cplc(\fp Y, \fp S^\mu)$, and therefore $\CW_p(\fp Y, \fp S^\mu)=0$. As $\fp X \sim_{\HK} \fp S^\mu$, we conclude $\CW_p(\fp Y, \fp X) =  \CW_p(\fp Y, \fp S^\mu )=0$. 

For the reverse implication suppose that $\CW_p( \fp Y, \fp X)=0$ for every $\fp Y \in \FP_p$ satisfying $\law(Y)=\mu$.  In particular, we have $\CW_p(\fp S^\mu, \fp X)=0$. As $(\id,X)_\#\P^{\fp X} \in \cplc(\fp X, \fp S^\mu)$, we have $\CW_p(\fp X, \fp S^\mu)=0$ as well, hence $\fp X \sim_{\HK} \fp S^\mu$ by Lemma~\ref{lem:scw.definite}.
\end{proof}

Lemma~\ref{lem:plaineq} allows us to simplify the characterization of continuity points given in Proposition~\ref{prop:contPt} in the case of naturally filtered processes.

\begin{corollary}\label{cor:contPtPlain}
Let $\fp X \in \mathcal{FP}$ be a naturally filtered process and set $\G_s := \bigcap_{\epsilon>0} \sigma^{\P^{\fp X}}( X_r : r \le s+\epsilon) $. For $t \in [0,1]$ the following are equivalent:
	\begin{enumerate}[label = (\roman*)]
        \item $t \in \cont(\fp X)$.		
        \item The filtration $(\G_s)_{s \in [0,1]}$ is continuous at $t$, that is $\G_t=\G_{t-}$.
	\item The paths of $\pp^1(\fp X)$ are a.s.\ continuous at $t$.
	\end{enumerate}  
\end{corollary}
\begin{proof}
(i) $\iff$ (ii) : Set $\G^\infty_s := \bigcap_{\epsilon >0} \sigma(\pp^\infty_r´(\fp Y) : r \le s+\epsilon)$. 
By Lemma~\ref{lem:plaineq} we have $\G=\G^\infty$, and by Proposition~\ref{prop:contPt}, assertion (i) is equivalent to continuity of $\G^\infty$ at $t$.
Therefore, (i) and (ii) are equivalent.

(ii) $\implies$ (iii) : This follows immediately from a suitable version of the martingale convergence theorem for measure-valued martingales, see e.g.\ \cite[Lemma~4.16]{BePaScZh23}.

(iii) $\implies$ (ii) : It clearly suffices to  show that every $U \in L^\infty(\G_t)$ is $\G_{t-}$-measurable. As $\G_t \subset \sigma(X)$, there is a Borel function $f$ such that $U=f(X)$. By (iii) we have that $\law(X| \G_{t-} ) = \law(X | \G_t )$ a.s., and thus,
\[
U = \E[U | \G_t ] = \E[f(X) | \G_t ] = f_\# \law( X | \G_t) = f_\# \law( X | \G_{t-}) = \E[f(X)| \G_{t-}].
\]
Hence, $U$ is $\G_{t-}$-measurable as claimed. 
\end{proof}

	Next, we show that $\NFP$  is a Polish topological space.
 It is  important to emphasize that the significance of this statement lies more in its theoretical implications, as there does not seem to be a natural complete metric for $(\NFP,\HK)$.

	\begin{proposition}
	\label{prop:Plain.is.G.delta}
	$\NFP $ is a $G_{\delta}$ subset of $\FP$.
	In particular, $\NFP $  is a Polish space.
	\end{proposition}
\begin{proof}[Sketch of the proof]
For $s<t$, let $h_{s,t} : [0,1] \to [0,1]$ be a continuous function such that $h_{s,t}=1$ on $[0,s]$ and $h_{s,t}=0$ on $[t,1]$. Using Lemma~\ref{lem:plaineq}\ref{it:lem:plaineq3} one may show that $\NFP $ is the intersection of the sets 
\[
 G_{s,t} := \{  \fp X \in \FP :  \law( (h_{s,t}(u)X_u)_u, \pp^1_{[0,s]}(\fp X)) \text{ is supported on the graph of a Borel function} \},  
\]
where the intersection is taken over all $s,t \in \Q \cap (0,1]$ satisfying $s<t$. Next, recall that if $A, B$ are Lusin spaces, then the set of $\mu \in \prob(A \times B)$ that are concentrated on a graph of a Borel map from $A$ to $B$ is $G_\delta$ in $\prob(A \times B)$.\footnote{This result is well-known if $A,B$ are Polish, see e.g.\ \cite{Ed19}. For the definition of Lusin spaces and the extension of this result to Lusin space see \cite[Remark~2.17]{BePaScZh23}. In our proof, we take $A=\mathcal X$ endowed with the $J_1$-topology and $B= D([0,s];\prob(\mathcal X))$ endowed with convergence in measure w.r.t.\ the Lebesgue measure on $[0,s]$.} Since $\fp X \mapsto \law( (h_{s,t}(u)X_u)_u , \pp^1_{[0,s]}(\fp X)) $ is continuous, $G_{s,t}$ is $G_\delta$ as continuous pre-image of a $G_\delta$-set. Hence,  $\NFP $ is $G_\delta$ as countable intersection of $G_\delta$-sets.
\end{proof}

The main result of this section is the following
\begin{proposition}
\label{prop:Plain.Dense}
    $\NFP_p$ is a dense subset of $(\FP_p, \AW_p)$.
\end{proposition}

The proof of Proposition~\ref{prop:Plain.Dense} follows from relating the present continuous-time setting to a discrete-time setting, and using  the analogous result to Proposition~\ref{prop:Plain.Dense} that has already been established in discrete time (see \cite[Theorem~5.4]{BaBePa21}).

\begin{proof}
The proof is a slight modification of the construction given in the proof of Proposition~\ref{prop:FiniteOmegaDense}, and we use the notation therein. Regarding step 2 of that proof, we notice that the discrete-time result \cite[Theorem 5.4]{BaBePa21} used there not only guarantees denseness of processes defined on a finite $\Omega$, but also denseness of naturally filtered\footnote{Note that naturally filtered processes are called plain in \cite{BaBePa21}.} processes defined on a finite $\Omega$. Therefore, it  remains to argue that the construction given in step 3 of the proof of Proposition~\ref{prop:FiniteOmegaDense} preserves the property of being naturally filtered. Specifically, we need to show that if the discrete-time process $\fp Z = (\Omega^{\fp Z }, \F^{\fp Z}, \P^{\fp Z}, (\F_{t_i}^{\fp Z})_{i=0}^N, (Z_{t_i})_{i=0}^N)$ is naturally filtered, then the continuous-time process
\[
\fp Y := (\Omega^{\fp Z }, \F^{\fp Z}, \P^{\fp Z}, (\F_{t_{ \lceil t \rceil_T }   }^{\fp Z})_{t \in [0,1]}, \iota_T(Z))
\]
is naturally filtered as well. By Lemma~\ref{lem:plaineq}\ref{it:lem:plaineq2},  suffices to show that for $0<s<t\le 1$, $\pp_s(\fp Y)$ is a Borel function of $\iota_T(Z)|_{[0,t]}$. To that end, fix $i$ such that $\lceil s \rceil_T = t_i$. As $\fp Z$ is naturally filtered, there is a Borel function $g$ such that $\pp_{t_i}(\fp Z) =\law(Z |\F_{t_i}^{\fp Z}) = g(Z_{t_0},\dots, Z_{t_i})$. As $t >t_i$, there is a Borel function $h$ such that $(Z_{t_0},\dots, Z_{t_i}) = h( \iota_T(Z)|_{[0,t]})$. Hence,
\[
\pp_s(\fp Y) = \law(\iota_T(Z) | \F^{\fp Z}_{ \lceil s \rceil_T }) = {\iota_T}_\# \law(Z |\F_{t_i}^{\fp Z}) = {\iota_T}_\# g(Z_{t_0},\dots, Z_{t_i}) = {\iota_T}_\# g(h(  \iota_T(Z)|_{[0,t]})). \qedhere
\]
\end{proof}

\section{Adapted topologies on processes with {natural} filtration} \label{sec:adapted topologies on plain}

The goal of this section is to prove Theorem \ref{thm:intro.all.topo.equal.bounded}. 
We start by recalling the topologies therein and rigorously define our continuous-time version of Hellwig's information topology. 
\begin{itemize}
    \item \emph{The optimal stopping topology} is the coarsest topology for which all maps of the form
		\[
		\fp{X} \mapsto {\rm OS}(\fp X, \varphi) 
		:= 
		\inf_{\tau \in \rm{ST}(\fp X)} 
		\mathbb E_{\P^{\fp X}}[\varphi(X,\tau)],
		\]
		where $\varphi \colon \mathcal
        X \times [0,1] \to\mathbb{R}$ is continuous, bounded, and non-anticipative, are continuous.
\item 
The \emph{Hoover--Keisler topology} was defined in Section \ref{sec:Recap}: it is the coarsest topology for which  $\fp X\mapsto \law( {\rm pp}^\infty (\fp X) )$ is continuous. 
\item 
The \emph{adapted weak topology} was introduced in Section~\ref{sec:Recap} via $\fp X^n \to \fp X$ if and only if $\E[f[\fp X^n]] \to \E[f[\fp X]]$ for every adapted function $f$ satisfying certain regularity conditions regarding the continuity points of $\fp X$. 

\item The 
\emph{Aldous\textsubscript{J1} topology} is the initial topology of 
\begin{align}\label{eq:def:Ald_mainbody}
\fp X \mapsto \law_{\P^{\fp X}}( (X_t,\pp^1_t(\fp X))_{t\in[0,1]}) \in \mathcal{P}\big(D([0,1];\R \times \mathcal{P}(\X))  \big),    
\end{align}
where $D([0,1];\R \times \mathcal{P}(\X))$ is equipped with Skorohod's $J_1$-topology.
\item The
 \emph{Aldous\textsubscript{MZ} topology} is the initial  topology  of \eqref{eq:def:Ald_mainbody}, 
where $D([0,1];\mathcal{P}(\X))$ is equipped with the Meyer--Zheng topology instead of the $J_1$-topology.\footnote{We note that Aldous\textsubscript{MZ} can equivalently be defined as initial topology w.r.t.\ the map $\fp X \mapsto \law_{\P^{\fp X}}( (\pp^1_t(\fp X))_t) \in \mathcal{P}(D([0,1]; \mathcal{P}(\X)))$, i.e.\ by just considering the paths of $\pp^1(\fp X)$ and not of $(X,\pp^1(\fp X))$. In particular, Aldous\textsubscript{MZ} is the rank-1 Hoover--Keisler topology mentioned in Section~\ref{sec:Recap}. While these two conventions make no difference for the definition of Aldous\textsubscript{MZ}, they do make a difference in Aldous\textsubscript{J1}; see Example~\ref{ex:counter}.  }

\item 	
\emph{Hellwig's information topology} is originally formulated in discrete time.
To explain his original definition, let us consider a finite set $T\subset[0,1]$.
For any fixed time point $t\in T$, Hellwig partitions the time index set $T$ into the past and the future and explores the initial topology in which all mappings,
\[  
\fp X \mapsto \law_{\mathbb{P}^{\fp X}}\left( (X_s)_{s\in T\cap[0,t]} , \law_{\mathbb{P}^{\fp X}} ((X_s)_{s\in T\cap (t,1]} | \mathcal{F}^\fp X_t) \right),\quad t\in T
\]
are continuous.
It is straightforward to verify that latter topology is the same as the initial topology of $\fp X\mapsto \law_{\mathbb{P}^{\fp X}}(  {\rm pp}_t^1(\fp X))$ for $t\in T$.
While a verbatim extension to continuous time presents certain obstacles (due the uncountable number of $t$’s), it is reasonably straightforward to come up with a suitable modification (using a similar approach as when defining weak convergence of distribution functions by pointwise convergence at \emph{continuity} points):

We say that a sequence $(\fp X^n)_{n \in \N}$ in $\FP$ converges in \emph{Hellwig's information topology} to  $\fp X\in \FP$ if 
		for every $t\in\cont(\fp X)\cup\{1\}$, $\law(\pp_t^1(\fp X^n))$ converges weakly to $\law(\pp_t^1(\fp X)) )$.
\end{itemize}

\begin{remark}
It follows from \cite[Theorem~3.10]{BePaScZh23} that  Hellwig's information topology is metrized by
	\begin{align*}
     d_{\rm Hellwig}(\fp X,\fp Y) 
	   &=\int_0^1 d_w( \law (\pp_t^1(\fp X) ), \law ( \pp_t^1(\fp Y)) )  \, dt +  d_w( \law (\pp_1^1(\fp X)),\law ( \pp_1^1(\fp Y)) ),
	\end{align*}
 where $d_w$ is a bounded metric compatible with weak convergence. 
\end{remark}

The following is the main result of this section.

\begin{theorem}
\label{thm:plan.topo.same}
Let $p\in[1,\infty)$ and $\fp X, \fp X^n \in \NFP_p$  for $n\in\mathbb{N}$.
	Then the following are equivalent.
	\begin{enumerate}[label = (\roman*)]
\item\label{it:thm.plan.topo.same.1}
     $\mathcal{AW}_p(\fp X^n,\fp X)\to 0$.
     \item\label{it:thm.plan.topo.same.2}
     $\mathcal{SCW}_p(\fp X^n, \fp X)\to 0$.
      \item\label{it:thm.plan.topo.same.2b}
     $\mathcal{CW}_p(\fp X, \fp X^n)\to 0$.
	\item\label{it:thm.plan.topo.same.3}
    $\fp X^n\to \fp X$ in the Hoover--Keisler topology and the $p$-th moments converge.\footnote{That is, $\E_{\P^{\fp X^n}}[d_\mathcal{X}^p(X^n,0)] \to\E_{\P^{\fp X}}[d_\mathcal{X}^p(X,0)]$ as $n\to\infty$.}
    \item \label{it:thm.plan.topo.same.3a} $\fp X^n\to \fp X$ in the adapted weak topology and the $p$-th moments converge.
    \item \label{it:thm.plan.topo.same.4}
    $\fp X^n\to \fp X$ in the Aldous\textsubscript{MZ} topology and the $p$-th moments converge.

	\item\label{it:thm.plan.topo.same.5}
    $\fp X^n\to \fp X$ in Hellwig's topology and the $p$-th moments converge. 

	\end{enumerate}
	If $\fp X$ has $\P^{\fp X}$-a.s.\ continuous paths, then \ref{it:thm.plan.topo.same.1}--\ref{it:thm.plan.topo.same.5} are equivalent to

	\begin{enumerate}[resume, label = (\roman*)]
	\item\label{it:thm.plan.topo.same.6} $\fp X^n \to \fp X$ in the optimal stopping topology and the $p$-th moments converge.
	\end{enumerate}
    If $\cont(\fp X)=[0,1]$, then  \ref{it:thm.plan.topo.same.1}--\ref{it:thm.plan.topo.same.5} are equivalent to
	
	\begin{enumerate}[resume, label = (\roman*)]
\item\label{it:thm.plan.topo.same.7} $\fp X^n \to \fp X$ in the Aldous\textsubscript{J1} topology and the $p$-th moments converge.
	\end{enumerate}
\end{theorem}

Recall that by Corollary \ref{cor:contPtPlain}, $\cont(\fp X)=[0,1]$ if and only if the augmented right-continuous version of the {natural} filtration of $\fp X$ is continuous.

\begin{proof}
First note that since the assertion regarding the convergence of the $p$-th moments appears in every statement, it clearly suffices to focus on the different topologies only.

   The implications \ref{it:thm.plan.topo.same.1} $\implies$ \ref{it:thm.plan.topo.same.2} and \ref{it:thm.plan.topo.same.2} $\implies$
\ref{it:thm.plan.topo.same.2b} are clear. 
To see that \ref{it:thm.plan.topo.same.2b} $\implies$ \ref{it:thm.plan.topo.same.1}, suppose that   $\mathcal{CW}_p(\fp X, \fp X^n)\to 0$. By Proposition~\ref{prop:Compact}, the sequence $(\fp X^n)_n$ is relatively compact in $\AW_p$. 
Therefore, it suffices to show that $\fp X$ is the only $\AW_p$-limit point of this sequence. Indeed, if $\fp Y \in \FP_p$ is such a limit point,
\[
\CW_p(\fp X, \fp Y) 
\le \inf_n \big( \CW_p(\fp X, \fp X^n)   + \CW_p(\fp X^n,\fp Y) \big) 
=0,
\] 
where the  inequality follows from Lemma \ref{lem:glueing} and the equality since $\CW_p\leq \AW_p$.
As $\fp X \in \cSFP_p$, Corollary~\ref{cor:plain.via.ordering} yields $\CW_p(\fp Y, \fp X)=0$. Hence, $\fp X\sim_{\HK} 
  \fp Y$ by Lemma~\ref{lem:scw.definite}. 
    Moreover, Theorem \ref{thm:AWmetrizesHK} implies that \ref{it:thm.plan.topo.same.1} $\Longleftrightarrow$ \ref{it:thm.plan.topo.same.3} and Theorem~\ref{thm:HKAF} implies that 
  \ref{it:thm.plan.topo.same.3} $\Longleftrightarrow$ \ref{it:thm.plan.topo.same.3a}.
  Thus \ref{it:thm.plan.topo.same.1}--\ref{it:thm.plan.topo.same.3a} are all equivalent.

  \vspace{0.5em}
    
    Specifically, it suffices to show that $\fp X^n\to \fp X$ in the Hoover--Keisler topology if and only if $\fp X^n\to \fp X$  Hellwig's topology, and likewise for the other statements.

   By definition of the topologies, \ref{it:thm.plan.topo.same.3}$\implies$\ref{it:thm.plan.topo.same.4}.
   Moreover, it follows from \cite[Theorem~3.10]{BePaScZh23} that \ref{it:thm.plan.topo.same.4}$\implies$\ref{it:thm.plan.topo.same.5}.
    We proceed to show that \ref{it:thm.plan.topo.same.5}$\implies$\ref{it:thm.plan.topo.same.3}.
    To that end, let $\fp X^n \to \fp X$ in Hellwig's information topology.
    Then, by definition, 
    \[\delta_{X^n} = \pp^1_1(\fp X^n) \to \pp^1_1(\fp X) = \delta_X\]
    weakly (in $\mathcal{P}(\mathcal{P}(\mathcal{X}))$) and hence  $X^n \to X$ weakly (in $\mathcal{P}(\mathcal{X})$).
    Thus, by Proposition \ref{prop:Compact}, the sequence $(\fp X^n)_n$ is precompact w.r.t.\ the Hoover--Keisler topology.
    Let $\fp Y$ be an accumulation point of this sequence in the Hoover--Keisler topology.
    We claim that $\fp X=\fp Y$.
    If that is the case, then $\fp X$ is the only accumulation point of  $(\fp X^n)_n$ and thus $\fp X^n\to \fp X$ in the Hoover--Keisler topology.
    As for the claim, it follows from the fact that  Hellwig's topology is coarser than the Hoover--Keisler topology that
    \begin{align} 
    \label{eq:in.proof.atae}
    \law(\pp_t^1(\fp X)) =\law(\pp_t^1(\fp Y)) \quad\text{for every } t \in (\cont(\fp X) \cap \cont(\fp Y)) \cup \{1\}.
    \end{align}
    Moreover, since $\fp X$ is naturally filtered,  Lemma \ref{lem:Hellwig_separates_plain} below  guarantees that \eqref{eq:in.proof.atae} implies that $\fp X = \fp Y$, as claimed.

  Thus, \ref{it:thm.plan.topo.same.1}--\ref{it:thm.plan.topo.same.5} are all equivalent.

  \vspace{0.5em}
    
    Next, let us assume that $\fp X$ has continuous paths. 
    Then it follows from Proposition \ref{prop:OST.qual} that \ref{it:thm.plan.topo.same.1}$\implies$\ref{it:thm.plan.topo.same.6}.
    Moreover, Lemma \ref{lem:OS.finer.Hellwig} below asserts that \ref{it:thm.plan.topo.same.6}$\implies$\ref{it:thm.plan.topo.same.5}.

    \medskip 
    Finally, assume that $(\mathcal{G}_t)_{t\in[0,1]}$, the right-continuous version of the {natural} filtration of $\fp X$, is continuous. 
    As the Meyer--Zheng topology is weaker than  the $J_1$-topology, it is clear that    \ref{it:thm.plan.topo.same.7}$\implies$\ref{it:thm.plan.topo.same.4} (in fact, for this direction we do not need that $\cont(\fp X) = [0,1]$). 
    We proceed to show that \ref{it:thm.plan.topo.same.5}$\implies$\ref{it:thm.plan.topo.same.7}.
    As $\cont(\fp X)=[0,1]$, \cite[Theorem~3.10]{BePaScZh23} asserts that all finite dimensional distributions of the first order prediction processes converge, i.e.\ $\pp^1_T(\fp X^n) \to \pp_T^1(\fp X)$ in distribution for every finite $T \subset [0,1]$. Corollary~\ref{cor:contPtPlain} asserts that $\pp^1(\fp X)$ has continuous paths. By  \cite[Proposition 16.13]{Al81}, these two assertions imply that $\fp X^n \to \fp X$ in Aldous\textsubscript{J1} topology. 
    \end{proof}

\begin{lemma}\label{lem:Hellwig_separates_plain}
    Let $\fp X \in \NFP $ and $\fp Y \in \FP$.
    If $\law(\pp_t^1(\fp X))=\law(\pp_t^1(\fp Y))$ for every $t\in (\cont(\fp X) \cap \cont(\fp Y)) \cup \{1\}$, then $\fp X \sim_{\HK} \fp Y$.
\end{lemma}

\begin{proof}
   By assumption we have $\law(\delta_X) = \law(\pp_1^1(\fp X)) = \law(\pp_1^1(\fp Y)) = \law(\delta_Y)$ and thus $\law(X) = \law(Y)$.
    Therefore, by Lemma \ref{lem:plaineq}, it remains to show that $\fp Y$ is naturally filtered.

    To this end fix $t\in (\cont(\fp X) \cap \cont(\fp Y))$. By Lemma \ref{lem:plaineq} and  Corollary \ref{cor:contPtPlain}, we have 
    \begin{equation}
        \label{eq:lem.plain.Hellwig.1}
        \pp^1_t(\fp X) \mbox{ is }\sigma(X_{[0,t]})\mbox{-measurable.}
    \end{equation}
    Writing $r_t \colon \X \to D([0,t];\R^d),  x \mapsto x_{[0,t]}$ for the restriction map, the adaptedness of $\fp X$ yields ${r_t}_\# \pp^1_t(\fp X) = \law(X_{[0,t]} | \F_t^\fp X) = \delta_{X_{[0,t]}}$ $\mathbb{P}^{\fp X}$-a.s.\
    Therefore, \eqref{eq:lem.plain.Hellwig.1} is equivalent to $\law( {r_t}_\# \pp^1_t(\fp X), \pp^1_t(\fp X) )$ being concentrated on the graph of a measurable function.
    Note that the latter solely depends on the law of $\pp_t^1(\fp X)$.
    Using the assumption $\law(\pp_t^1(\fp X))= \law( \pp_t^1(\fp Y))$, we conclude by the same reasoning as above that $\pp^1_t(\fp Y)$ is $\sigma(Y_{[0,t]})$-measurable.
    
    As $\cont(\fp X) \cap \cont(\fp Y)$ is dense in $[0,1]$ (cf.\  Lemma~\ref{lem:cont_pt_dense}), the above argument shows that the right-continuous filtration generated by $\pp^1(\fp Y)$ is contained in the right-continuous filtration generated by $Y$. Thus $\fp Y$ is naturally filtered by Lemma \ref{lem:plaineq}.
\end{proof}

\subsection{Proof of Theorem \ref{thm:plan.topo.same}, part \ref{it:thm.plan.topo.same.6}$\implies$\ref{it:thm.plan.topo.same.5}}

The proof needs some preparation.

    \begin{lemma}
	\label{lem:OS.sequential.finer.than.weak}
        The optimal stopping topology is a sequential topology on $\FP$. Moreover, it is finer than the weak topology, i.e.\ if $\fp X^n \to \fp X$ in the optimal stopping topology, then $\law(X^n) \to \law(X)$ weakly. 
	\end{lemma}
	
    The proof of this lemma is technical but straightforward and postponed to Appendix \ref{sec:os.topoloy}.
    The proof of the first claim follows by showing that, in the definition of the optimal stopping topology,  one can restrict to a countable set of $\varphi$'s without changing the resulting topology.
	The second claim would follow immediately if  $\varphi(x,t):=f(x) + \infty 1_{t<1}$ for $f\in C_b(\R^d)$ were an admissible choice in the definition of the optimal stopping topology (because then ${\rm OS}(\fp X, \varphi)=\E_{\P^{\fp X}}[ f(X)]$); the actual proof relies on a suitable approximation argument.

   The next observation is a standard stability result. For completeness we present a proof in  Appendix \ref{sec:os.topoloy}.
 \begin{lemma}
	\label{lem:os.continuous}
		Let $\varphi^n \colon D([0,1];\mathbb{R}^d)\times[0,1]\to \mathbb{R}$ be continuous bounded and non-anticipative functions  such that $\varphi^n$ increases pointwise to $ \varphi : D([0,1];\mathbb{R}^d)\times[0,1]\to \mathbb{R} \cup \{+\infty\}$.
		Then, for every $\fp X\in \FP$,
		\[ {\rm OS}\left(\fp X, \varphi \right)
		=\sup_{n\in\N} {\rm OS}(\fp X, \varphi^n).\]
	\end{lemma}

	\begin{lemma}
		\label{lem:martingale_strict_minimum}
		Let $(M_1,M_2)$ be a real-valued martingale satisfying  $\law(M_1)\neq\law(M_2)$. We then have
		$
		\mathbb E[M_1] > \mathbb E[M_1 \wedge M_2]
		$.
	\end{lemma}
	
	\begin{proof}
 
		We have $M_1 \ge M_1 \wedge M_2$ with equality if and only if $M_2 \le M_1$.
		Since $\law(M_1) \not = \law(M_2)$ and $(M_1,M_2)$ is a martingale, the set $\{ M_1 < M_2 \}$ has positive probability; thus
		\begin{equation*}
			0 = \mathbb E[M_2 - M_1] > \mathbb E[\mathbbm{1}_{M_1 \ge M_2}(M_2 - M_1)]
			= \mathbb E[M_1\wedge M_2 - M_1].
			\qedhere
		\end{equation*}
	\end{proof}

\begin{lemma}\label{lem:OS.finer.Hellwig}
Let  $\fp X, \fp X^n \in \NFP_p$ for $n \in \N$. If $\fp X$ has $\P^{\fp X}$-a.s.\ continuous paths and $\fp X^n \to \fp X$ in the optimal stopping topology, then $\fp X^n \to \fp X$  in Hellwig's information topology. 
\end{lemma}

\begin{proof}
Without loss of generality, throughout the proof, we work with the representatives of filtered processes in $\mathcal{CFP}_p$. 
In order to prove the claim by contradiction, suppose that the sequence $(\fp X^n)_n$ converges to $\fp X$ in the optimal stopping topology, but not in Hellwig's information topology.

Due to Lemma \ref{lem:OS.sequential.finer.than.weak}, the set of laws $\{ \law(X^n) : n \in \mathbb N\}$ is tight. By Proposition \ref{prop:Compact}, $\{ \fp X^n : n \in \mathbb N\}$ is relatively compact in the Hoover--Keisler topology. As $(\fp X^n)_n$ is assumed not to converge to $\fp X$ in the Hellwig's information topology, there is a subsequence $(\fp X^{n_k})_{k}$ converging to some $\fp Y\in \FP$ in the Hoover--Keisler topology, such that $\fp X$ and $\fp Y$ are not equal w.r.t.\ Hellwig's information topology, i.e.\ there is  $t \in \cont(\fp X)$ such that 
\begin{align}\label{eq:prf:OS.finer.Hellwig.1}
    \law(\pp_t^1(\fp X)) \not = \law(\pp_t^1(\fp Y)).
\end{align}
Note that the sequence $(\fp X^{n_k})_k$ converges in the optimal stopping topology to $\fp X$ by assumption and to $\fp Y$ because the optimal stopping topology is weaker than the Hoover--Keisler topology, cf.\ Proposition~\ref{prop:OST.qual}. By Lemma~\ref{lem:OS.sequential.finer.than.weak}, we have $\law(X)=\law(Y)$. As $\fp X$ is naturally filtered, Corollary~\ref{cor:plain.via.ordering} asserts that $\CW_p(\fp Y, \fp X)=0$. As the infimum in the definition of $\CW_p$ is attained, cf.\ Remark~\ref{rem:CWattainment}, there is a causal coupling $\pi \in \cpl_{\rm c}(\fp Y, \fp X)$ such that $X = Y$ $\pi$-almost surely.

Fix $t \in \cont(\fp X)$ such that \eqref{eq:prf:OS.finer.Hellwig.1} holds true. By \cite[Lemma~7.4]{BaBaBeEd19b} there is $f \in C_b(\mathcal{X})$ for which $M_1:=\int f \, d \pp^1_t(\fp X)$ and $M_2:=\int f \, d\pp_t^1(\fp Y)$ satisfy $\law(M_1) \not = \law(M_2)$. It follows from the causality of $\pi$  that  $(M_1, M_2)$ is a  martingale under $\pi$. Indeed, Lemma~\ref{lem:causal_equiv} implies that $\E_\pi[f(Y) | \F_t^{\fp Y} ] = \E_\pi[f(Y) | \F_{t,t}^{\fp X, \fp Y}]$ and thus we obtain
  \begin{align*}
  \E_\pi[M_2 |\F_t^{\fp X}] 
  &= \E_\pi[ \E_\pi[f(Y) | \F_t^{\fp X} ] |\F_t^{\fp X}] 
  = \E_\pi[ \E_\pi[f(Y) | \F_{t,t}^{\fp X, \fp Y} ] |\F_{t}^{\fp X}]   \\
  &=\E_\pi[f(Y) | \F_t^{\fp X} ] =\E_\pi[f(X) | \F_t^{\fp X} ]  =M_1. 
  \end{align*}
Lemma \ref{lem:martingale_strict_minimum} therefore shows that
		\begin{equation}
		    \label{eq:lem.OS.finer.Hellwig.0}
    		\mathbb E_\pi[M_1] > \mathbb E_\pi[M_1 \wedge M_2].
		\end{equation}
		In the following, we exploit \eqref{eq:lem.OS.finer.Hellwig.0} to construct a non-anticipative continuous bounded function $\varphi \colon C([0,1];\R^d)\times[0,1] \to \R$ for which ${\rm OS}(\fp X,\varphi)\neq {\rm OS}(\fp Y,\varphi)$ and thereby obtain a contradiction to the fact that $(\fp X^{n_k})_k$ converges to both $\fp X$ and $\fp Y$ in the optimal stopping topology.
		
		To that end, recall that as $t \in \cont(\fp X)$ and $\fp X$ is naturally filtered, $M_1$ is $\sigma^{\P^{\fp X}}(X_s : s \le t)$-measurable by Corollary \ref{cor:contPtPlain}.
		Hence, there is a measurable function $g \colon \mathcal{X} \to \R$ such that $\P^{\fp X}$-almost surely $g(X^t) = M_1$ (recalling the notation $\tilde f^t := ( \tilde f(s \wedge t) )_{s \in [0,1]}$ for $\tilde f \in \mathcal X$). In particular,
	\[
		\mathbb E_{\P^{\fp X}}[g(X^t)] > \mathbb E_\pi[g(X^t) \wedge M_2]
		=\mathbb E_{\P^{\fp Y}}[g(Y^t) \wedge M_2],\]
		where the equality holds because $X=Y$ almost surely under $\pi$. By a standard approximation argument, there is $h\in C_b(\mathcal{X} )$ such that
		\begin{align}
	\label{eq:opt.stop.martingle}
	 \mathbb E_{\P^{\fp X}}[h(X^t)\wedge M_1] 
		> \mathbb E_{\P^{\fp Y}}[h(Y^t) \wedge M_2].
		\end{align}

	Next, consider the non-anticipative cost function
	
        \[\varphi(x,s) :=				h(x^t) 1_{\{t\}}(s) + 				f(x) 1_{\{1\}}(s) + \infty 1_{[0,1]\setminus\{t,1\}}(s).\]
	Then, by the Snell envelope theorem,  
	\begin{align*}
	{\rm OS}(\fp X, \varphi) 
	&=  \mathbb E_{\P^{\fp X}}[h(X^t) \wedge \mathbb E_{\P^{\fp X}}[f(X) | \F_t^\fp X ]]
	=\mathbb E_{\P^{\fp X}}[h(X^t) \wedge M_1] \text{ and}\\
	{\rm OS}(\fp Y, \varphi) 
	&=\mathbb E_{\P^{\fp Y}} [h(Y^s) \wedge \mathbb E_{\P^{\fp Y}} [f(Y) | \F_t^\fp Y]]
	= \mathbb E_{\P^{\fp Y}}[h(Y^t) \wedge M_2].
	\end{align*}
	In particular, \eqref{eq:opt.stop.martingle} implies that ${\rm OS}(\fp X, \varphi) > 	{\rm OS}(\fp Y, \varphi) $.
	
	Finally, it is standard to construct  bounded continuous non-anticiaptive functions  $(\varphi_j)_j$ that approximate $\varphi$ from below.\footnote{Take e.g.\ $\varphi_n(x,s) :=h(x^s) \frac{(1 - s)}{1 - t} + f(x^s) \frac{s - t \wedge s}{1 - t} + j(1 - s)(s - t\wedge s) + j(t - t \wedge s)$.}
    Then, by Lemma \ref{lem:os.continuous}, we get
	\[{\rm OS}(\fp X,\varphi)=\sup_{j \in \N} 	{\rm OS}(\fp X, \varphi_j) 
	\quad\text{and}\quad
	{\rm OS}(\fp Y,\varphi)=\sup_{j \in \N} 	{\rm OS}(\fp Y, \varphi_j).\]
	Thus, there is $j \in \N$ such that ${\rm OS}(\fp X, \varphi_j) > 	{\rm OS}(\fp Y, \varphi_j) $.  This contradicts the fact that  $\fp X^{n_k}\to \fp X$ and $\fp X^{n_k} \to \fp Y$ in the optimal stopping topology.

\end{proof}

		\section{Convergence of random walks and Euler schemes}
  \label{sec:donsker}

    For a simpler presentation, we consider only one dimensional processes, i.e.\ $d=1$, in this section.

  \subsection{The Brownian Motion}

  Formally a filtered process $\fp X = (\Omega,\F,\P,(\F_t)_{t\in[0,1]}, X)$ is  a \emph{Brownian motion} if
    \begin{enumerate}
        \item[(a)] $X_0 = 0$,
        \item[(b)] $X_t - X_s$ is normally distributed with mean $0$ and variance $t-s$ and independent of $\F_s$, for every $0\leq s<t\leq 1$,
        \item[(c)] $t \mapsto X_t(\omega)$ is continuous for \emph{a.e.} $\omega \in \Omega$.
    \end{enumerate}
In the present setting, the following characterization of Brownian motion turns out to be technically convenient.
    To state it, denote by $\fp W$ the Wiener measure on $C([0,1];\R)$ and define the \emph{standard Brownian motion} $\mathbb{B}$ to be the standard naturally filtered process (see Definition~\ref{def:plain}) associated to the Wiener measure $\fp W$.

\begin{lemma}\label{lem:BM}
		Let $\fp X \in \FP$.
        Then the following are equivalent:
        \begin{enumerate}[label = (\roman*)]
            \item \label{it:lem.BM.1} $\fp X$ is a Brownian motion.
            \item \label{it:lem.BM.2} $\fp X$ is a martingale and $\law(X) = \mathbb W$.
            \item \label{it:lem.BM.3} $\AW_p(\fp X, \fp B) = 0$.
        \end{enumerate}
	\end{lemma}
   \begin{proof}
        Since the set of martingales is closed in $(\FP_1,\AW_1)$ by Proposition~\ref{lem:mart_closed}, we have \ref{it:lem.BM.3}$\implies$\ref{it:lem.BM.2}.
        On the other hand, the implication \ref{it:lem.BM.2}$\implies$\ref{it:lem.BM.1} is well-known as L\'evy's characterization of the Brownian motion.
        \ref{it:lem.BM.1}$\implies$\ref{it:lem.BM.3}: Since $\fp X$ is a Brownian motion, it is straightforward to see that $\pp^1(\fp X)$ is adapted to the ($\P^\fp X$-completed) filtration generated by $X$.
        Hence, by Lemma \ref{lem:plaineq}, $\fp X$ is naturally filtered and $\AW_p(\fp X, \fp B) = 0$.
    \end{proof} 

    In particular, it follows from Lemma \ref{lem:BM} that the class of Brownian motions is an equivalence class w.r.t.\ $\AW_p$.

    A fundamental property of the Brownian motion is that it emerges as the continuous-time limit of scaled random walks, which is formalized in Donsker's theorem.
    It is worthwhile to emphasize that with all the preparatory work, it is straightforward to show that Donsker's theorem extends from the setting of weak convergence to the present setting (of weak adapted convergence).

     \begin{proposition}
     \label{prop:qualitative_donsker}
      Let $(\fp X^n)_n$ be a sequence of martingales whose laws converge to $\mathbb W$ in $\W_p$.
        Then $(\fp X^n)_n$ converges to $\fp B$ in $\AW_p$.
   \end{proposition} 
   
    In particular, it follows that standard scaled random walks (see below for the precise definition) converge to Brownian motion w.r.t.\ $\mathcal{AW}_p$, i.e.\  $\AW_p(\fp B^n,\fp B) \to 0$ as $n\to \infty$. 
   
    \begin{proof}
        The set of laws $\{ \law(X^n) : n  \in \mathbb N \}$ is $\W_p$-relatively compact.
        Therefore, $\{ \fp X^n : n \in \mathbb N \}$ is relatively compact in the Hoover--Keisler topology due to \cite[Theorem 5.12]{BePaScZh23}.
        It follows from the characterization of $\AW_p$-convergence provided by Theorem \ref{thm:intro.all.topo.equal.bounded} that there is a subsequence $(\fp X^{n_j})_j$ that converges w.r.t.\ $\AW_p$ to $\fp X$, where the latter is a martingale by Proposition~\ref{lem:mart_closed}.
        Invoking Lemma \ref{lem:BM} yields that $\fp X$ has to be a Brownian motion. Thus, $\fp B$ is the only accumulation point of $(\fp X^n)_n$.
        We conclude that $(\fp X^n)_n$ converges to $\fp B$ w.r.t.\ $\AW_p$.
    \end{proof}

\subsection{Quantitative approximations}

In this section we give two quantitative approximation results in the realm of Donsker's theorem and for discretizations of SDEs.

First consider the random walk approximation of the standard Brownian motion $\fp B$. 
Suppose $(U_i)_{i }$ is a sequence of \emph{i.i.d.} random variables with the symmetric Bernoulli distribution on a probability space $(\Omega,\mathcal{F},\mathbb{P})$. 
Take the \emph{standard scaled random walk} with step size $1/n$  as \[\mathbb{B}^n=(\Omega,\mathcal{F},\mathbb{P}, (\mathcal{F}^{\mathbb{B}^n}_t)_{t \in [0,1]}, B^n),\] where $B^n$ is the scaled sum of $(U_i)_{i }$ and $\mathcal{F}^{\mathbb{B}^n}$ is the natural filtration generated by $B^n$, i.e. 
	\[ B^n_t:= \frac{1}{\sqrt{n}} \sum_{i=1}^{\lfloor nt \rfloor} U_i, \quad \mathcal{F}^{\mathbb{B}^n}_t:= \sigma^{\mathbb P}(B^n_s: \, s \leq t), \quad  t \in [0,1]. \]
Then the following holds.

\begin{proposition}
		\label{prop:RW}
        There exist constants $C, \tilde{C}>0$ such that for every $\varepsilon\in(0,\frac{1}{2})$ and every $n\geq 2$, there is an $\varepsilon$-bicausal coupling ${\pi}$ of $\mathbb{B}$ and $\mathbb{B}^n$ that satisfies
		$\E_{\pi^n}[\lVert B-B^n \rVert_{\infty} ] \leq \frac{C \log (n) }{ \sqrt{n \varepsilon}}$.
      In particular, 
        \[ \AW_1(\fp B, \fp B^n) 
        \leq \frac{ \tilde{C} \log (n) }{ \sqrt[3]{n}}.\]
	\end{proposition}

	Our next result is concerned with discrete approximation of stochastic differential equations, i.e.\ convergence of Euler schemes. 
	Suppose $\mu,\sigma\colon[0,1]\times \R \to \R$ are bounded, uniformly Lipschitz in the state variable, uniformly $\frac{1}{2}$-H\"{o}lder in the time variable and that $\inf_{t,x}\sigma(t,x)>0$.
	Further, let $\mathbb{B} \in \mathcal{FP} $ be a Brownian motion defined on $(\Omega, \mathcal{F},\mathbb{P}, (\mathcal{F}_t)_{t \in [0,1]})$, and consider the  stochastic differential equation 
	\begin{align}
		\label{eq:SDE}
		X_t=x_0 +\int_0^t \mu_u(X_u) \,du+ \int_0^t \sigma_u(X_u) \, dB_u
	\end{align}
	for $t\in[0,1]$ with initial value  $x_0\in\mathbb{R}$. 
	Under the assumptions on $\mu$ and $\sigma$, \eqref{eq:SDE} has a unique strong solution denoted by $\fp X=(\Omega,\mathcal{F},\mathbb{P}, (\mathcal{F}^{\fp X}_t)_{t \in [0,1]}, X)$ where $\mathcal{F}^{\fp X}_t:=\bigcap_{\epsilon>0} \sigma^{\mathbb{P}}(X_s:s \leq t+\epsilon)$, $t \in [0,1]$, is the natural filtration. Moreover, consider the Euler scheme  with step size $1/n$: Set $X^n_0=x_0$, and  recursively for $k=0, \dotso n-1,$
	\begin{align}
		\label{eq:SDE.Euler}
		X^n_{t_{k+1}}:= X^n_{t_k} + \mu_{t_k}(X^n_{t_k}) \frac{1}{n}+ \sigma_{t_k}(X^n_{t_k}) ( B_{t_{k+1}} -  B_{t_{k}}), 
	\end{align}
	where $t_k:=\frac{k}{n}$.
	Finally let $X^n$ be the piecewise constant interpolation of $(X_{t_k}^n)_{k=0}^n$, i.e. $X^n_t=X^n_{t_k}$ for $t \in [t_k,t_{k+1})$, and denote by $\mathbb{X}^n$ the filtered process $(\Omega,\mathcal{F},\mathbb{P}, (\mathcal{F}^{\fp X^n}_t)_{t \in [0,1]},X^n )$  where $ (\mathcal{F}^{\fp X^n}_t)_{t \in [0,1]}$ is the natural filtration generated by $X^n$. 
	
	\begin{proposition}
		\label{prop:Euler}
		For every $n\geq 2$, there is a $\frac{1}{n}$-bicausal coupling ${\pi^n}$ between $\mathbb{X}$ and $\mathbb{X}^n$  such that
		$
		\E_{\pi^n} \left[ \lVert X-X^n \rVert_{\infty} \right] 
		\leq (C-1)\sqrt{\frac{ \log (n) }{ n }}$,
		where $C>1$ is a constant that only depends on the Lipschitz constants, H\"{o}lder constants and supremum norms of $\mu$ and $\sigma$.
        In particular, 
        \[\AW_{1}(\fp X, \fp X^n) 
        \leq C\sqrt{\frac{ \log (n)  }{ n }}.\]
	\end{proposition}

	We proceed with the proof of Proposition \ref{prop:RW}.
	In a first step, we use a coupling between standard random walks and Brownian motion from \cite{KoMaTu76}:
	
	\begin{lemma}
		\label{lem:infi}

    There are constants $C_1,C_2,C_3>0$ such that the following holds.
    For every $\varepsilon \in(0,1/2)$ satisfying $1/\varepsilon \in\mathbb{N}$ and $K\in\mathbb{N}$, there exists a probability space $(\Omega, \mathcal{F}, \P)$ that supports a Brownian motion $(Y_t)_{t\in[0,\varepsilon]}$ and a scaled random walk $(X_t)_{t\in[0,\varepsilon]}$ with step size $\varepsilon/K$ such that 
		\begin{align*}
			\P\left( \max_{t \in [0,\varepsilon]} |Y_{t}-X_{t}| \geq  \sqrt{\frac{\varepsilon}{K}}\left(  C_1 \log(K) + x \right) \right) < C_2\exp(-C_3 x)
		\end{align*}
		for all $x \geq 0$.
  In particular, there exists $C_4>0$ such that
$ \E[ \sup_{t\in [0,\varepsilon]} |Y_t-X_t|^2 ]
  \leq \frac{C_4 \varepsilon \log (K)^2}{K}$.  
	\end{lemma}
	\begin{proof}
		According to \cite[Theorem 1]{KoMaTu76}, there exists a probability space $(\Omega, \mathcal{F}, \P)$ supporting a Brownian motion $W$ and a standard random walk $L$ with piecewise constant paths such that 
		\begin{align}\label{eq:komlos}
			\P\left( \max_{i=1,\dotso, K} |W_{i}-L_{i}| \geq C'_1 \log(K)+x \right) \leq C'_2 \exp(-C'_3 x)
		\end{align}
		for all $x\geq 0$, where $C'_1,C'_2,C'_3$ are some positive  constants. 
  
  Let us start by estimating the term $\P\left(\max_{t \in [0,K]} |W_t-W_{\lfloor t \rfloor}| \geq y \right)$. 
  To that end, define $Z^K_i:= \sup_{t \in [i-1,i]}|W_t - W_{i-1}|$ for $i=1,\dotso, K$.
		By the reflection principle of Brownian motion, we have that
		\begin{align}
		    \label{eq:bm.reflection}
		\P( Z^K_i \geq y )
		\leq 2 \P( |W_{1}| \geq y)
		\leq 4 \exp\left(\frac{-y^2}{2} \right)
  \end{align} 
		for every $y\geq 0$.
		In particular, by the union bound
		\begin{align}\label{eq:tail}
			\P\left(\max_{t \in [0,K]} |W_t-W_{\lfloor t \rfloor}| \geq y \right)=\P\left( \max_{i=1,\dotso,K }Z^K_i \geq y \right) 
			\leq 4K \exp\left(\frac{-y^2}{2} \right).
		\end{align}
		 For any  $y\geq \sqrt{4 \log(4K)}$, it is straightforward that  $4K \exp(\frac{-y^2}{2})\leq \exp(\frac{-y^2}{4})$.
		
		Note that since $L$ is piecewise constant,  
		\begin{align*}
			\P \left(\max_{t \in [0,K]} |W_t - L_t| \geq 2y  \right) \leq \P\left( \max_{i=1,\dotso,K} |W_{i}-L_{i}| \geq y \right)+ \P\left(\max_{t \in [0,K]} |W_t-W_{\lfloor t \rfloor}| \geq y \right). 
		\end{align*}
		Then one can deduce from \eqref{eq:komlos} and \eqref{eq:tail} that 
		\begin{align*}
			\P\left( \max_{t \in [0,K]} |W_{t}-L_{t}| \geq C_1 \log(K)+x \right)\leq C_2\exp(-C_3 x),
		\end{align*}
		for all $x\geq 0$, where $C_1,C_2,C_3$ are  positive constants. Let us take 
  \[ Y_t=\sqrt{\frac{\varepsilon}{K}}W_{tK/\varepsilon}, \quad X_t=\sqrt{\frac{\varepsilon}{K}}L_{tK/\varepsilon}, \quad t \in [0,\varepsilon], \]
  which are a  Brownian motion and a scaled random walk, respectievly, with the desired property.
  
    The `in particular' claim follows from tail integration (i.e.\ for every random variable $U$, $\E[U^2]=\int_0^\infty \P(|U|\geq \sqrt{u }) \, du$) and the previous estimate.
	\end{proof}

	Now we are ready to prove Proposition~\ref{prop:RW}. 
	The idea is to use Lemma~\ref{lem:infi} to couple $B$ and $B^n$ over small time-intervals, and paste them together in a bicausal way. 
	
	\begin{proof}[Proof of Proposition~\ref{prop:RW}]
		Let us assume without loss of generality that $1/\varepsilon$ and $K:=\varepsilon n$ are integers.
        We start by proving the first claim in the proposition.
		For each $i=1,\dotso, 1/\varepsilon$,  take an independent copy of the probability space in Lemma~\ref{lem:infi} with the given $\varepsilon$ and $K$. We denote it by $(\Omega^i, \mathcal{F}^i, \P^i)$, on
		which there is a scaled random walk $(X_t^{i})_{ t\in[0,\varepsilon]}$ with step size $1/n$ and a Brownian motion $(Y^{i}_t)_{t\in[0,\varepsilon]}$  as in Lemma~\ref{lem:infi}. 
		Consider the product space $(\Omega, \mathcal{F}, \P)$ of $(\Omega^i, \mathcal{F}^i, \P^i)_{i=1,\dotso,1/\varepsilon}$, and define the processes $(B^n_t)_{ t\in[0,1]}$ and $(B_t)_{t\in[0,1]}$ on $(\Omega, \mathcal{F},\P)$ by
  \begin{align*}
			B^n_t:= \sum_{i=1}^\ell X^{i}_{\varepsilon} + X^{\ell+1}_{t-\ell \varepsilon}
			\quad \text{and}\quad B_t:= \sum_{i=1}^\ell Y^{i}_{\varepsilon} + Y^{\ell+1}_{t-\ell \varepsilon}
		\end{align*}
		for all $t \in \left[\ell \varepsilon, (\ell+ 1)\varepsilon \right) $ where $\ell=0,\dotso, \frac{1}{\varepsilon}-1$. 
		
		According to the construction, it is clear that $(B^n_t)_{t\in[0,1]}$ is a standard scaled random walk with step size $1/n$ and $(B_t)_{t\in[0,1]}$ is a Brownian motion in their respective natural  filtrations. Then since $B^n$ is piecewise constant, 
\begin{equation}\label{eq:boundmaximum}
\E \left[\sup_{t \in [0,1]}|B^n_t-B_t| \right] \leq \E \left[\max_{1 \leq i \leq 1/\varepsilon}|B^n_{i\varepsilon}-B_{i\varepsilon} | \right]+\E\left[\max_{1 \leq i \leq 1/\varepsilon}\sup_{t \in [0,\epsilon]}|X^i_t-Y^i_t| \right]
=: \E[I] + \E[II].
\end{equation}
In the remainder of the argument, let $C_1,\dotso,C_{6}$ be suitable positive constants. 
In order to estimate the first term appearing on the right-hand side in \eqref{eq:boundmaximum}, note that the process $(B^n_{i\varepsilon}-B_{i\varepsilon})_{i=1,\dotso,1/\varepsilon}$ is a discrete-time martingale since its increments are independent with mean $0$. Hence an application of BDG inequality \cite[Chapter 3, Theorem 3.28]{KaSh91} shows that 
\[ 
    \E[I] \leq C_1 \sqrt{\E\left[|B^n_1-B_1 |^2 \right]} =C_1 \sqrt{\frac{1}{\varepsilon}\E\left[|X^1_{\varepsilon}-Y^1_{\varepsilon}|^2\right] } \leq \frac{C_2\log (K)}{\sqrt{\varepsilon n}}, 
\]
where we use Lemma~\ref{lem:infi} in the last inequality. 
		
		Let us then estimate the second term on right hand side in \eqref{eq:boundmaximum}. For every $x\geq 0$, using Lemma~\ref{lem:infi}, we get that
		\begin{align*}
            \P \left( II \geq \frac{1}{\sqrt{n}}\left(C_3 \log(K) +x\right) \right) 
			\leq  \frac{C_4 \exp(-C_5 x)}{\varepsilon}.
		\end{align*}
For $x \geq \frac{2}{C_5} \log(1/\varepsilon)$, the  term on the right-hand side is bounded from above by $ C_4\exp(-\frac{C_5x}{2}  )$. 
		Therefore, by tail integration, $	\mathbb{E} [II] \leq C_6\frac{\log(1/\varepsilon)+\log(K)}{\sqrt{n}}$.
		
In conjunction with \eqref{eq:boundmaximum}, we thus obtain
\begin{align}
\label{eq:BM.RW.approx.in.proof}\E\left[\sup_{t \in [0,1]} |B^n_t-B_t| \right] \leq C_2 \frac{\log (\varepsilon n)}{\sqrt n}+C_6\frac{\log(1/\varepsilon)+\log(K)}{\sqrt{n}} \leq (C_2+2C_6)\frac{\log(n) }{\sqrt{\varepsilon n}}. 
  \end{align}

    As for the second claim in the proposition, note that by the definition of $\AW_1$ and by \eqref{eq:BM.RW.approx.in.proof}, 
    \[\AW_1( \fp B, \fp B^n) 
    \leq \inf_{\varepsilon\in[0,1]}\left(  (C_2+2C_6)\frac{\log (n) }{\sqrt{\varepsilon n}}  +\varepsilon \right)
    \leq  2(C_2+2C_6) \frac{\log(n)}{ \sqrt[3]{n}}. \qedhere \]
	\end{proof}

	We finish this section with the proof of Proposition \ref{prop:Euler}:
	
	\begin{proof}[Proof of Proposition \ref{prop:Euler}]
    Let us start by observing that  $(X, X^n)_*\mathbb{P}$ is a $\frac{1}{n}$-bicausal coupling between $\mathbb{X}$ and $\mathbb{X}^n$.
    Indeed, since the volatility $\sigma$ of SDE is strictly positive,  we have that for every $t \in [0,1]$,
  \[\mathcal{F}^{\mathbb{X}}_t=\bigcap_{\epsilon>0} \sigma^{\mathbb{P}}(B_s :s\leq t+\epsilon) \quad \text{and} \quad \mathcal{F}^{\mathbb{X}^n}_t= \sigma^{\mathbb P}(B_{k/n} : k/n\leq t),\]
  from which it clearly follows that  $(X, X^n)_*\mathbb{P}$ is $\frac{1}{n}$-bicausal.
  
  It remains to estimate $\E[\sup_{t\in[0,1]} |X_t-X_t^n|]$.
  To that end, first note that since $X^n$ is piecewise constant,
       \[ \sup_{t\in[0,1]} |X_t-X_t^n|
       \leq \max_{0\leq k\leq n }  | X_{k/n}-X^n_{k/n} | + \max_{0\leq k< n } \sup_{t\in[k/n, (k+1)/n)}  | X_t-X_{k/n} | 
       =: I + II.\]
      The term $\E[I]$ can estimated by \cite{KlPl92}: By Theorem 10.2.2 therein, 

		 $\E[I] \leq C_1/\sqrt{n}$ where  $C_1$ is a constant depending only on  the Lipschitz constants, H\"{o}lder constants and supremum norms of $\mu$ and $\sigma$.
        Thus, it remains to estimate $\E[II]$.
        By the definition of $X$, 
        \begin{align*} 
        II  
        &\leq \max_{0\leq k< n}\sup_{t\in[k/n, (k+1)/n)} \left| \int_{k/n}^t \mu_u(X_u) \, du  \right| + \max_{1\leq k <n}\sup_{t\in[k/n, (k+1)/n)} \left| \int_{k/n}^t \sigma_u(X_u) \, dB_u \right| \\
       & =:II^\mu + II^\sigma. 
       \end{align*}
        Clearly $|\int_{k/n}^t \mu_u(X_u) \, du|\leq \|\mu\|_\infty/n$ almost surely for every $k$ and $t\in[k/n,(k+1)/n)$ and hence $\E[II^\mu]\leq \|\mu\|_\infty/n$.
        The estimate on $\E[II^\sigma]$ follows from  similar arguments as presented for the proof of Lemma \ref{lem:infi}: first, in analogy to \eqref{eq:bm.reflection}, we have that for every $k$ and every $y\geq 0$,
        \begin{align}
            \label{eq:dubins-schwarz}
        \mathbb{P}\left(\sup_{t\in[k/n,(k+1)/n)}  \left|  \int_{k/n}^t \sigma_u(X_u) \, dB_u \right| \geq  y \frac{\|\sigma\|_\infty}{\sqrt n} \right) \leq  4 \exp\left(\frac{- y^2}{2 } \right).
        \end{align}
        Indeed, if  $\sigma$ were constant, this would  immediately follow from \eqref{eq:bm.reflection}.
        For non-constant $\sigma$, it can be deduced e.g.\ from the  Dubins-Schwarz theorem \cite[Chapter 3, Theorem 4.6]{KaSh91}.
        Therefore, it follows from \eqref{eq:dubins-schwarz}, the union bound over $0\leq k < n$, and  tail-integration that  $\E[II^\sigma]\leq C_2\frac{\lVert \sigma \rVert_{\infty} \sqrt{\log(n)}}{\sqrt{n}}$ for an absolute constant $C_2$.

        The combination of all the established estimates on $I$ and $II$ completes the proof.
	\end{proof}

	\appendix

    \section{Beyond standard Borel base spaces}\label{sec:NonPolish} 
    
Throughout this section, we will denote by $\overline{\fp X}$ the canonical filtered process associated with $\fp X$, cf.\ Definition \ref{def:asociatedCFP}. The aim of this section is to prove the following proposition:

    \begin{proposition} \label{prop:nonPolish}
        For every $\fp X,\fp Y \in \mathcal{FP}_p$, the following hold.
        \begin{enumerate}[label = (\roman*)]
            \item $ \AW_p(\fp X,\fp Y)
        =\AW_p\left(\overline{\fp X},\overline{\fp Y} \right)$.
        \item $\CW_p(\fp X,\fp Y) = \CW_p(\overline{\fp X},\overline{\fp Y})$.
        \item $\SCW_p(\fp X,\fp Y) = \SCW_p(\overline{\fp X},\overline{\fp Y})$.
        \end{enumerate}
    \end{proposition}

To establish Proposition \ref{prop:nonPolish}, it will turn out to be convenient to introduce the gluing for couplings of general filtered processes.
	
\begin{lemma}
\label{lem:glueing.appendix}
    Let $\fp X^1,\fp X^2,\fp X^3, \fp X^4 \in \mathcal{FP}$, $\pi \in \cpl(\fp X^1,\fp X^2)$, $\rho \in \cpl(\fp X^2,\fp X^3)$ and $\gamma \in \cpl(\fp X^3,\fp X^4)$ such that there exist regular conditional disintegrations of $\pi$, $\rho$ and $\gamma$ w.r.t.\ their $\fp X^2$- and $\fp X^3$-coordinates.
    Then there is $\Pi\in \prob(\Omega^{\fp X^1}\times\Omega^{\fp X^2}\times\Omega^{\fp X^3}\times \Omega^{\fp X^4})$ such that 
    \[ \pi = ({\pr_{\Omega^{\fp X^1} \times \Omega^{\fp X^2}}})_\#\Pi, 
    \quad
    \rho = ({\pr_{\Omega^{\fp X^2} \times \Omega^{\fp X^3}}})_\#\Pi,
    \quad
    \gamma = ({\pr_{\Omega^{\fp X^3} \times \Omega^{\fp X^4}}})_\#\Pi.
   \]
    If, in addition, $\pi$ is $\epsilon_1$-causal from $\fp X^1$ to $\fp X^2$, $\rho$ is $\varepsilon_2$-causal from $\fp X^2$ to $\fp X^3$, and $\gamma$ is $\epsilon_3$-causal from $\fp X^3$ to $\fp X^4$, then $\sigma:=({\pr_{\Omega^{\fp X^1} \times \Omega^{\fp X^4}}})_\#\Pi$ is $(\varepsilon_1+\varepsilon_2+\epsilon_3)$-causal from $\fp X^1$ to $\fp X^4$.
\end{lemma}

\begin{proof}
    Since we assume the existence of regular conditional disintegrations of $\pi$, $\rho$ and $\gamma$ w.r.t.\ their $\fp X^2$- and $\fp X^3$-coordinates, we can define a probability measure on the measurable space $(\Omega^{\fp X^1} \times\Omega^{\fp X^2}  \times\Omega^{\fp X^3} \times\Omega^{\fp X^4} , \F^{\fp X^1} \otimes\F^{\fp X^2}  \otimes\F^{\fp X^3} \otimes\F^{\fp X^4})$ via
    \[
        \Pi(A_1 \times A_2 \times A_3 \times A_4) = \int_{A_2 \times A_3} \pi_{\omega^{\fp X^2}}(A_1) \gamma_{\omega^{\fp X^3}}(A_4) \, d\rho(\omega^{\fp X^2},\omega^{\fp X^3}),    
    \]
    for all $A_i \in \F^{\fp X^i}$, $i \in \{1,\dots,4\}$.
    The rest of the proof follows from a similar reasoning as used in the proof of Lemma \ref{lem:glueing}.
\end{proof}

\begin{lemma} \label{lem:canonical.couplings}
    Let $\fp X, \fp Y \in \mathcal{FP}$ and $\epsilon \ge 0$.
    Then we have:
    \begin{enumerate}[label = (\roman*)]
        \item \label{it:lem.canonical.couplings.1} $(\id, \pp^\infty(\fp X))_\# \P^\fp X \in \cplbc(\fp X, \overline{\fp X})$,
        \item \label{it:lem.canonical.couplings.2} If $\pi \in \cplc^\epsilon(\fp X,\fp Y)$, then $(\pp^\infty(\fp X), \pp^\infty(\fp Y))_\# \pi \in \cplc^\epsilon(\overline{\fp X}, \overline{\fp Y})$,
        \item \label{it:lem.canonical.couplings.3} If $\pi \in \cplbc^\epsilon(\fp X,\fp Y)$, then $(\pp^\infty(\fp X), \pp^\infty(\fp Y))_\# \pi \in \cplbc^\epsilon(\overline{\fp X}, \overline{\fp Y})$.
    \end{enumerate}
\end{lemma}

\begin{proof}
    We first show \ref{it:lem.canonical.couplings.1}.
   Since $\pp^\infty(\fp X)$ is $(\mathcal F^\fp X_t)_{t \in [0,1]}$-adapted, the coupling $(\id, \pp^\infty(\fp X))_\# \P^\fp X$ is causal.
    Moreover, for every $t \in [0,1]$, we have by \cite[Corollary 4.11]{BePaScZh23} that $\F_t^\fp X$ is conditionally independent of $\pp^\infty(\fp X)$ given $\pp^\infty_t(\fp X)$.
    Therefore, $(\id, \pp^\infty(\fp X))_\# \P^\fp X$ is causal from $\overline{\fp X}$ to $\fp X$, thus, bicausal.
    
    We proceed to prove \ref{it:lem.canonical.couplings.2}.
    Since $\pi$ is $\varepsilon$-causal, we have that under $\pi$,
    $
        \F_t^\fp Y \indep_{\F^\fp X_{t + \epsilon}} \F^\fp X$.
    As in the previous part, we find by \cite[Corollary 4.11]{BePaScZh23} that
    $
        \F_{t + \epsilon}^\fp X \indep_{\pp^\infty_{t + \epsilon}(\fp X)} \pp^\infty(\fp X).$
        
    Hence, invoking the chain rule of conditional independence, see \cite[Proposition 5.8]{Ka06}, yields
    $
        \F_t^\fp Y \indep_{\pp^\infty_{t + \epsilon}(\fp X)} \pp^\infty(\fp X).
    $
    We write $(\mathcal G_t^\fp X)_{t \in [0,1]}$ and $(\mathcal G_t^\fp Y)_{t \in [0,1]}$ 
    for the right-continuous filtration generated by $\pp^\infty(\fp X)$ and $\pp^\infty(\fp Y)$, respectively.
    Since 
    $\pp^\infty(\fp Y)$  
    is  
    $(\F^\fp Y_t)_t$-
    adapted, 
    $
        \mathcal G_t^\fp Y \indep_{\pp^\infty_{t + \epsilon}(\fp X)} \pp^\infty(\fp X).
    $
    As $t$ was arbitrary, we conclude that  $\G_{t}^{\fp Y} \indep_{\G_t^{\fp X}} (\G_s^{\fp X})_s$ 
    and thus $(\pp^\infty(\fp X),\pp^\infty(\fp Y))_\# \pi \in \cplc^\epsilon(\overline{\fp X}, \overline{\fp Y})$.

    Finally, \ref{it:lem.canonical.couplings.3} follows by symmetry.
\end{proof}

\begin{proposition} \label{prop:bicausal_cplapprox}
    Let $\fp X,\fp Y \in \mathcal{FP}$ and $\overline{\pi} \in \cplbc^\epsilon(\overline{\fp X}, \overline{\fp Y})$.
    Then there exists a sequence $(\pi^n)_n$ with $\pi^n \in \cplbc^{\epsilon + \frac1n}(\fp X,\fp Y)$ such that
    \[
        (X,Y)_\# \pi^n \to (\overline{X},\overline{Y})_\# {\overline{\pi}}
    \]
     w.r.t.\ the $p$-Wasserstein distance on $\X \times \X$.
\end{proposition}
\begin{proof}
    We fix a coupling $\overline{\pi} \in \cplbc^\epsilon(\overline{\fp X},\overline{\fp Y})$.
    For the canonical representatives $\overline{\fp X}$ and $\overline{\fp Y}$, there exist by Proposition \ref{prop:FiniteOmegaDense} sequences $({\fp X}^n)_{n}$ and $({\fp Y}^n)_{n }$, where $\Omega^{{\fp X}^n}$ and $\Omega^{{\fp Y}^n}$ are finite, such that
    \[
        \AW_p(\overline{\fp X}, {\fp X}^n) \le \frac1n \text{ and } \AW_p(\overline{\fp Y}, {\fp Y}^n) \le \frac1n,
    \]
    for every $n \in \N$.
    We write $\overline{\gamma}^{n,X}$ and $\overline{\gamma}^{n,Y}$ for $\frac{1}{n}$-bicausal couplings of $({\fp X}^n, \overline{\fp X})$ and $(\overline{\fp Y}, {\fp Y}^n)$, respectively, that are $\frac1n$-optimal for $\AW_p$.
    For every $n \in \mathbb N$, applying Lemma \ref{lem:glueing.appendix} with $\overline{\gamma}^{n,X}$, $\overline{\pi}$ and $\overline{\gamma}^{n,Y}$ we find a coupling $\overline{\pi}^n$ that is $(\epsilon + \frac2n)$-bicausal between ${\fp X}^n$ and ${\fp Y}^n$ and the sequence satisfies
    \begin{equation} \label{eq:prop.approx.1}
        ({X}^n,{Y}^n)_\# \overline{\pi}^n \to (\overline{X},\overline{Y})_\# \overline{\pi}.    
    \end{equation}
    By Lemma \ref{lem:canonical.couplings}, $\eta^X := (\id,\pp(\fp X))_\# \P^\fp X$ is bicausal from $\fp X$ to $\overline{\fp X}$.
    Next, by gluing $\eta^X$ and $\overline{\gamma}^{n,X}$ we obtain a $\frac1n$-bicausal coupling $\gamma^{n,X}$ between $\fp X$ and ${\fp X}^n$,  which can be derived, e.g., directly from \cite[Lemma 2.11]{Pa22}.
    Similarly, we obtain a $\frac1n$-bicausal coupling $\gamma^{n,Y}$ between $\fp Y$ and ${\fp Y}^n$.
    Note that, as ${\fp X}^n$ and ${\fp Y}^n$ have a discrete probability measure, the regular conditional disintegrations of the measures $\gamma^{n,X}$ and $\gamma^{n,Y}$ exist w.r.t.\ both coordinates.
    Therefore, we can apply Lemma \ref{lem:glueing.appendix} with $\gamma^{n,X}$, $\overline{\pi}^n$ and $\gamma^{n,Y}$ and obtain a probability measure $\Pi^n$ on the product space $\Omega^{\fp X} \times \Omega^{{\fp X}^n} \times \Omega^{{\fp Y}^n} \times \Omega^\fp Y$ such that $(\pr_{\Omega^{\fp X}},\pr_{\Omega^\fp Y})_\# \Pi^n =: \pi^n$ is $(\epsilon + \frac4n)$-bicausal between $\fp X$ and $\fp Y$.
    
    Furthermore,
    \begin{equation} \label{eq:prop.approx.2}
        \mathbb E_{\Pi^n} \left[ d^p(X,{X}^n) + d^p(Y,{Y}^n) \right]^{\frac1p} \leq
        \AW_p(\overline{\fp X}, {\fp X}^n) + \AW_p(\overline{\fp Y}, {\fp Y}^n) \le \frac4n.   
    \end{equation}
    Hence, by \eqref{eq:prop.approx.1} and \eqref{eq:prop.approx.2} we conclude that $(X,Y)_\# \pi^n \to (\overline{X},\overline{Y})_\# \overline{\pi}$.
\end{proof}

By the same line of argument, the statement of the preceding proposition remains true when replacing ``bicausal'' with ``causal''.
\begin{corollary} \label{cor:causal_cplapprox}
    Let $\fp X,\fp Y \in \mathcal{FP}_p$ and $\overline{\pi} \in \cplc^\epsilon(\overline{\fp X}, \overline{\fp Y})$ be given.
    Then there exists a sequence $(\pi^n)_n$ with $\pi^n \in \cplc^{\epsilon + \frac1n}(\fp X,\fp Y)$ such that
    \[
        (X,Y)_\# \pi^n \to (\overline{X},\overline{Y})_\# {\overline{\pi}}.
    \]    
\end{corollary}

\begin{proof}[Proof of Proposition \ref{prop:nonPolish}]
    We only show the statement for $\AW_p$, since the other assertions follow similarly, using Corollary \ref{cor:causal_cplapprox} instead of Proposition \ref{prop:bicausal_cplapprox}.
    Fix $\varepsilon \ge 0$ and let $\pi \in \cplbc^\varepsilon(\fp X, \fp Y)$.
    By Lemma \ref{lem:canonical.couplings} we have $\bar \pi := (\pp^\infty(\fp X),\pp^\infty(\fp Y))_\# \pi \in \cplbc^\varepsilon(\overline{\fp X}, \overline{\fp Y})$ and  $\law_\pi(X,Y) = \law_{\bar \pi}(\bar X, \bar Y)$.
    Hence, $\AW_p(\fp X,\fp Y) \ge \AW_p(\overline{\fp X}, \overline{\fp Y})$.
    
    On the other hand, let $\bar \pi \in \cplbc^\varepsilon(\overline{\fp X}, \overline{\fp Y})$.  By Proposition \ref{prop:bicausal_cplapprox}, there exist $(\pi^n)_n$ with  $\pi^n \in \cplbc^{\varepsilon + \frac1n}(\fp X, \fp Y)$ such that $\law_{\pi^n}(X,Y) \to \law_{\bar \pi}(\bar X, \bar Y)$.
    Since $p$-Wasserstein convergence is equivalent to weak convergence together with convergence of the $p$-moments, we have that $\law_{\pi^n}(X,Y) \to \law_{\bar \pi}(\bar X, \bar Y)$ holds also w.r.t.\ $\W_p$ (using $d((x,y),(\bar x,\bar y))^p := d_\X(x,y)^p + d_\X(\bar x,\bar y)^p$ as the metric on $\mathcal X \times \mathcal X$).
    Hence,
    \[
        \E_{\pi^n}[d(X,Y)^p] \to \E_{\bar \pi}[d(\bar X,\bar Y)^p],
    \]
    from where we immediately deduce the reverse inequality $\AW_p(\fp X, \fp Y) \le \AW_p(\overline{\fp X},\overline{\fp Y})$, which concludes the proof.
\end{proof}

 \section{Auxiliary statements}

        \subsection{Auxiliary statements for quantitative optimal stopping}
        \label{sec:appQuantOS}

\begin{proof}[Proof of Example \ref{ex:modulus}]
		We first note the following standard fact:
		For a square integrable martingale $\bbY$, by the Cauchy-Schwartz inequality and the  BDG-inequality (conditionally on $\mathcal{F}_\tau$) we have that 
		\begin{align}
			\label{eq:modulus.martingales}
			\delta_\bbY(\epsilon) \leq  2 \sup_{\tau\in\rm{ST}(\bbY)} \sqrt{ \E\left[| Y_{\tau+\varepsilon} -  Y_{\tau}|^2 \right] } .
		\end{align}
		
		Further note that $\E[| Y_{\tau+\varepsilon} -  Y_{\tau}|^2 ] =\E[Y_{\tau+\varepsilon}^2] - \E[  Y_{\tau}^2]$ by the martingale property.
		
		From \eqref{eq:modulus.martingales} the claim for the Brownian motion is straightforward:  $\E[ B_{\tau}^2]=\E[\tau]$ and $\E[ B_{\tau+\varepsilon}^2] \leq \E[\tau] + \epsilon$.
		
		The proof for the SDE case follows from similar arguments,  using \eqref{eq:modulus.martingales} for the martingale part.
		Indeed, using the triangle inequality and the It\^{o} isometry.
		\begin{align*}
			\delta_{\mathbb{S}}(\varepsilon)
			&\leq \sup_{\tau\in\mathrm{ST}(\mathbb{S})} \E\left[  \sup_{s\in[0,\epsilon]} \left| \int_{\tau}^{\tau+s} \mu_u(S_u)\,du \right| \right] 
			+\sup_{\tau\in\mathrm{ST}(\mathbb{S})} \E\left[  \left| \int_{\tau}^{\tau+s} \sigma_u^2(S_u)\,du\right| \right]^{1/2}.
		\end{align*}

		The claim for the random walk follows from similar arguments as in the case of Brownian motion. Suppose $\mathbb{B}^n$ is a scaled random walk with step size $1/n>0$. 
		Then it is clear that $ \E[(B^n_{\tau})^2 ]=  \E[\lfloor n \tau  \rfloor / n]$, and hence $\E[(B^n_{\tau+\varepsilon})^2]-\E[(B^n_{\tau})^2] \leq \varepsilon+1/n$. 
		Using \eqref{eq:modulus.martingales}, the proof follows.
		
		Denote by $M^n$ the martingale part of the Euler scheme $\mathbb{S}^n$ in \eqref{eq:SDE.Euler}. 
		Using the triangle inequality, 
		\begin{align*}
			\delta_{\mathbb{S}^n}(\varepsilon) \leq 2 \lVert \mu \rVert_{\infty} \left( \frac{1}{n} \vee \varepsilon \right) + \sup_{\tau\in\mathrm{ST}(\mathbb{S}^n)} \E\left[  \sup_{s \in [0,\varepsilon]} |M^n_{\tau+s}-M^n_{\tau}| \right],
		\end{align*}
		and denoting by  $\langle M^n \rangle$ the quadratic variation of $M^n$,
		\[\E[ (M^n_{\tau+\varepsilon })^2] -\E[ (M^n_{\tau})^2] =\E[\langle M^n \rangle_{\tau+\varepsilon}]-\E[\langle M^n \rangle_{\tau} ].\] 
		 Moreover, using \eqref{eq:modulus.martingales} and the fact that  $\langle M^n \rangle_{\tau+\epsilon} -\langle M^n \rangle_{\tau}$ is bounded from above by $2\lVert \sigma \rVert_{\infty}^2 ( 1/n \vee \varepsilon)$, 
		\begin{align*}
			\delta_{\mathbb{S}^n}(\varepsilon) \leq 2 \lVert \mu \rVert_{\infty} \left( \frac{1}{n} \vee \varepsilon \right) + 4  \lVert \sigma \rVert_{\infty} \sqrt{\frac{1}{n} \vee \varepsilon}.
		\end{align*}
		This completes the proof.
	\end{proof}
 
	\subsection{Auxiliary statements for the optimal stopping topology}
 \label{sec:os.topoloy}
	
	For shorthand notation, in this section we set $S:=C([0,1];\R^d)\times[0,1]$.

	\begin{proof}[Proof of Lemma \ref{lem:os.continuous}]
	We may assume w.l.o.g.\ that there is a uniformly distributed random variable $U$ that is $\F^{\fp X}_0$-measurable and independent of $X$ (since adding such $U$ does not change the adapted distribution).

	The advantage of this assumption is that then ${\rm ST}(\fp X)$ corresponds to the set of so-called `randomized stopping times' and by \cite[Theorem~1.5]{BaCh77}, the set 
 \[\Gamma:=\{\law(X,\tau) : \tau\in{\rm ST}(\fp X)\}\subset\mathcal{P}(S)\]
 is convex and compact w.r.t.\ the weak topology. 
	Define $
	 \Psi (n,\gamma):=  \int_{S} \varphi^n  \, d\gamma$ for $(n,\gamma) \in \mathbb{N} \times \Gamma$.
	For every fixed $n\in\N$, the mapping $\Psi(n,\cdot)$ is continuous and linear; and for every fixed $\gamma\in\Gamma$, the  mapping $\Psi(\cdot,\gamma )$ is increasing. 
	In particular, we may apply Sion's minimax theorem \cite{Si58} which guarantees that 
	\begin{align*}
	\sup_{n\in \N} {\rm OS}(\fp Y,\varphi^n)
	&=\sup_{n\in \N} \inf_{\gamma\in\Gamma} \Psi(n,\gamma)
	= \inf_{\gamma\in\Gamma} \sup_{n\in \N} \Psi(n,\gamma)
	= {\rm OS}(\fp Y,\varphi),
	\end{align*}
	where the last equality follows from the monotone convergence theorem.
	\end{proof}

	\begin{lemma}
	\label{lemma:pointwiseapprox}
		There exists a countable set $\mathcal{C}\subset C_b(S)$ such that every $f \in C_b(S)$  is the pointwise increasing limit of a sequence in $\mathcal{C}$.
	\end{lemma}
	\begin{proof}
		Let $\B$ be a countable base of the topology on $S$. 
		For $U \in \B$ and $p,q\in\Q$ with $p\geq q$, set 
		\[ f_{U,p,q}:= p1_U + q1_{U^c}.\]
		We first  show that every $f \in C_b(S)$ is the pointwise supremum of such functions.  
		To that end, let $\epsilon>0$. 
		By continuity of $f$, for every $x\in S$, there is  $B_x \in \B$ such that $f(y) \ge f(x) - \epsilon$ for all $y \in B_x$. 
		Let $q\in \Q$ with $q \le -\|f\|_\infty - 2 \epsilon$, and $p_x\in \Q$ with $f(x)-2\epsilon\leq p_x \leq f(x)-\epsilon$.
		Then 
		\[	f_{B_x,p_x,q}\leq f 
		\quad\text{and}\quad
		f_{B_x,p_x,q}(x)\geq f(x)-2\varepsilon.\]

		Next we approximate $f_{U,p,q}$ from below by continuous functions.
		By Urysohn's lemma, for all $p,q \in \Q$  with $p\geq q$ and all  $U,V \in \B$ satisfying ${\rm cl}(V) \subset U$ (where ${\rm cl}(V)$ denotes the closure of $V$), there is a continuous function $g_{U,V,p,q}$ such that 
		\[ g_{U,V,p,q}=p \text{ in } V,
		\qquad
		g_{U,V,p,q}(x)=q \text{ in } U^c
		\quad\text{and } q\leq g_{U,V,p,q} \leq p.\]
		Using the separation axiom $T_4$, it follows that 
		\[
		f_{U,p,q} = \sup_{ V \in \B  \text{ such that } {\rm cl}(V) \subset U } g_{U,V,p,q}.
		\]
		Hence, we can choose $\mathcal{C}$ as the collection of all finite maxima of functions of the form $g_{U,V,p,q}$, where $p,q \in \Q$ with $p \ge q$ and $U,V \in \B$ with ${\rm cl}(V) \subset U$.
	\end{proof}
	
	\begin{corollary}
	\label{cor:pointwiseapprox.non-anticipative}
    There exists a countable set $\mathcal{C}\subset C_b(S)$ of non-anticipative functions such that every non-anticipative $f \in C_b(S)$  is the pointwise increasing limit of a sequence in $\mathcal{C}$.
	\end{corollary}
	\begin{proof}
	Define the mapping $\iota \colon S\to S$ by $\iota( (f,t)) = (f^t,t)$ where  we recall that $f^t(s)=f(t\wedge s)$.
	In particular, $\varphi\in C_b(S)$ is non-anticipative if and only if $\varphi=\varphi\circ\iota$.
	Let $\mathcal{C}'$ be the set of Lemma \ref{lemma:pointwiseapprox}.
	It can be readily checked that  $\mathcal{C}:= \{ \varphi \circ\iota : \varphi\in\mathcal{C}'\}$ satisfies the claim in the corollary.
	\end{proof}

	\begin{proof}[Proof of Lemma \ref{lem:OS.sequential.finer.than.weak}]
	We start by showing that the optimal stopping topology is sequential.
	Let $\mathcal{C}$ be the set of Corollary \ref{cor:pointwiseapprox.non-anticipative} and 		denote by $\mathcal{T}$ the initial topology w.r.t.\ the maps $ \fp X \mapsto {\rm OS}(\fp X,\varphi) $ where $\varphi \in \mathcal C$.
    Clearly,  $\mathcal{T}$ is sequential and coarser than the optimal stopping topology.
	To complete the first part of the proof, it remains to show that  $\mathcal{T}$ is finer than the optimal stopping topology.
		
		To that end, it suffices to  prove that for every $\bbX\in \FP$, every non-anticipative $\varphi\in C_b(S;\R)$, and every $\epsilon >0$, there are $\varphi^1,\varphi^2 \in \mathcal{C}$ such that 
		\begin{align}\label{eq:OSbase}
	\begin{split}
	&\{\bbY \in \FP: |{\rm OS}(\bbX,\varphi^1)- {\rm OS}(\bbY,\varphi^1)| \leq \tfrac{1}{2}\epsilon \text{ and } |{\rm OS}(\bbX,\varphi^2)-{\rm OS}(\bbY,\varphi^2)| \leq \tfrac{1}{2}\epsilon \} \\
 &\subset \{\bbY  \in \FP : | {\rm OS}(\bbX,\varphi)- {\rm OS}(\bbY,\varphi)| \leq \epsilon \} .
\end{split}
		\end{align}
		As consequence of Lemmas \ref{lem:os.continuous} and \ref{cor:pointwiseapprox.non-anticipative}, there  are $\underline{\varphi},\overline{\varphi}\in \mathcal{C}$  such that $\underline{\varphi}\leq \varphi\leq \overline{\varphi}$ and
		\begin{equation*}
			{\rm OS}(\fp X,\overline{\varphi}) - \tfrac{1}{2}\epsilon  
			\le {\rm OS}(\fp X, \varphi) 
			\le {\rm OS}(\fp X, \underline{\varphi}) + \tfrac{1}{2}\epsilon.
		\end{equation*}
		By setting $\varphi^1=\underline{\varphi}$ and $\varphi^2=\overline{\varphi}$, \eqref{eq:OSbase} clearly follows.
		This completes the proof that the optimal stopping topology is sequential.
		
		\vspace{0.5em}
		We proceed to show that the optimal stopping topology is stronger than the weak convergence topology.
		Since the optimal stopping topology is sequential by the first part, it suffices to show that for every  sequence $(\fp X^n)_n$ converging to $\fp X$ w.r.t.\ the optimal stopping topology and every $f\in C_b ( C([0,1];\R^d))$, we have that $\E_{\P^{\fp X}}[f(X)] =  \lim_n \E_{\P^{\fp X^n}}[f(X^n)]$.
		To that end, fix such $f$ and consider, for $k \in \N$, the non-anticipative function $\varphi_k\in C_b(S)$ defined by 
		\[
		\varphi_k(x,t) := f(x^t) + k(1 - t).
		\]
		In particular, the sequence $(\varphi_k)_k$ pointwise increases to $\varphi(x,t):=f(x) + \infty 1_{[0,1)}(t)$ and thus, we have by Lemma \ref{lem:os.continuous} for every $\fp Y \in \FP$
		\begin{equation}
		\label{eq:opt.stop.approx}
		\sup_{k \in \N} {\rm  OS}(\fp Y, \varphi_k) 
		={\rm  OS}(\fp Y, \varphi) 
		= \mathbb E_{\P^{\fp Y}}[f(Y)].
		\end{equation}
		
		Since each $\varphi_k$ is admissible for the optimal stopping topology, 
		we get ${\rm OS} (\fp X, \varphi_k)=\lim_n {\rm OS}(\fp X^n, \varphi_k)$, and therefore by \eqref{eq:opt.stop.approx}
		\begin{align*}
		\mathbb E_{\P^{\fp X}}[f(X)]
		=\sup_{k \in \N} {\rm OS} (\fp X, \varphi_k) 
		&=\sup_{k\in \N } \lim_{n \to \infty} {\rm OS} (\fp X^n, \varphi_k) \\
		&\leq  \liminf_{n\to\infty} \sup_{k \in \N} {\rm OS}(\fp X^n, \varphi_k)
		=
		\liminf_{n \to \infty} \mathbb E_{\P^{\fp X^n}}[f(X^n)].
		\end{align*}
		Replacing $f$ by $-f$ in the above arguments concludes the proof.
	\end{proof}

\section{On the strict adapted Wasserstein distance}
\label{sec:aw.strong}

In this section, we elaborate on the assertion made in the introduction that topology induced by the strict adapted Wasserstein metric
\[ \AW_p^{(\rm s)}(\fp X,\fp Y):=\inf_{\pi\in\cplbc(\fp X,\fp Y)} \E_\pi[d^p_{\mathcal{X}}(X,Y)]^{1/p} \] 
has certain shortcomings for the analysis of general stochastic processes.

First, let us show that $\AW_p^{(\rm s)}$  does not generate a `weak topology'.
To that end, denote by $\textup{TC}$  the set of all strictly increasing, absolutely continuous functions $\varphi\colon [0,1] \to [0,1]$ that satisfy $\varphi(0) = 0$ and $\varphi(1) = 1$. Consequently, each $\varphi$ represents the cumulative distribution function of a probability measure on $[0,1]$, and we identify $\textup{TC}$ with a subset of the set of all probability measures on $[0,1]$. Now, let $\fp B$ be the standard Brownian motion. For every $\varphi\in \textup{TC}$, we denote by $\fp X^\varphi$ the naturally filtered process defined by $X_t^\varphi = \int_0^t \sqrt{ \dot \varphi(t) } \, dB_t$, where $\dot \varphi$ denotes the (weak) derivative of $\varphi$.
Hence, $X^\varphi$ has the same law as the time-changed Brownian motion $(B_{\varphi(t)})_{t\in[0,1]}$. Taking $\mathcal{X}=C([0,1])$ with $d_{\mathcal{X}}=\lVert \cdot \rVert_{\infty}$, the topologies generated by $\AW_p^{(\rm{s})}$ and $\AW_p$ on $\{\fp X^{\varphi}: \, \varphi \in \textup{TC} \}$ are analogous to the topologies of  total variation and weak convergence on $\textup{TC}$ respectively:

\begin{lemma}\label{lem:timechangeBM}
 Let $\varphi,\varphi^n\in \textup{TC}$ for $n\in\mathbb{N}$.
Then the following holds.
\begin{enumerate}[label = (\roman*)]
    \item  $\mathcal{AW}_p^{(s)}(\fp X^{\varphi^n},\fp X^{\varphi}) \to 0$ if and only if $\varphi^n\to \varphi$ in total variation, i.e.\ $\dot \varphi^n\to \dot \varphi$ in $L^1(dt)$.
    \item $\mathcal{AW}_p(\fp X^{\varphi^n},\fp X^{\varphi}) \to 0$ if and only if $\varphi^n\to \varphi$ weakly,  i.e.\ $\varphi^n(t)\to\varphi(t)$ for every $t\in[0,1]$.
\end{enumerate}
\end{lemma}
\begin{proof}
We start with (i). 
It follows from an application of \cite[Proposition 3.3]{BaBaBeEd19a} and the BDG inequality that there are constants $c_1,c_2>0$ depending only on $p$ such that  
\[c_1\mathcal{AW}_p^{(s)}(\fp X^{\varphi^n},\fp X^{\varphi} ) 
\leq  \left( \int_0^1 \left( \sqrt{\dot \varphi^n(t)}-\sqrt{\dot \varphi(t)} \right)^2 \, dt \right)^{1/2}
\leq c_2 \mathcal{AW}_p^{(s)}(\fp X^{\varphi^n},\fp X^{\varphi} ) .\]
Hence, $\mathcal{AW}_p^{(s)}(\fp X^{\varphi^n},\fp X^{\varphi}) \to 0$ if and only if $\varphi^n$ converges to $\varphi$ in the Hellinger distance and this is equivalent to $\varphi^n \to \varphi$ in total variation. 

As for (ii), suppose first that  $\varphi^n \to \varphi$ weakly.
Since $\varphi$ is continuous, it follows that $\lVert \varphi^n - \varphi \rVert_{\infty} \to 0$.
Therefore, for any $\epsilon>0$,  if $n$ is large enough,
\[\pi=\mathcal{L}\left( (B_{\varphi^n(t)})_{t\in[0,1]}, (B_{\varphi(t)})_{t\in[0,1]} \right)\]
is an $\epsilon$-bicausal coupling between $\fp X^{\varphi^n}$ and $\fp X^{\varphi}$.
It readily follows that $\mathcal{AW}_p(\fp X^{\varphi^n},\fp X^{\varphi}) \to 0$. 

Now suppose that $\mathcal{AW}_p(\fp X^{\varphi^n},\fp X^{\varphi}) \to 0$.

Then clearly $\mathcal{L}( (X^{\varphi^n}_t)_{t\in[0,1]})\to \mathcal{L}((X^{\varphi}_t)_{t\in[0,1]})$ weakly which implies that  $\mathcal{L}(X^{\varphi^n}_t) \to \mathcal{L}(X^{\varphi}_t)$ for  every fixed $t\in[0,1]$.
Since $X^{\varphi}_t$ and $X^{\varphi^n}_t$ are Gaussian with mean zero and variance $\varphi(t)$ and $\varphi^n(t)$ respectively,  the claim follows.
\end{proof}

Next, we show that the \emph{only} bicausal coupling between a random walk and the Brownian motion is the product coupling. Therefore, the convergence of scaled random walks to Brownian motion does not hold w.r.t.\ the strict adapted Wasserstein metric. 
As it happens, this is the outcome of a more general phenomenon.

 	\begin{proposition}\label{prop:cplbconlyindep}
	Let  $\fp X, \fp Y \in\mathcal{FP}$ such that 
    \begin{enumerate}[(a)]
        \item $\F^{\fp X}_0$ is trivial and $\F_{t-}^{\fp X}=\F_t^{\fp X}$ for all $t \in (0,1]$;
        \item there are $0=t_0 < t_1< \dotso <t_N=1$ such that $\F_t^{\fp Y}=\F_{t_n}^{\fp Y}$ for all $t \in [t_n,t_{n+1})$. 
    \end{enumerate}
    Then, for every $\pi \in \cpl_{\rm bc}(\fp X, \fp Y)$, $X$ and $Y$ are independent under $\pi$.
	\end{proposition}

 In particular, the standard Brownian motion  $\mathbb{B}$ and the standard scaled random walk $\mathbb{B}^n$ satisfy the assumptions of Proposition \ref{prop:cplbconlyindep}, and hence  $\AW^{(s)}_p(\fp B,\fp B^n) \not \to 0$ as $n \to \infty$.

\begin{remark}\label{rem:sSCW}
In analogy to the strict adapted Wasserstein distance one may also define the strict symmetrized causal Wasserstein distance, $\mathcal{SCW}^{(\rm s)}_p(\fp X, \fp Y)$, by restring to causal (not $\varepsilon$-bicausal) couplings in the original definition \eqref{eq:def.CW}.
It turns out that $\mathcal{SCW}^{(\rm s)}_p$ is topologically coarser than $\AW_p^{(\rm s)}$.
Specifically, we have $\AW^{(\rm s)}_p(\fp B^n,\fp B) \not \to 0$ while $\mathcal{SCW}^{(\rm s)}_p(\fp B^n,\fp B)  \to 0$. 
Indeed, recall that $\lim_n\AW_p(\fp B^n,\fp B)= 0$, fix $n$, and pick $\epsilon$ and a $\epsilon$-bicausal coupling $\pi$ that attain the infima in the definition of $\AW_p(\fp B^n,\fp B)$. 
Clearly $\pi$ is causal from $\fp B^n$ to the shifted process $t \mapsto B_{t -\epsilon}$  over the time interval $[\epsilon,1]$. 
Concatenating $\fp B$ with another standard Brownian motion $\fp B'$ independent of $\fp B^n,\fp B$, i.e.\ $\bar B_t:= \mathbbm{1}_{t \in [0, \epsilon]} B'_t+ \mathbbm{1}_{t \in [\epsilon ,1]}(B'_{\epsilon}+B_{t-\epsilon})$, we obtain a causal coupling from $\fp B^n$ to $\bar{\fp B}$, where the latter is still a Brownian motion. 

Under such causal coupling, the distance between $\fp B^n$ and $\bar{\fp B}$ is small for large $n$, and hence $\mathcal{CW}_p^{(\rm s)}(\fp B^n,\fp B) \to 0$. Similarly $\mathcal{CW}_p^{(\rm s)}(\fp B,\fp B^n) \to 0$, and thus $\mathcal{SCW}^{(\rm s)}(\fp B^n, \fp B) \to 0$.
\end{remark}

	\begin{proof}[Proof of Proposition \ref{prop:cplbconlyindep}]
	Throughout this proof (conditional) independence is always meant under $\pi$. 
	We shall prove via induction over $n$ that $\F_1^{\fp X} \indep \F_{t_n}^{\fp Y}$ which  clearly implies that $X \indep Y$. 
	
	$n=0$: By the causality of $\pi$ from $\fp X$ to $\fp Y$, we have that $\F_0^{\fp Y} \indep_{\F_0^{\fp X}} \F_1^{\fp X}$.
	And since $\F_0^{\fp X}$ is trivial, $\F_0^{\fp Y} = \F_{t_0}^{\fp Y}$, we conclude that  $\F_1^{\fp X} \indep  \F_{t_0}^{\fp Y} $.
	
	$n \mapsto n+1$: 
	By the causality of $\pi$ from $\fp Y$ to $\fp X$, we have that $\F_t^{\fp X} \indep_{\F^{\fp Y}_{t}} \F^{\fp Y}_{t_{n+1}}$ for all $t\in [t_n,t_{n+1})$. 
     Hence, by the left-continuity of  $(\F_t^{\fp X})_{t\in[0,1]}$ and   $\F^{\fp Y}_t=\F^{\fp Y}_{t_n}$,
\begin{equation}\label{eq:prf:bcindep1}
		\F_{t_{n+1}}^{\fp X} \indep_{\F^{\fp Y}_{t_n}} \F^{\fp Y}_{t_{n+1}}.
	\end{equation}

    Next, we claim that 
    \begin{equation}
    \label{eq:prf:bcindep2}
	\F_1^{\fp X} \indep_{\F^{\fp X}_{t_{n+1}} , \, \F_{t_n}^{\fp Y}} \F_{t_{n+1}}^{\fp Y} .
	\end{equation}
    To that end, note that the causality of $\pi$  from $\fp X$ to $\fp Y$ implies that $\E[A \, | \, \F^{\fp X}_{t_{n+1}}]=\E[A \, | \sigma( \F^{\fp X}_{t_{n+1}}, \F^{\fp Y}_{t_{n+1}} )]$ for all $A \in \F^{\fp X}_1$.
    Moreover, since
$
\F^{\fp X}_{t_{n+1}} \cup \F^{\fp Y}_{t_n} \subset \F^{\fp X}_{t_{n+1}}\cup \F^{\fp Y}_{t_{n+1}},$
it is straightforward to verify that
\[ \E[A \, | \, \F^{\fp X}_{t_{n+1}}]
=\E[A \, | \sigma( \F^{\fp X}_{t_{n+1}}, \F^{\fp Y}_{t_n} )]
=\E[A \, | \sigma( \F^{\fp X}_{t_{n+1}}, \F^{\fp Y}_{t_{n+1}}) ] \quad \text{for any $A \in \F^{\fp X}_1$},\]
     which is equicalent to \eqref{eq:prf:bcindep2}.
 
	The chain rule for conditional independence applied to \eqref{eq:prf:bcindep1} and \eqref{eq:prf:bcindep2} shows that $\F_1^{\fp X}	\indep_{ \F_{t_n}^{\fp Y}} \F^{\fp Y}_{t_{n+1}} $.
 Indeed apply \cite[Proposition 5.8]{Ka97} with $\mathcal{H}=\F^{\fp Y}_{t_{n+1}}$, $\mathcal{G}=\F^{\fp Y}_{t_{n}}$, $\mathcal{F}_1=\F^{\fp X}_{t_{n+1}}$, and $\mathcal{F}_2=\F^{\fp X}_1$.
    Combined with the inductive hypotheses  $\F_1^{\fp X} \indep \F_{t_{n}}^{\fp Y}$, another application of \cite[Proposition 5.8]{Ka97} with $\mathcal{H}=\mathcal{F}^{\fp X}_1$, $\mathcal{G}=\sigma(\{\emptyset\})$, $\mathcal{F}_1=\mathcal{F}^{\fp Y}_{t_n}$, $\mathcal{F}_2=\mathcal{F}^{\fp Y}_{t_{n+1}}$, shows that $\F_1^{\fp X} \indep \F_{t_{n+1}}^{\fp Y}$.
	This completes the proof.
\end{proof}

Finally, let us also mention that the space $(\mathrm{FP}_p,\AW^{(\rm s)}_p)$ is not separable.
\begin{example}
Let $X$ be the canonical process on $\mathcal{X}=D([0,1];\R)$, set $\F$ to be the Borel $\sigma$-algebra on $D([0,1];\R)$, let  $f(t) =  \mathbbm{1}_{\{1\}}(t) $, and take
    $\P:= \frac{1}{2}\delta_{\{f\}}+\frac{1}{2}\delta_{\{-f\}}$.
	For every  $r \in [0,1]$, consider 
 \begin{align*} 
 \F_t^r&:=\begin{cases}
 \{\emptyset,\mathcal{X}\} &\quad \text{for }t\in[0,r),\\
     \F &\quad \text{for } t\in[r,1],
     \end{cases}
     \\
     \fp X^r &:= \left( \Omega, \F, \mathbb{P}, (\F_t^r)_{t\in[0,1]}, X \right) .
	\end{align*}
	A short computation shows that for $r_1 \neq r_2$, the only bicausal coupling between $\fp X ^{r_1} $ and $\fp X ^{r_2} $ is the product coupling.
 Since  $d_{J_1}(f,-f)=2$, this implies that $\AW^{(s)}_p(\fp X^{r_1}, \fp X^{r_2}) =2^{(p-1)/p}\geq 1$. In particular, as  $\{ \fp X^r : r \in [0,1]\}$ is uncountable, $(\mathrm{FP}_p, \AW^{(\rm s)}_p)$ is not separable. 
\end{example}

\section{A counterexample}

\begin{example}\label{ex:counter}
    We consider the \cadlag{} filtered processes $(\bbX^n)_n$ and $\bbX$ that are constructed as follows: Let $U$ be a random variable that is uniformly distributed on $[0,1]$ and $V$ be an independent random variable taking the values $\pm1$ each with probability $\frac{1}{2}$. We set 
    \begin{align*}
			X_t:= V 1_{[U,1]}(t)
			\quad\text{and}\quad
			X^n_t:= \frac{V}{n} 1_{[ U, U+\frac{1}{n})}(t) + V 1_{[U+\frac{1}{n},1]}(t).
		\end{align*}
    I.e., the process $X$ jumps at time $U$ up or down according to the value of $V$, whereas $X^n$ first has a small jump at time $U$ that reveals the direction of the big jump that occurs at time $U+\frac{1}{n}$. We equip both $X^n$ and $X$ with their natural (right-continuous, completed) filtration to define the  naturally filtered processes $\fp X^n, \fp X$. 

    These processes have the following properties:
    \begin{enumerate}
        \item $\AW_p(\fp X^n, \fp X) \to 0$.  Indeed, up to a time-shift of at most $\frac{1}{n}$, the filtrations of  $\bbX$ and $\bbX^n$ coincide, and  $d_{J_1}(X^n,X)\leq \frac{1}{n}$ almost surely where $d_{J_1}$ is the Skorokhod metric.
        \item $\fp X^n \to \fp X$ in the Aldous\textsubscript{MZ} topology. This is because $\AW_p$-convergence is equivalent to ${\rm HK}_p$ convergence, which clearly implies convergence in Aldous\textsubscript{MZ} topology.
        \item $\fp X^n \not\to \fp X$ in Aldous\textsubscript{J1} topology. This can be seen by considering the map
        \[
            f((x_t,\rho_t)_{t \in [0,1]}) = \sup_{t \in [0,1]} \Big(x_t - \int y_1 \, d\rho_t \Big)^2,
        \]
        with the domain $D([0,1]; \R \times \mathcal P_1(D([0,1])))$.
        Observe that $f$ can be written as composition of continuous maps, see for example \cite[Theorem 13.4.1]{Wa02} for continuity of the supremum. In particular, the map $F(\fp Y) := \mathbb E[f((Y_t,\pp^1_t(\fp Y))_{t \in [0,1]})] = \E[\sup_{t \in [0,1]} (Y_t - \E[Y_1|\F_t^{\fp Y}])^2]$ is continuous in the Aldous\textsubscript{J1} topology.
        However, computing the values for our constructed processes yields
        \[
            F(\fp X) = 0 < 1 = \lim_{n \to \infty} \mathbb E[(\tfrac Vn - V)^2] = \lim_{n \to \infty} F(\fp X^n).
        \]
        
        \item The optimal stopping values do not converge. To see this, note that $\bbX$ is a martingale and thus $\E[ X_{\tau}] = 0$ for every $\tau\in \rm{ST}(\bbX)$. On the other hand, set $\tau^{n,x}\in\mathrm{ST}(\bbX^n)$ to be the hitting time of the process $X^n$ of the value $x\in[0,1]$. Then, for $\sigma^n=\tau^{n,1/n} \wedge \tau^{n,-1}$, we have that
		$ \E[X^n_{\sigma^n}] = \frac{1}{2}\left( \frac{1}{n}-1\right)$, which does not converge to zero.
    \end{enumerate}
\end{example}

\medskip
\noindent
{\bf Acknowledgment:}
 This research was funded in whole or in part by the Austrian Science Fund (FWF) [doi: 10.55776/P34743 and 10.55776/ESP31 and  10.55776/P35197] and the Austrian National Bank [Jubil\"aumsfond, project 18983].

\bibliographystyle{abbrv}
\bibliography{joint_biblio}

\end{document}